\documentclass[10pt, a4paper]{article}
\usepackage{amsmath, amssymb, amsthm}
\usepackage[top=2cm, bottom=2cm, left=2cm, right=2cm]{geometry}
\usepackage{xcolor}
\usepackage{enumitem}
\usepackage{bbm}
\usepackage[style=numeric, sorting=nty, sortcites=true, backend=bibtex, maxbibnames=10]{biblatex}
\AtEveryBibitem{%
	\clearfield{doi}%
} 
\renewbibmacro*{publisher+location+date}{%
	\ifentrytype{book}
	{\printlist{publisher}%
		\iflistundef{location}
		{\setunit*{\addcomma\space}}
		{\setunit*{\addcomma\space}}%
		\printlist{location}}
	{\printlist{location}%
		\iflistundef{publisher}
		{\setunit*{\addcomma\space}}
		{\setunit*{\addcomma\space}}%
		\printlist{publisher}}
	\setunit*{\addcomma\space}%
	\usebibmacro{date}%
	\newunit} 
\numberwithin{equation}{section}
\newcommand{\td}[2]{\dfrac{\mathrm{d} #1}{\mathrm{d} #2}} 
\newcommand{\ii}[4]{\int_{#1}^{#2} #3 \: \mathrm{d}#4} 
\newcommand{\aevv}{\text{ a.e. in } (0,T) } 
\newcommand{\dq}[1]{\partial_t^h #1} 
\newcommand{\hh}[1]{\left\| #1 \right\|_{H^1(\Omega)}}

\newcommand\restr[2]{{
		\left.\kern-\nulldelimiterspace 
		#1 
		\vphantom{\big|} 
		\right|_{#2} 
}} 
\newcommand{\uu}{\mathbf{u}}

\newcommand{\ene}{\mathcal{E}_{\omega,\varepsilon}}
\newcommand{\eneg}{\mathcal{E}_{\omega}}
\newcommand{\LL}{\mathbf{L}}
\DeclareMathOperator{\spn}{span}

\theoremstyle{remark}
\newtheorem{remark}{Remark}[section]
\theoremstyle{plain}
\newtheorem{proposition}{Proposition}[section]
\newtheorem{lemma}{Lemma}[section]
\newtheorem{theorem}{Theorem}[section]
\theoremstyle{definition}
\newtheorem{definition}{Definition}[section]

\allowdisplaybreaks[4]
\begin{document}
\begin{center}
		\large{\textbf{\uppercase{Well-posedness for a Navier--Stokes--Cahn--Hilliard System for Incompressible Two-phase Flows with Surfactant}}}\\
		\vspace{0.5cm}
		\small{\uppercase{Andrea Di Primio}\footnote[1]{Dipartimento di Matematica, Politecnico di Milano, Milano 20133, Italy. Email: \texttt{andrea.diprimio@polimi.it}}},\hspace{0.5cm}
		\uppercase{Maurizio Grasselli}\footnote[2]{Dipartimento di Matematica, Politecnico di Milano, Milano 20133, Italy. Email: \texttt{maurizio.grasselli@polimi.it}},\hspace{0.5cm}
		\uppercase{Hao Wu}\footnote[3]{
School of Mathematical Sciences and Shanghai Key Laboratory for Contemporary Applied Mathematics, Fudan University, Shanghai 200433, China. Key Laboratory of Mathematics for Nonlinear Sciences (Fudan University), Ministry of Education, Shanghai 200433, China.
Email: \texttt{haowufd@fudan.edu.cn}, \texttt{haowufd@yahoo.com}}
\end{center}

\begin{abstract}
\noindent
We investigate a diffuse-interface model that describes the dynamics of incompressible two-phase viscous flows with surfactant. The resulting system of partial differential equations consists of a sixth-order Cahn--Hilliard equation for \textcolor{black}{the difference of local concentrations of the binary fluid mixture} coupled with a fourth-order Cahn--Hilliard equation for the local concentration of the surfactant. The former has a smooth potential, while the latter has a singular potential. Both equations are coupled with a Navier--Stokes system for the (volume averaged) fluid velocity. The \textcolor{black}{evolution system} is endowed with suitable initial conditions, a no-slip boundary condition for the velocity field and homogeneous Neumann boundary conditions for the phase functions as well as for the chemical potentials. We first prove the existence of a global weak solution, which turns out to be unique in two dimensions. Stronger regularity assumptions on the initial data allow us to prove the existence of a unique global (resp. local) strong solution in two (resp. three) dimensions. In the two dimensional case, we can derive a continuous dependence estimate with respect to the norms controlled by the total energy. Then we establish instantaneous regularization properties of global weak solutions for $t>0$. In particular, we show that the surfactant concentration stays uniformly away from \textcolor{black}{the pure states} $0$ and $1$ after some positive time.
\end{abstract}
\medskip

  \medskip


	\renewcommand{\contentsname}{Contents}
	\tableofcontents

\section{Introduction} \label{sec:intro}

Surfactants are substances that can significantly alter the behavior of a fluid mixture, in particular, at the free interfaces between two components. They can change (reduce) the interfacial tension and allow the mixing of substances that are not able to blend under normal circumstances (e.g., water and oil). \textcolor{black}{The gradients in surface tension may also produce Marangoni flows, which are phenomenologically different from the temperature-driven ones. The rich phenomena induced by surfactants have been exploited extensively in science and they have led to a lot of applications in Engineering (see, e.g., \cite{Myers}).}



\textcolor{black}{The dynamics of a binary fluid mixture in presence of a surfactant} can be effectively modeled through the so-called diffuse-interface (or phase-field) approach \cite{AMW}. \textcolor{black}{Within} this framework, various models have been proposed \textcolor{black}{in the literature,} which account for the rich microstructures in the mixture as well as complicated morphological changes of the interfaces. The possibly first one among them, although neglecting hydrodynamical effects, dates back to the work by Laradji et al. \cite{Lar1} (see also \cite{Lar2}), where the authors investigated the dynamics of phase separation using an evolution system derived from a suitable Ginzburg--Landau free energy functional depending on two order parameters: one for \textcolor{black}{the difference in local concentrations} of the two immiscible components (denoted by $\phi$) and the other one for the local concentration of the surfactant (denoted by $\rho$). The resulting system consists of two (weakly) coupled Allen--Cahn type equations subject to thermal noises. However, in the past years, the structure of the free energy functional has been debated and refined, leading to a variety of descriptions.

In order to motivate our choice, we present a brief review of a number of models, without considering hydrodynamical effects at first. Let $\Omega\subset \mathbb{R}^d$, $d=2,3$ be a bounded domain \textcolor{black}{with smooth boundary} $\partial\Omega$. The starting point is a coarse-grained model based on a two-component Ginzburg--Landau free energy functional of possibly the simplest form:
	\[
	\mathcal{E}(\phi, \rho) = \int_{\Omega} \left(k_1|\nabla\phi|^2 + k_2|\nabla \rho|^2 + F_0(\phi) + F_1(\rho) + F_\text{int}(\phi, \rho)\right) \: \mathrm{d}x,
	\]
where $k_1, k_2 \geq 0$. In \cite{Lar1}, the parameter $k_2$ was taken to be zero as a physically reasonable approximation, since the energy cost of the fluid-surfactant attachment is small (see also \cite{Kom97}). Besides, the potential energy densities $F_0$ and $F_1$ are modeled by some double well polynomial functions of the concentrations, while the interaction energy density is given by
	\[
	F_\text{int}(\phi, \rho) = -\theta\rho|\nabla\phi|^2 + p(\phi,\rho),
	\]
where $\theta > 0$ is a given phenomenological parameter and $p$ is a bivariate polynomial. The first coupling term favors the surfactant to reside at the free interfaces between the two fluids, while $p$ is suitably chosen to penalize the presence of free surfactant in the domain. Nonetheless, as mentioned in \cite{GS94} (see also \cite{Kom97}), the energy functional proposed in \cite{Lar1} may not be well defined, since it is not bounded from below for large values of the surfactant concentration $\rho$ at the interfaces. For this reason, in \cite{Kom97} the authors proposed a slight modification of the energy $\mathcal{E}$ (with $k_2=0$) by including a regularizing term, namely,
	\[
	\mathcal{E}(\phi, \rho) = \int_{\Omega} \left(k_1|\nabla\phi|^2 + k_3|\Delta \phi|^2 + F_0(\phi) + F_1(\rho) + F_\text{int}(\phi, \rho)\right) \: \mathrm{d}x,
	\]
	where $k_3 > 0$ and $F_0, F_\text{int}$ are the same as those in \cite{Lar1}, except that $p \equiv 0$. The additional term $k_3|\Delta \phi|^2$ corresponds to the second-order term in the expansion of a free energy density in the region of nonuniform composition for a binary mixture (see e.g., \cite{CH58}). In particular, the potential $F_1$ takes the form $F_1(\rho)=\rho^2(\rho-1)^2$, where the minimum state $\rho=0$ means that the interfacial layer is occupied by the two-component mixture and there is no surfactant in the local volume, while the normalised state $\rho=1$ indicates that the interface is fully saturated with the surfactant. On the other hand, in \cite{SVDG06}, the authors did not add any higher-order regularizing term in the energy functional, but for the surfactant they chose an entropy term
	\[
	F_1(\rho) = c_1\left[ \rho \ln \rho + (1-\rho)\ln(1-\rho) \right],\quad \text{for some }c_1>0.
	\]
Besides, the potentials $F_0$ and $F_\text{int}$ were kept almost unchanged with respect to \cite{Lar1} with $p(\phi,\rho)=\frac12 W\rho\phi^2$ (for some $W>0$), which counteracts the occurrence of free surfactant and serves as an enthalpic contribution for numerical reasons. The entropy term \textcolor{black}{has the advantage that} it guarantees the order parameter $\rho$ for the surfactant will take its values in the physically relevant interval $[0,1]$. However, it has been pointed out in \cite{illposed} that when $k_2=0$, it may exist a relevant set of initial data for which the resulting problem is ill-posed. Moreover, therein the authors suggested replacing the entropy term by the Flory--Huggins type potential (see e.g., \cite{Flory42,Huggins41}). Therefore $F_1$ becomes
	\[
	F_1(\rho) = c_1\left[ \rho \ln \rho + (1-\rho)\ln(1-\rho) \right] + c_2\rho(1-\rho) + c_3|\nabla \rho|^2,
	\]
	for some $c_2 \in \mathbb{R}$ and $c_3 > 0$. We note that the choice $c_3 > 0$ is equivalent to assume $k_2>0$. This term fits \textcolor{black}{with the classical diffuse-interface description of binary mixtures \cite{CH58, Ell89, NC08}.}

The phase-field model can further include hydrodynamical effects through a suitable coupling with a system of Navier--Stokes equations for the (volume) averaged velocity $\mathbf{u}$ of the fluid mixture. To this end, one can add a term related to the kinetic energy
	\[
	\int_\Omega k_4|\mathbf{u}|^2 \: \mathrm{d}x,
	\]
for some $k_4 > 0$, into the energy functional $\mathcal{E}(\phi, \rho)$. Then the full hydrodynamical coupled system of evolutionary partial differential equations can be derived via a variational method. It consists of two convective Cahn--Hilliard type equations and the Navier--Stokes system subject to capillary forces, see, for instance, \cite{PD95} and also \cite{illposed} for the case $k_3=0$.

In light of the above considerations, throughout of this paper, we shall work with the following energy functional for the binary fluid-surfactant system
%
\begin{equation}
\mathcal{E}_{\mathrm{tot}}(\mathbf{u},\phi, \rho) = \int_{\Omega} \left(\frac{1}{2}|\mathbf{u}|^2+ \frac{\alpha}{2}|\Delta \phi|^2 +  \frac{1}{2}|\nabla\phi|^2 + S_\phi(\phi) + \frac{\beta}{2}|\nabla \rho|^2  + S_\rho(\rho) -\frac{\theta}{2}\rho|\nabla \phi|^2  \right) \mathrm{d}x,
\label{total}
\end{equation}
where $\alpha, \beta, \theta$ are positive constants. The potential function $S_\phi$ for $\phi$ is assumed to be a regular one with double-well structure, whose typical form is
\begin{equation} \label{eq:regpot}
		S_\phi(s) = \dfrac{1}{4}(s^2-1)^2, \quad s \in \mathbb{R},
\end{equation}
while the surfactant potential $S_\rho$ is assumed to be a singular one. For instance it can be the Flory--Huggins potential
\begin{equation}
S_\rho(s) = \dfrac{\theta_1}{2}[s\ln s + (1-s)\ln(1-s)] + \dfrac{\theta_2}{2}s(1-s), \quad s \in (0,1),
\label{eq:mixentropy}
\end{equation}
where $\theta_1 > 0$ and $\theta_2 \in \mathbb{R}$. In \eqref{total}, we simply set the coupling polynomial $p$ to zero, since from the mathematical point of view, those physically interesting cases considered in \cite{illposed,Lar1,SVDG06} can be easily controlled by the potential functions $S_\phi$ and $S_\rho$.

The double-well regular potential \eqref{eq:regpot} is a well-known approximation of the Flory--Huggins potential. This does not ensure that $\phi$ takes values in its physical range $[-1,1]$ \textcolor{black}{due to the loss of maximum principle for the higher-order parabolic equation}. Yet, in models as well as in numerical simulations of immiscible fluids, this approximation is easy to handle and has been widely used. Then a natural question is: why we do not assume a Flory--Huggins type potential for $\phi$ as well ? Our consideration is as follows. Observe that the evolution of $\phi$ is described by a sixth-order Cahn--Hilliard equation (see \eqref{eq:strongNSCH}). However, this kind of equation with a singular potential is rather difficult to handle. Indeed, even the existence of a weak solution in the usual sense remains an open problem (see Remark \ref{FHphi}). Only the existence of a weaker solution has been established by replacing the equation with a suitable variational inequality \cite{AM15} (see also \cite{SP13b,SW20} for the analysis of some other sixth-order Cahn--Hilliard type equations with singular potentials). On the other hand, one might think to take $\alpha=0$ and use a singular potential for $\phi$. Then the problem is to take care of the nonlinear coupling due to the term $(\theta/2) \rho|\nabla \phi|^2$, which is highly nontrivial (see Remark \ref{teta} and \cite{SP13a} for a related problem). The system with a singular potential for $\phi$ is interesting and will be the subject of a further investigation. Therefore, in this work, we confine ourselves to the case with a regular potential for $\phi$, which seems a reasonable choice in order to prove a number of theoretical results (see also \cite{SVDG06} for remarks about modeling).

All the phase-field models mentioned above have gained particular interest as far as numerical simulations are concerned. For example, the models with regular potentials have been numerically investigated in \cite{Yang17, Yang18, Xu21,Yang21a}. However, it was also noted in \cite{Xu21} that even this modification does not simplify a rigorous proof that the resulting energy functional is bounded from below. This problem is left unanswered, and the authors chose to introduce an artificial modification to the regularizing higher-order term, in order to provide a simple, yet rigorous, proof that the energy functional is bounded from below (provided that some solution exists). Instead, as we shall prove rigorously in this paper, the fact that $\rho$ takes its values within the physical range $[0,1]$ guarantees the boundedness from below of the energy functional \eqref{total}.

On account of \eqref{total}, assuming the two-phase flow to be isothermal and incompressible with matched densities, the system we want to analyze here, on some time interval $[0,T]$, $T>0$, is the following
\begin{equation} \label{eq:strongNSCH}
\begin{cases}
			\partial_t \mathbf{u} + (\mathbf{u} \cdot \nabla)\mathbf{u} - \nabla \cdot ( \nu(\phi,\rho) D\mathbf{u} ) + \nabla \pi =\mu \nabla \phi +  \psi\nabla \rho& \quad \text{in } \Omega \times (0,T),\\
			\nabla \cdot \mathbf{u} = 0& \quad \text{in } \Omega \times (0,T),\\
			\partial_t\phi + \mathbf{u} \cdot \nabla \phi = \Delta \mu & \quad \text{in } \Omega \times (0,T),\\
			\mu =  \alpha \Delta^2 \phi -\Delta \phi + S_\phi'(\phi) + \theta\nabla \cdot (\rho \nabla \phi) & \quad \text{in } \Omega \times (0,T),\\
			\partial_t\rho + \mathbf{u} \cdot \nabla \rho = \Delta \psi & \quad \text{in } \Omega \times (0,T), \\
			\psi = -\beta \Delta \rho + S'_\rho(\rho) - \dfrac{\theta}{2}|\nabla \phi|^2 & \quad \text{in } \Omega \times (0,T), \\	
\end{cases}
\end{equation}
%
System \eqref{eq:strongNSCH} is subject to the following boundary and initial conditions:
\begin{equation} \label{eq:bcs}
\begin{cases}
			\mathbf{u} = \mathbf{0} & \quad \text{on } \partial \Omega \times (0,T), \\
			\partial_\mathbf{n} \phi = \partial_\mathbf{n}\Delta \phi = \partial_\mathbf{n}\mu =0 & \quad \text{on } \partial \Omega \times (0,T), \\
			\partial_\mathbf{n} \rho = \partial_\mathbf{n} \psi = 0 & \quad \text{on } \partial \Omega \times (0,T), \\
			\mathbf{u}|_{t=0}=\mathbf{u}_0,\quad \phi|_{t=0} = \phi_0,\quad \rho|_{t=0} = \rho_0,	 & \quad \text{in } \Omega,
\end{cases}
\end{equation}
where the vector $\mathbf{n}=\mathbf{n}(x)$ denotes the unit outer normal to $\partial \Omega$. We recall that $\phi=\phi(x,t)$ stands for \textcolor{black}{the difference in local concentrations of} the two immiscible fluid components and $\rho=\rho(x,t)$ denotes the local surfactant concentration. The velocity field $\mathbf{u}=\mathbf{u}(x,t)$ is taken as the volume-averaged velocity of the binary fluid mixture, which is equivalent to the mass-averaged velocity since we only consider the case with matched densities here. The symmetric tensor $D\mathbf{u}=\frac{1}{2}(\nabla\mathbf{u}+(\nabla\mathbf{u})^ \mathrm{T})$ denotes the strain rate and the scalar function $\pi=\pi(x,t)$ stands for the (modified) pressure. The latter can be viewed as a Lagrange multiplier corresponding to the incompressibility condition $\nabla \cdot \mathbf{u} = 0$ for the fluid. The chemical potentials corresponding to $\phi$ and $\rho$ are denoted by $\mu=\mu(x,t)$ and $\psi=\psi(x,t)$, respectively, which can be obtained as variational derivatives of the free energy functional. We note that when the parameter $\theta\neq 0$, the homogeneous Neumann boundary conditions for $\mu$ (resp. $\psi$) is not equivalent to $\partial_\mathbf{n}\Delta^2 \phi=0$ (resp. $\partial_\mathbf{n} \Delta\rho=0$) on $\partial\Omega$. For the sake of simplicity, the density as well as the mobilities and other physical constants are assumed to be equal to one, but we allow the binary fluid mixture to have an unmatched kinematic viscosity $\nu=\nu(\phi, \rho)$. As we shall see, even if the potential $S_\phi$ is regular, the higher-order regularizing term in the energy functional entails the global boundedness of $\phi$ (albeit not necessarily by 1).

Our goal is to provide a first-step theoretical analysis of the initial boundary value problem \eqref{eq:strongNSCH}--\eqref{eq:bcs}. More precisely, we first prove the existence of a global weak solution in both two and three dimensions and this solution is indeed unique in dimension two \textcolor{black}{(see Theorem \ref{th:weaksolution}).} Then we establish the existence of a (unique) strong solution, which is local in time in dimension three and global in time in dimension two \textcolor{black}{(see Theorem \ref{th:wellposedlocal}).} Further results can be obtained in dimension two. First, we derive a continuous dependence estimate for strong solutions $(\mathbf{u},\phi,\rho)$ with respect to the norms \textcolor{black}{in $\mathbf{L}^2(\Omega)\times H^2(\Omega)\times H^1(\Omega)$, which are corresponding to the energy norms associated with \eqref{total} (see Theorem \ref{th:contdep}).} Next, we show that every global weak solution regularizes in finite time and the strict separation property holds for the surfactant concentration $\rho$ \textcolor{black}{(see Theorem \ref{th:regularity3}).} The latter implies that $\rho$ stays uniformly away from the pure states $0$ and $1$ for positive time (cf. \cite{MZ04,GGM}, see also \cite{GGW18,GMT18,HeWu}). \textcolor{black}{This also holds for the strong solution on $[0,+\infty)$ (see Proposition \ref{prop:sepstr}). For the proof, we shall take advantage of what has been done in \cite{GMT18} for the Navier--Stokes--Cahn--Hilliard system with singular potential (see also \cite{Ab09,HeWu}). Nevertheless, extra efforts are required to overcome those mathematical difficulties due to the complicated nonlinear coupling structure of problem \eqref{eq:strongNSCH}--\eqref{eq:bcs}. }

Theoretical analysis of fluid-surfactant type systems (yet different from the one of interest in this paper) have been conducted, starting from sharp interface models (see e.g., \cite{Garcke13,TSLV}), typically investigating only the existence of weak solutions (see \cite{Abels19} and references therein). It is also worth mentioning that other phase-field models for mixtures with surfactant have been analyzed theoretically or numerically (see, for instance, \cite{Fon07} for a stationary model and \cite{Z19,Zhu20} for hydrodynamic problems involving moving contact lines and non-constant density).

Besides the possibility of considering a Flory--Huggins potential also for $\phi$ (see above), other interesting future issues  \textcolor{black}{include, for instance, long-time behavior of global weak/strong solutions (existence of global/exponential attractors and convergence to a single equilibrium as $t\to+\infty$), rigorous mathematical analysis of extended systems with non-constant (or even degenerate) mobilities, dynamic boundary conditions (moving contact lines) as well as non-constant densities. Also, suitable optimal control problems could be formulated and analyzed.} \medskip

\textbf{Plan of the paper.} In Section \ref{sec:main results}, we first introduce some notations and the functional setting. Subsection \ref{sec:wellposed} is devoted to illustrating the weak formulation of problem \eqref{eq:strongNSCH}--\eqref{eq:bcs} and to stating the main results. Proofs of well-posedness results are given in Section \ref{sec:proof1} (existence of global weak solutions and uniqueness of weak solutions in dimension two) and Section \ref{sec:proof3} (existence and uniqueness of strong solutions). \textcolor{black}{In the final Section \ref{sec:proof4}, when $d=2$, we derive a continuous dependence estimate for strong solutions in the norms controlled by \eqref{total} and then establish the regularization property for weak solutions. In particular, we show the validity of the strict separation property for $\rho$.}

\section{Preliminaries and Main Results}
\label{sec:main results}
\subsection{Preliminaries}	
\label{sub:notation}

We first introduce the function spaces and recall some known results in functional analysis. Let $X$ be a (real) Banach space. Its dual space is denoted by $X^*$, and the duality pairing between $X$ and $X^*$ will be denoted by $\langle \cdot,\cdot\rangle_{X^*,X}$. Given an interval $I$ of $\mathbb{R}$, we introduce the function space $L^p(I;X)$ with $p\in [1,+\infty]$, which consists of Bochner measurable $p$-integrable functions with values in the Banach space $X$. The boldface letter $\mathbf{X}$ denotes the vector-valued (resp. matrix-valued) space $X^d$ (resp. $X^{d\times d}$) endowed with the corresponding norms.

For the standard Lebesgue and Sobolev spaces, we use the notations $L^{p} := L^{p}(\Omega)$ and $W^{k,p} := W^{k,p}(\Omega)$ for any $p \in [1,+\infty]$ and $k > 0$, equipped with the norms $\|\cdot\|_{L^{p}(\Omega)}$ and $\|\cdot\|_{W^{k,p}(\Omega)}$.  When $p = 2$, we denote $H^{k}(\Omega) := W^{k,2}(\Omega)$ and the norm $\|\cdot\|_{H^{k}(\Omega)}$. The norm and inner product on $L^{2}(\Omega)$ are simply denoted by $\|\cdot\|$ and $(\cdot,\cdot)$, respectively.
The spaces $H^2_N(\Omega)$ and $H^4_N(\Omega)$ consisting of functions subject to homogeneous Neumann boundary conditions are defined as
\begin{align*}
&H^2_N(\Omega)=\big\{u\in H^2(\Omega): \partial_\mathbf{n} u=0\ \text{a.e. on}\ \partial\Omega\big\},\\
&H^4_N(\Omega)=\big\{u\in H^4(\Omega): \partial_\mathbf{n} u=\partial_\mathbf{n}\Delta u=0\ \text{a.e. on}\ \partial\Omega\big\}.
\end{align*}
For every $f\in (H^1(\Omega))^*$, we denote by $\overline{f}$ its generalized mean value over $\Omega$ such that $\overline{f}=|\Omega|^{-1}\langle f,1\rangle_{(H^1)^*,\,H^1}$. If $f\in L^1(\Omega)$, then its mean value is simply given by $\overline{f}=|\Omega|^{-1}\int_\Omega f \,\mathrm{d}x$. In the subsequent analysis, we will use the well-known Poincar\'{e}--Wirtinger inequality
\begin{equation}
\label{poincare}
\|f-\overline{f}\|\leq C_P\|\nabla f\|,\quad \forall\,
f\in H^1(\Omega),
\end{equation}
where $C_P$ is a constant depending only on $\Omega$.
We introduce the space $L^2_{0}(\Omega):=\{f\in L^2(\Omega):\overline{f} =0\}$ and
$$
V_0=H^1(\Omega)\cap L^2_0(\Omega)=\{ u \in H^1(\Omega):\ \overline{u}=0\}, \quad
V_0^*= \{ u \in (H^1(\Omega))^*:\ \overline{u}=0 \}.
$$
Then we see that $f\mapsto (\|\nabla f\|^2+|\overline{f}|^2)^\frac12$ is an equivalent norm on $H^1(\Omega)$ while $f\mapsto \|\nabla f\|$ is an equivalent norm on $V_0$. Besides, we recall the following elliptic estimates.
\begin{lemma}\label{th:ellreg}
Let $\Omega$ be a bounded domain with a $\mathcal{C}^4$-boundary. The following estimates hold:
		\begin{align*}
		&\| u\|_{H^2(\Omega)} \leq C\|\Delta u\|, \qquad \ \forall\, u\in H^2_N(\Omega)\cap L^2_0(\Omega),\\
        &\| u\|_{H^3(\Omega)} \leq C\|\nabla \Delta u\|, \quad\ \ \forall\, u\in H^3(\Omega)\cap H^2_N(\Omega)\cap L^2_0(\Omega),\\
        &\| u\|_{H^4(\Omega)} \leq C\|\Delta^2 u\|,\qquad \, \forall\, u\in H^4_N(\Omega)\cap L^2_0(\Omega).
		\end{align*}
In all cases, the constant $C>0$ only depends on $\Omega$, $d$, but is independent of $u$.
\end{lemma}

Consider now the realization of $-\Delta$ with homogeneous Neumann boundary condition, that is, the linear operator
$A_N\in \mathcal{L}(H^1(\Omega),H^1(\Omega)^*)$ defined by
\begin{equation}\nonumber
   \langle A_N u,v\rangle_{(H^1)^*,H^1} := \int_\Omega \nabla u\cdot \nabla v \, \mathrm{d}x,\quad \text{for }\,u,v\in H^1(\Omega).
\end{equation}
Then the restriction of $A_N$ from the linear space $V_0$ onto $V_0^*$ is an isomorphism. In particular, $A_N$ is positively defined on $V_0$ and self-adjoint. We denote its inverse map by $\mathcal{N} =A_N^{-1}: V_0^* \to V_0$. Note that for every $f\in V_0^*$, $u= \mathcal{N} f \in V_0$ is the unique weak solution of the Neumann problem
$$
\begin{cases}
-\Delta u=f, \quad \text{in} \ \Omega,\\
\partial_{\mathbf{n}} u=0, \quad \ \  \text{on}\ \partial \Omega.
\end{cases}
$$
It is straightforward to verify that
\begin{align*}
&\langle A_N u, \mathcal{N} g\rangle_{V_0^*,V_0} =\langle  g,u\rangle_{(H^1)^*,H^1}, \quad \forall\, u\in H^1(\Omega), \ \forall\, g\in V_0^*,\\
&\langle  g, \mathcal{N} f\rangle_{V_0^*,V_0}
=\langle f, \mathcal{N} g\rangle_{V_0^*,V_0} = \int_{\Omega} \nabla(\mathcal{N} g)
\cdot \nabla (\mathcal{N} f) \, \mathrm{d}x, \quad \forall \, g,f \in V_0^*.
\end{align*}
Also, we have the chain rule
\begin{align}
&\langle \partial_t u(t), \mathcal{N} u(t)\rangle_{V_0^*,V_0}=\frac{1}{2}\frac{\mathrm{d}}{\mathrm{d}t}\|\nabla \mathcal{N} u\|^2,\ \ \textrm{a.e. in}\ (0,T),\nonumber
\end{align}
for any $u\in H^1(0,T; V_0^*)$. For any $f\in V_0^*$, we set $\|f\|_{V_0^*}=\|\nabla \mathcal{N} f\|$.
It is well-known that $f \mapsto \|f\|_{V_0^*}$ and $
f \mapsto(\|f-\overline{f}\|_{V_0^*}^2+|\overline{f}|^2)^\frac12$ are equivalent norms on $V_0^*$ and $(H^1(\Omega))^*$, respectively.

Concerning the Navier--Stokes equations, we introduce the spaces (see, for instance, \cite{Galdi11})
$$
\mathbf{H}_\sigma := \overline{\{ \mathbf{u} \in  [C^\infty_0(\Omega)]^d: \nabla \cdot \mathbf{u} = 0 \text{ in } \Omega\}}^{[L^2(\Omega)]^d}, \quad
\mathbf{V}_\sigma := \overline{\{ \mathbf{u} \in [C^\infty_0(\Omega)]^d : \nabla \cdot \mathbf{u} = 0 \text{ in } \Omega\}}^{[H^1(\Omega)]^d}.
$$
endowing the former with the $\mathbf{L}^2(\Omega)$-Hilbert structure, whereas for the latter we set
	\[
	(\mathbf{u}, \mathbf{v})_{\mathbf{V}_\sigma}:=(\nabla \mathbf{u}, \nabla \mathbf{v}), \qquad \| \mathbf{u} \|_{\mathbf{V}_\sigma}:= (\nabla \mathbf{u}, \nabla \mathbf{u})^\frac12.
	\]
\textcolor{black} {The latter is a norm equivalent to the canonical one because of Korn's inequality and}
$$
\|\nabla \uu\|\leq \sqrt{2}\|D\uu\|\leq \sqrt{2}\|\nabla\uu\|,\quad \forall\, \uu\in\mathbf{V}_\sigma.
$$
Next, we consider the Stokes operator $\mathbf{A}:  \mathbf{V}_\sigma \to \mathbf{V}_\sigma^*$, which is the Riesz isomorphism between $\mathbf{V}_\sigma$ and its topological dual, that is,	
    \[
\langle \mathbf{A}\uu, \mathbf{v}\rangle_{\mathbf{V}_\sigma^*,\mathbf{V}_\sigma } = \ii{\Omega}{}{\nabla \uu : \nabla \mathbf{v}}{x}.
    \]
Here, we have adopted the notation $M_1 : M_2 = \mathrm{trace}(M_1M_2^\mathrm{T})$ for two arbitrary $d\times d$ matrices $M_1, M_2$. The inverse of $\mathbf{A}$ is denoted by $\mathbf{A}^{-1}$. In a similar fashion to what has been carried out for the operator $A_N$, we can define the equivalent norm
$\| \mathbf{u}\|_{\mathbf{V}_\sigma^*} := \| \nabla \mathbf{A}^{-1}\mathbf u \|$ in $\mathbf{V}_\sigma^*$. Besides, the following chain rule holds
\begin{equation*}
\langle\partial_t\mathbf{f}(t),\mathbf{A}^{-1}\mathbf{f}(t)\rangle_{\mathbf{V}_\sigma^*,\mathbf{V}_\sigma } =\frac{1}{2}\frac{\mathrm{d}}{\mathrm{d}t}\|\nabla\mathbf{A}^{-1}\mathbf{f}\|^2,\quad   \textrm{a.e.}\ t \in (0,T),\nonumber
\end{equation*}
for any $\mathbf{f}\in H^1(0,T;\mathbf{V}_\sigma^*)$. Next, we define the space $\mathbf{W}_\sigma := \mathbf{H}^2(\Omega) \cap \mathbf{V}_\sigma$ and recall the following regularity result for Stokes operator (see e.g.,  \cite[Appendix B]{GMT18}):
\begin{lemma}\label{stokes}
Let $d=2,\,3$. For any $\mathbf{f} \in \mathbf{H}_{\sigma}$,
there exists a unique $\mathbf{u}\in \mathbf{W}_{\sigma}$ and $q\in H^1(\Omega)\cap L_0^2(\Omega)$ such that $-\Delta \mathbf{u}+\nabla q=\mathbf{f}$ a.e. in $\Omega$, that is, $\mathbf{u}=\mathbf{A}^{-1}\mathbf{f}$. Moreover, we have
\begin{align*}
&\|\mathbf{u}\|_{\mathbf{H}^2}+\|\nabla q\|\le C\|\mathbf{f}\|,
\\
& \|q\|\le C \|\mathbf{f}\|^\frac12\|\nabla \mathbf{A}^{-1}\mathbf{f}\|^\frac12,
 \end{align*}
where $C$ is a positive constant that may depend on $\Omega$, $d$, but is independent of $\mathbf{f}$.
\end{lemma}	
\noindent Then it follows that the norm $\|\uu\|_{\textbf{W}_\sigma} := \|\mathbf{A}\uu\|$ is equivalent to the standard $\mathbf{H}^2$-norm in $\mathbf{W}_\sigma$.

For the sake of convenience, below we
%
%
 report the Ladyzhenskaya and Agmon inequalities (see e.g., \cite{RTDDS})
\begin{align}
		&\|f\|_{L^4(\Omega)} \leq C\|f\|^{1-\frac{d}{4}}\|f\|_{H^1(\Omega)}^\frac{d}{4}, \qquad  \forall\, f \in H^1(\Omega), \text{ if } d = 2,3, \label{eq:gn1}\\
		&\|f\|_{L^\infty(\Omega)} \leq C\|f\|^\frac{1}{2}\|f\|^\frac{1}{2}_{H^2(\Omega)}, \qquad \forall \, f \in H^2(\Omega), \text{ if } d = 2, \label{eq:agmon1}\\
		&\|f\|_{L^\infty(\Omega)} \leq C\|f\|^\frac{1}{2}_V\|f\|^\frac{1}{2}_{H^2(\Omega)}, \qquad  \forall \, f \in H^2(\Omega), \text{ if } d = 3,
\label{eq:agmon2}
	\end{align}
and the Gagliardo--Nirenberg inequality
\begin{align} \label{eq:gagnir}
\|D^jf\|_{L^p(\Omega)}\leq C
\|f\|_{L^q(\Omega)}^{1-a} \| f\|_{W^{m,r}(\Omega)}^a ,
\quad \forall \, f\in W^{m,r}(\Omega)\cap L^q(\Omega),
\end{align}
where $D^j f$ denotes the $j$-th weak partial derivatives of $f$, $j, m$ are arbitrary integers satisfying $0\leq j< m$ and $\frac{j}{m}\leq a\leq 1$, and
$1\leq q, r\leq +\infty$ such that
   \[
\frac{1}{p}-\frac{j}{d}=a\left(\frac{1}{r}-\frac{m}{d}\right)+\frac{1-a}{q}.
   \]
If $1<r<+ \infty$ and $m-j-\frac{n}{r}$ is a nonnegative integer, then the
above inequality holds only for $\frac{j}{m}\leq a< 1$. The above inequalities will be frequently used in the subsequent analysis.	

In the remaining part of this paper, the letters $C$, $C_i$ will denote genetic positive constants possibly depending on the domain $\Omega$, the coefficients of the system as well as on the boundary and initial data at most. These constants may vary in the same line in the subsequent estimates and their special dependence will be pointed out explicitly in the text, if necessary.

\subsection{Main results}	
\label{sec:wellposed}
We first state the following assumptions that will be needed in our analysis.
\begin{enumerate}[label=\textbf{(H\arabic*)}]
\item  $\nu \in \mathcal{C}^2(\mathbb{R}^2)$ and there exist two positive constants $\nu_*$ and $\nu^*$ such that
	\[ 0 < \nu_* \leq \nu(s_1,s_2) \leq \nu^*,\quad \forall\, (s_1,s_2) \in \mathbb{R}^2. \]
\item $S_\phi\in \mathcal{C}^2(\mathbb{R})$ satisfies
\begin{align*}
& S_\phi''(s)\geq -c_0,\quad c_0\geq 0, \quad \forall\, s\in \mathbb{R},\\
& S_\phi'(s)s\geq c_1 S_\phi(s)-c_2,\quad S_\phi(s)\geq c_3 s^4-c_4,\quad \text{for some}\ c_1, c_3>0,\ c_2,c_4\geq 0,\quad \forall\, s\in\mathbb{R},\\
& |S_\phi'(s)|\leq \varepsilon S_\phi(s)+c_\varepsilon,\quad\ \forall\, \varepsilon>0, s\in\mathbb{R},
\end{align*}
where $c_\varepsilon>0$ depends on $\varepsilon$.
\item $S_\rho$ can be written as follows
$$
S_\rho(s)= \widehat{S}_\rho(s)+ R_\rho(s),
$$
where $\widehat{S}_\rho: [0,1]\to \mathbb{R}$ satisfies $\widehat{S}_\rho\in \mathcal{C}^0([0,1])\cap \mathcal{C}^2((0,1))$. We make the extension by (right or left) continuity at the endpoints $0,1$ and then over the whole real line with $\widehat{S}_\rho(s) = +\infty$ whenever $s \notin [0,1]$. Moreover, it holds
\begin{align*}
\lim_{s \to 0^+} \widehat{S}_\rho'(s) = -\infty, \qquad \lim_{s \to 1^-} \widehat{S}_\rho'(s) = +\infty, \qquad
\lim_{s \to 0^+} \widehat{S}_\rho''(s) = +\infty, \qquad \lim_{s \to 1^-} \widehat{S}_\rho''(s) = +\infty,
\end{align*}
and there exists a small $\epsilon_1 \in (0,1)$ such that $\widehat{S}''_\rho$ is nondecreasing in $[1-\epsilon_1, 1)$ and nonincreasing in $(0, \epsilon_1]$.
Moreover, we suppose $R_\rho\in \mathcal{C}^2(\mathbb{R})$ is such that
$$
|R''_\rho(s)|\leq L_1,\qquad \forall\, s\in\mathbb{R},
$$
with $L_1 > 0$ being a certain given constant.
\item  $\alpha, \beta$ and $\theta$ are given positive constants.
\end{enumerate}

\begin{remark}
It is easy to see that the fourth-order polynomial \eqref{eq:regpot} fulfills \textbf{(H2)}. Besides, for the physically relevant potential \eqref{eq:mixentropy}, we can simply take
	\begin{equation*} \label{eq:singularpart}
		\widehat{S}_\rho(s) = \dfrac{\theta_1}{2}\left[s\ln s + (1-s)\ln(1-s)\right], \qquad s \in (0,1),
	\end{equation*}
and $R_\rho(s)=\dfrac{\theta_2}{2}s(1-s)$ so that \textbf{(H3)} is satisfied.
\end{remark}
\begin{remark}
\label{teta}
We do not impose any restriction on the size of the parameters $\alpha, \beta$ and $\theta$ in \textbf{(H4)}. Thus, the higher-order term $\alpha \Delta^2\phi$ is necessary to guarantee the well-posedness of the system. Without this regularization, the term $-\Delta \phi+\theta\nabla\cdot(\rho \nabla \phi)= (-1+\theta \rho)\Delta\phi +\theta \nabla \rho\cdot\nabla\phi$ may lead to certain backward diffusion \textcolor{black}{when the coupling parameter $\theta>0$ is} such that $-1+\theta\rho>0$ (see \cite{Ell89,NC08} for the classical Cahn--Hilliard equation).
\end{remark}

Let us introduce the notion of finite energy weak solution to the initial boundary value problem \eqref{eq:strongNSCH}--\eqref{eq:bcs}.
\begin{definition}
\label{def:solution}
Assume that $\Omega\subset\mathbb{R}^d$, $d = 2, 3$ is a smooth bounded domain and $p$ denotes an exponent such that $p > 2$ if $d = 2$ and $2 < p \leq 6$ if $d = 3$. Let $\uu_0 \in \mathbf{H}_\sigma$, $\phi_0 \in H^2_N(\Omega)$, $\rho_0 \in H^1(\Omega)$. Suppose \textcolor{black}{that} $S_\rho(\rho_0) \in L^1(\Omega)$ and $\overline{\rho_0} \in (0,1)$.  Given $T>0$, a quintuplet $( \uu, \phi, \rho, \mu, \psi)$ is called a weak (or finite energy) solution to problem \eqref{eq:strongNSCH}--\eqref{eq:bcs} on $[0,T]$, if
\begin{enumerate}[label=(\roman*)]
\item $\uu \in L^{\infty}(0,T; \mathbf{H}_\sigma) \cap L^2(0,T;\mathbf{V}_\sigma)\cap W^{1,\frac{4}{d}}(0,T; \mathbf{V}^*_\sigma)$;
\item $\phi \in L^{\infty}(0,T; H^2_N(\Omega)) \cap L^{2}(0,T; H^5(\Omega)\cap H^4_N(\Omega))\cap H^1(0,T; (H^1(\Omega))^*)$;
\item $\rho \in L^{\infty}(0,T; H^1(\Omega)) \cap L^{4}(0,T;H^2_N(\Omega)) \cap L^2(0,T; W^{2,p}(\Omega))\cap H^1(0,T; (H^1(\Omega))^*)$;
\item $\mu, \psi \in L^2(0,T; H^1(\Omega))$;
\item $\rho\in L^\infty(\Omega\times(0,T))$\ \text{and}\ $0 < \rho(x, t) < 1$ for a.a. $(x,t) \in \Omega \times (0,T)$;
\item $(\uu, \phi, \rho)$ satisfies the weak formulation
\begin{equation} \label{eq:weakNSCH}
\begin{cases}
\left\langle \partial_t \mathbf{u}, \mathbf{v} \right\rangle_{\mathbf{V}^*_\sigma, \mathbf{V}_\sigma} + ((\mathbf{u} \cdot \nabla)\mathbf{u},\mathbf{v}) + (\nu(\phi,\rho) D\mathbf{u}, D\mathbf{v}) \\
\qquad\qquad = (\mu\nabla \phi, \mathbf{v}) + (\psi\nabla \rho, \mathbf{v}) & \quad \forall \: \mathbf{v} \in \mathbf{V}_\sigma, \aevv, \\
\left\langle \partial_t \phi, v \right\rangle_{(H^1(\Omega))^*,H^1(\Omega)} + (\mathbf{u} \cdot \nabla \phi, v) + (\nabla \mu, \nabla v) = 0& \quad \forall \: v \in H^1(\Omega), \aevv,\\
\mu =  \alpha \Delta^2 \phi-\Delta \phi  + S_\phi'(\phi) + \theta\nabla \cdot (\rho \nabla \phi) & \quad \text{a.e. in } \Omega \times (0,T), \\
\left\langle \partial_t \rho, v \right\rangle_{(H^1(\Omega))^*,H^1(\Omega)} + (\mathbf{u} \cdot \nabla \rho, v) + (\nabla \psi, \nabla v) = 0& \quad \forall \: v \in H^1(\Omega), \aevv,\\
\psi = -\beta \Delta \rho + S'_\rho(\rho) - \dfrac{\theta}{2}|\nabla \phi|^2 & \quad \text{a.e. in } \Omega \times (0,T);
\end{cases}
\end{equation}
\item the initial conditions $\uu|_{t=0} = \uu_0$, $\phi|_{t=0}=\phi_0$, $\rho|_{t=0}=\rho_0$ in $\Omega$ are fulfilled.
\end{enumerate}
\end{definition}
\begin{remark}
\label{rem:reg1}
\textcolor{black}{The properties $\phi\in L^2(0,T;H^5(\Omega))$ and $\partial_t\phi \in L^2(0,T;(H^1(\Omega))^*)$ entail that $\phi  \in \mathcal{C}^0([0,T];H^2(\Omega))$. Similarly, we have $\rho  \in \mathcal{C}^0([0,T];L^2(\Omega))$ and $\rho \in \mathcal{C}^0_{\text{w}}([0,T];H^1(\Omega))$, where the subscript ``w" stands for weak continuity (in time).}
In addition, in light of the regularity of $\uu$ and its time derivative, it follows that $\uu \in \mathcal{C}^0([0,T]; \mathbf{H}_\sigma)$ when $d = 2$ and $\uu \in \mathcal{C}^0_{\text{w}}([0,T]; \mathbf{H}_\sigma)$ when $d = 3$. A pressure $\textcolor{black}{\pi \in W^{-1,\infty}(0,T; L^2(\Omega))}$ can also be recovered, up to a \textcolor{black}{constant}, through the classical De Rham theorem (see, for instance, \cite[Section V.1.5]{BFNS}, see also \cite{RTNS}).
\end{remark}
\begin{remark}
\label{rem:reg3}
 On account of the global boundedness of $\rho$ and the $L^\infty(0,T;H^2(\Omega))$-regularity of $\phi$, it holds that the weak solutions satisfy $\phi, \rho \in L^\infty(0,T; L^p(\Omega))$ for every $p \geq 1$. In particular, the mapping $t \mapsto \| \rho(t) \|_{L^\infty}$ is measurable and essentially bounded (see \cite[Remark 3.3]{GGG}).
\end{remark}

Now we are in a position to state the main results of this paper. The first result concerns the existence of a global weak solution.
\begin{theorem}[Global weak solutions]
\label{th:weaksolution}
Let \textbf{(H1)}--\textbf{(H4)} hold. For any given $T>0$, Problem \eqref{eq:strongNSCH}--\eqref{eq:bcs} admits at least one global weak solution in the sense of Definition \ref{def:solution}. Moreover, every weak solution satisfies the following energy inequality
\begin{equation}
\mathcal{E}_{\mathrm{tot}}(\uu(t),\phi(t),\rho(t)) + \ii{0}{t}{\|\sqrt{\nu(\phi(\tau), \rho(\tau))}D\uu(\tau)\|^2 + \| \nabla \mu(\tau)\|^2 + \|\nabla \psi(\tau)\|^2 }{\tau} \leq \mathcal{E}_{\mathrm{tot}}(\uu_0,\phi_0,\rho_0),
\label{weak-IEN}
\end{equation}
for every $t \in (0,T]$, where $\mathcal{E}_{\mathrm{tot}}$ is given by \eqref{total}.
%
%
If $d = 2$, then the global weak solution $(\uu, \phi, \rho, \mu, \psi)$ is unique and \eqref{weak-IEN} becomes an equality.
\end{theorem}
\begin{remark}
\label{globalsol}
\textcolor{black}{Since $T>0$ is arbitrary, Theorem \ref{th:weaksolution} and its proof below entail that the global weak solution $( \uu, \phi, \rho, \mu, \psi)$ is indeed defined on $[0,+\infty)$ with}
\begin{align*}
&\uu \in L^{\infty}(0,+\infty; \mathbf{H}_\sigma) \cap L^2(0,+\infty;\mathbf{V}_\sigma)\cap W^{1,\frac{4}{d}}_{\text{loc}}(0,+\infty; \mathbf{V}^*_\sigma),\\
&\phi \in L^{\infty}(0,+\infty; H^2_N(\Omega)) \cap L^{2}_{\text{loc}}(0,+\infty; H^5(\Omega)\cap H^4_N(\Omega))\cap H^1_{\text{loc}}(0,+\infty; (H^1(\Omega))^*),\\
&\rho \in L^{\infty}(0,+\infty; H^1(\Omega)) \cap L^{4}_{\text{loc}}(0,+\infty;H^2_N(\Omega)) \cap L^2_{\text{loc}}(0,+\infty; W^{2,p}(\Omega))\cap H^1_{\text{loc}}(0,+\infty; (H^1(\Omega))^*),\\
&\mu, \psi \in L^2_{\text{loc}}(0,+\infty; H^1(\Omega)),\\
&\rho\in L^\infty(\Omega\times(0,+\infty))\ \text{and}\ 0 < \rho(x, t) < 1\text{ for a.a. }(x,t) \in \Omega \times (0,+\infty).
\end{align*}
\end{remark}
\begin{remark}\label{rem:rho1}
In order to ensure that the initial energy $\mathcal{E}_{\mathrm{tot}}(\uu_0,\phi_0,\rho_0)$ is finite, instead of $S_\rho(\rho_0) \in L^1(\Omega)$, we can alternatively assume that $0 \leq \rho_0 \leq 1$ a.e. in $\Omega$. Besides, the assumption $\overline{\rho_0} \in (0,1)$ implies that the initial state of the surfactant phase $\rho$ cannot be a pure state, namely, $\rho_0$ is not identically equal to $0$ or $1$ over $\Omega$.
\end{remark}

Concerning the strong solutions, we suppose in addition that
\begin{itemize}
\item[\textbf{(H2)}$'$] $S_\phi\in \mathcal{C}^3(\mathbb{R})$;
\item[\textbf{(H3)}$'$] $\widehat{S}_\rho\in \mathcal{C}^3((0,1))$, $R_\rho\in \mathcal{C}^3(\mathbb{R})$.
\end{itemize}
The above assumptions combined with more regular initial data allow us to establish the existence and uniqueness of a local strong solution to problem \eqref{eq:strongNSCH}--\eqref{eq:bcs} in three dimensions (global if $d=2$).

\begin{theorem}[Strong solutions]
\label{th:wellposedlocal}
Let \textbf{(H1)}--\textbf{(H4)} as well as \textbf{(H2)}$'$--\textbf{(H3)}$'$ hold .  For any $\uu_0 \in \mathbf{V}_\sigma$, $\phi_0 \in H^5(\Omega)\cap H^4_N(\Omega)$, $\rho_0 \in H^1(\Omega)$ satisfying $0 \leq \rho_0 \leq 1$ a.e. in $\Omega$, $\overline{\rho_0} \in (0,1)$, and $\widehat{\psi}_0 := -\Delta \rho_0 + \widehat{S}'_\rho(\rho_0) \in H^1(\Omega)$, there exists a time $T^* > 0$ such that problem  \eqref{eq:strongNSCH}--\eqref{eq:bcs} admits a unique local strong solution $( \uu, \phi, \rho, \mu, \psi)$
on $[0,T^*]$ satisfying
\begin{align*}
\uu & \in L^\infty(0,T^*;\mathbf{V}_\sigma) \cap L^2(0,T^*;\mathbf{W}_\sigma) \cap H^1(0,T^*;\mathbf{H}_\sigma), \\
			\phi & \in L^\infty(0,T^*;H^5(\Omega)\cap H^4_N(\Omega))\cap H^1(0,T^*;H^2(\Omega)), \\
			\rho & \in L^\infty(0,T^*;W^{2,p}(\Omega)) \cap H^1(0,T^*;H^1(\Omega)), \\
			\mu & \in L^\infty(0,T^*;H^1(\Omega)) \cap \textcolor{black}{L^2(0,T^*;H^4(\Omega)\cap H^2_N(\Omega))},  \\
			\psi & \in L^\infty(0,T^*;H^1(\Omega)) \cap L^2(0,T^*;H^3(\Omega)\cap H^2_N(\Omega)),
\end{align*}
with a pressure $\pi \in L^2(0,T^*;H^1(\Omega))$ \textcolor{black}{uniquely} defined up to a constant,
where the exponent $p$ satisfies $p\geq 2$ if $d=2$, $2 \leq p \leq 6$ if $d=3$. Moreover, the strong solution is global if $d=2$.
\end{theorem}

In the two dimensional case we are able to say more.
\textcolor{black}{Let us introduce a further assumption on the singular potential}
\begin{enumerate}[label=\textbf{(H\arabic*)}, start=5]
	\item there exists $C>0$ such that
	$$\widehat{S}_\rho''(s) \leq C e^{C\left|\widehat{S}_\rho'(s)\right|}, \quad \forall\, s \in(0,1). $$
\end{enumerate}
This property is fulfilled by the mixing entropy term in \eqref{eq:mixentropy}.
It enables us to derive estimates for the singular terms $\widehat{S}_\rho'(\rho)$ as well as $\widehat{S}_\rho''(\rho)$, which further entails
higher-order regularity of the solution $\rho$. Besides, it plays a role in establishing the strict separation property for $\rho$ in dimension two (see \cite[Section 5]{GGM}, see also \cite{MZ04}).

First of all, a continuous dependence estimate with respect to the norms in \textcolor{black}{$\mathbf{L}^2(\Omega)\times H^2(\Omega)\times H^1(\Omega)$} (i.e. controlled by \eqref{total}) can be obtained for strong solutions in dimension two. This will be useful, for instance, in the analysis of suitable optimal control problems. More precisely, we have
\begin{theorem}[Continuous dependence in dimension two] \label{th:contdep}
Let $d=2$. Suppose that the assumptions of Theorem \ref{th:wellposedlocal} hold \textcolor{black}{and  \textbf{(H5)} is satisfied.} For every pair of strong solutions $(\uu_1, \phi_1, \rho_1, \mu_1, \psi_1)$ and $(\uu_2, \phi_2, \rho_2, \mu_2, \psi_2)$ originating from the admissible initial data $(\uu_{01}, \phi_{01}, \rho_{01})$ and $(\uu_{02},\phi_{02}, \rho_{02})$, the following continuous dependence estimate holds
\begin{align*}
&\|\uu_1(t)- \uu_2(t)\| + \|\phi_1(t) - \phi_2(t)\|_{H^2(\Omega)} + \|\rho_1(t) - \rho_2(t)\|_{H^1(\Omega)} \notag \\
&\quad \leq C_T\left( \|\uu_{01}-\uu_{02}\| + \|\Delta(\phi_{01} - \phi_{02})\| + \|\nabla (\rho_{01}-\rho_{02})\|  + |\overline{\phi_{01}}-\overline{\phi_{02}}| + |\overline{\rho_{01}}-\overline{\rho_{02}}| \right),
\end{align*}
for every $t \in [0,T]$. Here, $C_T>0$ is a constant depending on $\|\uu_{0i}\|_{\mathbf{V}_\sigma}$, $\|\phi_{0i}\|_{H^5(\Omega)}$, $\|\rho_{0i}\|_{H^1(\Omega)}$, $\|\widehat{\psi}_{0i}\|_{H^1(\Omega)}$, $i=1,2$, coefficients of the system, $\Omega$ and $T$.
\end{theorem}

Next, we prove that the global weak solution regularizes in finite time, namely,
\begin{theorem}[Regularity of weak solutions in dimension two]
	\label{th:regularity3}
	Let $d = 2$. Assume that \textbf{(H1)}--\textcolor{black}{\textbf{(H5)}} and \textbf{(H2)}$'$--\textbf{(H3)}$'$ hold. Suppose that $p \geq 2$, $K > 0$, $m_1 \in \mathbb{R}$, $m_2 \in (0,1)$ and $\delta >0$ are given. For any $\uu_0 \in \mathbf{H}_\sigma$, $\phi_0 \in H^2_N(\Omega)$, $\rho_0 \in H^1(\Omega)$ be such that $S_\rho(\rho_0) \in L^1(\Omega)$ and $\overline{\phi_0} = m_1$, $\overline{\rho_0} = m_2$, $\mathcal{E}_{\mathrm{tot}}(\uu_0,\phi_0, \rho_0) \leq K$, let $(\uu, \phi, \rho, \mu, \psi)$ be the unique global weak solution to problem \eqref{eq:strongNSCH}--\eqref{eq:bcs} originating from the initial data $(\uu_0, \phi_0, \rho_0)$.
Then there exists $C_1 > 0$ depending on $K, p, m_1, m_2$ and $\delta$ such that
	\begin{equation}
		\|\uu(t)\|_{\mathbf{V}_\sigma}+\|\phi(t)\|_{H^5(\Omega)} + \|\rho(t)\|_{W^{2,p}(\Omega)}+\|\mu(t)\|_{H^1(\Omega)} + \|\psi(t)\|_{H^1(\Omega)} \leq C_1,\quad \forall\, t \geq \delta. \label{regg1}
	\end{equation}
Besides, there exists $\eta \in (0,1/2]$ such that
		\begin{equation}
		\textcolor{black}{\eta \leq \rho \leq 1-\eta,\quad \text{for all}\ x\in \Omega,\ t\geq \delta. }
        \label{regg2}
		\end{equation}
Moreover, there exists $C_2 > 0$ depending on $K, m_1, m_2$ and $\delta$ such that
	\begin{align}
		&\|\partial_t \uu(t)\| + \textcolor{black}{ \|\partial_t \phi(t)\|_{H^2(\Omega)}+\|\partial_t \rho(t)\|_{H^1(\Omega)} } \leq C_2,
           \label{regg11}\\
		&\|\partial_t\uu\|_{L^2(t,t+1;\mathbf{V}_\sigma)}
           + \textcolor{black}{\|\partial_t\phi\|_{L^2(t,t+1;H^5(\Omega))} + \|\partial_t\rho\|_{L^2(t,t+1;H^3(\Omega))} } \leq C_2,
           \label{regg12}\\
        & \|\uu(t)\|_{\mathbf{W}_\sigma} + \textcolor{black}{\|\mu(t)\|_{H^4(\Omega)} + \|\psi(t)\|_{H^2(\Omega)} + \|\phi(t)\|_{H^6(\Omega)} + \|\rho(t)\|_{H^4(\Omega)} }
        \leq C_2,
        \label{regg3}
	\end{align}
	hold for every $t \geq 2\delta$. In particular, any weak solution becomes strong for $t>0$.
\end{theorem}

\section{Proof of Theorem \ref{th:weaksolution}}
\label{sec:proof1}
The proof of Theorem \ref{th:weaksolution} consists of several steps. The first ingredient is the following
\begin{proposition}
\label{MEconv}
Let $(\uu, \phi, \rho, \mu, \psi)$ be a sufficiently smooth solution to problem \eqref{eq:strongNSCH}--\eqref{eq:bcs} on $[0,T]$. Then we have
\begin{align}
& \frac{\mathrm{d}}{\mathrm{d}t}\int_\Omega \phi(t)\,\mathrm{d} x = \frac{\mathrm{d}}{\mathrm{d}t} \int_\Omega \rho(t)\,\mathrm{d} x=0,\quad \forall\, t\in (0,T), \label{Mass} \\
& \frac{\mathrm{d}}{\mathrm{d}t} \mathcal{E}_{\mathrm{tot}}(\uu(t),\phi(t),\rho(t))
+ \|\sqrt{\nu(\phi(t), \rho(t))}D\uu(t)\|^2 + \| \nabla \mu(t)\|^2 + \|\nabla \psi(t)\|^2 =0,\quad \forall\, t\in (0,T),
\label{BEL}
\end{align}
where $\mathcal{E}_{\mathrm{tot}}(\uu(t),\phi(t),\rho(t)) $ is defined as in  \eqref{total}.
\end{proposition}
\begin{proof}
To deduce \eqref{Mass}, we simply test the Cahn--Hilliard type equations in \eqref{eq:strongNSCH} by $1$ and integrate over $\Omega$. Then \eqref{Mass} follows through an integration by parts, thanks to the homogeneous Neumann boundary conditions for $\mu$, $\psi$, the no-slip boundary condition for $\uu$ and \textcolor{black}{the incompressibility condition $\nabla \cdot \uu=0$.}
The energy identity \eqref{BEL} can be obtained by testing the first, third, fourth, fifth and sixth equations in \eqref{eq:strongNSCH} by $\uu$, $\mu$, $\partial_t \phi$, $\psi$ and $\partial_t\rho$, respectively, integrating over $\Omega$ and using the incompressibility condition as well as the boundary conditions for $(\uu, \phi, \rho)$. We thus get
\begin{align*}
\frac12 \frac{\mathrm{d}}{\mathrm{d}t} \|\uu\|^2 + \int_\Omega \nu(\phi, \rho) |D\uu|^2\, \mathrm{d}x
&= -\int_\Omega (\mathbf{u} \cdot \nabla)\mathbf{u} \cdot \uu\, \mathrm{d}x -\int_\Omega \nabla \pi \cdot \uu\, \mathrm{d}x
+ \int_\Omega (\mu \nabla \phi  +  \psi\nabla \rho) \cdot \uu\, \mathrm{d}x\\
& = \int_\Omega (\mu \nabla \phi  +  \psi\nabla \rho) \cdot \uu \, \mathrm{d}x,\\
&\int_\Omega \partial_t \phi \mu \,\mathrm{d}x +\int_\Omega (\mathbf{u} \cdot \nabla \phi)\mu\, \mathrm{d}x   = -\|\nabla \mu\|^2,\\
\int_\Omega  \mu \partial_t \phi \,\mathrm{d}x
&= \int_\Omega  \big(\alpha \Delta^2 \phi-\Delta \phi  + S_\phi'(\phi) + \theta\nabla \cdot (\rho \nabla \phi)\big) \partial_t \phi \,\mathrm{d}x\\
&= \frac{\mathrm{d}}{\mathrm{d}t}\int_\Omega \left(\frac{\alpha}{2}|\Delta \phi|^2 +\frac{1}{2}|\nabla \phi|^2+S_\phi(\phi) \right)\,\mathrm{d}x - \frac{\theta}{2}\int_\Omega \rho \partial_t|\nabla \phi|^2 \,\mathrm{d}x,\\
&\int_\Omega \partial_t \rho \psi \,\mathrm{d}x +\int_\Omega (\mathbf{u} \cdot \nabla \rho)\psi\, \mathrm{d}x = -\|\nabla \psi\|^2,\\
\int_\Omega \psi \partial_t \rho\, \mathrm{d} x &= \int_\Omega \left(  -\beta \Delta \rho + S'_\rho(\rho) - \dfrac{\theta}{2}|\nabla \phi|^2 \right)\partial_t \rho\,  \mathrm{d}x\\
&= \frac{\mathrm{d}}{\mathrm{d}t}\int_\Omega \left(\frac{\beta}{2}|\nabla \rho|^2+S_\rho(\rho)\right)\,\mathrm{d}x  - \dfrac{\theta}{2}\int_\Omega |\nabla \phi|^2 \partial_t \rho\,  \mathrm{d}x.
\end{align*}
Collecting the above identities together, we easily conclude \eqref{BEL}.
\end{proof}
\begin{remark} \label{Rm-ill}
Integrating \eqref{Mass} and \eqref{BEL} with respect to time, we find that
\begin{align}
\notag
&\int_\Omega \phi(t) \: \mathrm{d}x = \int_\Omega \phi_0 \:\mathrm{d}x,\quad  \int_\Omega\rho \:\mathrm{d}x = \int_\Omega\rho_0\:\mathrm{d}x, \quad \forall\, t\in [0,T],\\
&\mathcal{E}_{\mathrm{tot}}(\uu(t),\phi(t),\rho(t)) + \int_0^t \left(\|\sqrt{\nu(\phi(t), \rho(t))}D\uu(t)\|^2 + \| \nabla \mu(t)\|^2 + \|\nabla \psi(t)\|^2\right)\,\mathrm{d}\tau \notag \\
&=\mathcal{E}_{\mathrm{tot}}(\uu_0,\phi_0,\rho_0) ,\quad \forall\, t\in (0,T].
\label{BELa}
\end{align}
Identity \eqref{BELa} shows that, physically as well as mathematically, $\mathcal{E}_{\mathrm{tot}}$ must be bounded from below.
The singular potential $S_\rho$ (formally) implies that
\begin{align}
0\leq \rho\leq 1, \quad \text{for a.e. }(x,t)\in \Omega\times(0,T).
\label{rhobd1}
\end{align}
Thus, we can directly infer from \eqref{rhobd1} that
\begin{equation}
\label{theta}
\int_\Omega - \dfrac{\theta}{2} \rho |\nabla \phi|^2\,\mathrm{d}x
\geq -\frac{\theta}{2}\int_\Omega |\nabla \phi|^2\,\mathrm{d}x =  \frac{\theta}{2}\int_\Omega \phi \Delta \phi\,\mathrm{d}x
\geq -\frac{\alpha}{4}\|\Delta\phi\|^2- \frac{\theta^2}{4\alpha}\|\phi\|^2.
\end{equation}
The first term can be easily controlled by the higher-order term $\frac{\alpha}{2}\|\Delta\phi\|^2$ in $\mathcal{E}_{\mathrm{tot}}$. Concerning the second term, without making any assumption on the size of the positive parameters $\alpha, \theta$, it can still be handled thanks to the coercivity of $S_\phi$. Indeed, from \textbf{(H2)} and Young's inequality, we have
$$
\int_\Omega S_\phi(\phi)\,\mathrm{d} x\geq c_3\int_\Omega |\phi|^4\,\mathrm{d}x -c_4|\Omega|
\geq \frac{c_3}{2}\int_\Omega |\phi|^4\,\mathrm{d}x +\frac{\theta^2}{4\alpha}\|\phi\|^2
-\left(c_4 + \frac{\theta^4}{32\alpha^2 c_3} \right)|\Omega|.
$$
Therefore, the bound \eqref{rhobd1} of $\rho$ plays a crucial role. However, when we prove the existence of weak solutions to problem \eqref{eq:strongNSCH}--\eqref{eq:bcs} we need to introduce a suitable regularization of the singular potential $S_\rho$ (see \eqref{eq:taylor} below) and \eqref{rhobd1} can no longer be guaranteed due to the lack of maximum principle for the fourth order Cahn--Hilliard equation. This is the reason why \textcolor{black}{we have to introduce a further penalization term} in the approximating problem (see \eqref{penalenergy} below, see also \cite{Xu21} for a similar strategy in a numerical context). Note that the presence of the second-order term in the energy  $\mathcal{E}_{\mathrm{tot}}$ would not play any role in proving its boundedness from below provided that $\theta\in (0,1)$ (see Remark \ref{FHphi}).
\end{remark}

\subsection{A regularized problem} \label{sub:pert}
In view of Remark \ref{Rm-ill}, we consider the penalized energy
\begin{equation}
\label{penalenergy}
	\mathcal{E}_\omega(\mathbf{u},\phi,\rho) :=\ii{\Omega}{}{ \left( \dfrac{1}{2}|\mathbf{u}|^2+ \dfrac{\alpha}{2} |\Delta \phi|^2 + \dfrac{1}{2}|\nabla \phi|^2  + S_\phi(\phi) + \dfrac{\beta}{2}|\nabla \rho|^2 +  S_\rho(\rho) - \dfrac{\theta}{2} \rho |\nabla \phi|^2 + \frac{\omega}{4} |\nabla \phi|^4\right)}{x},
\end{equation}
where $\omega\in (0,1]$ is a given parameter. Correspondingly, the initial boundary value problem associated to the perturbed energy functional $\mathcal{E}_\omega$ is the following
\begin{equation} \label{eq:strongNSCHpert}
\begin{cases}
\partial_t\mathbf{u} + (\mathbf{u} \cdot \nabla)\mathbf{u} - \nabla \cdot (\nu(\phi,\rho) D\mathbf{u} ) + \nabla \pi =\mu \nabla \phi +  \psi\nabla \rho& \quad \text{in } \Omega \times (0,T),\\
\nabla \cdot \mathbf{u} = 0& \quad \text{in } \Omega \times (0,T),\\
\partial_t \phi + \mathbf{u} \cdot \nabla \phi = \Delta \mu & \quad \text{in } \Omega \times (0,T),\\
\mu = \alpha \Delta^2 \phi-\Delta \phi  + S_\phi'(\phi) + \theta\nabla \cdot (\rho \nabla \phi) - \textcolor{black}{\omega \nabla \cdot (|\nabla \phi|^2 \nabla \phi)}& \quad \text{in } \Omega \times (0,T),\\
\partial_t \rho + \mathbf{u} \cdot \nabla \rho = \Delta \psi & \quad \text{in } \Omega \times (0,T), \\
\psi = -\beta \Delta \rho + S'_\rho(\rho) - \dfrac{\theta}{2}|\nabla \phi|^2 & \quad \text{in } \Omega \times (0,T), \\
\end{cases}
\end{equation}
subject to the boundary and initial conditions
\begin{equation} \label{eq:bcspert}
\begin{cases}
			\mathbf{u} = \mathbf{0} & \quad \text{on } \partial \Omega \times (0,T), \\
			\partial_\mathbf{n} \phi = \partial_\mathbf{n}\Delta \phi = \partial_\mathbf{n}\mu =0 & \quad \text{on } \partial \Omega \times (0,T), \\
			\partial_\mathbf{n} \rho = \partial_\mathbf{n} \psi = 0 & \quad \text{on } \partial \Omega \times (0,T), \\
			\mathbf{u}|_{t=0}=\mathbf{u}_0(\mathbf{x}),\quad \phi|_{t=0} = \phi_0(\mathbf{x}),\quad \rho|_{t=0} = \rho_0(\mathbf{x}),	 & \quad \text{in } \Omega.
\end{cases}
\end{equation}

To prove the existence of a global weak solution to problem \eqref{eq:strongNSCHpert}--\eqref{eq:bcspert}, we introduce a suitable approximation of the singular potential $S_\rho$, dependent on some (small) parameter $\varepsilon>0$ in such a way that the original potential is recovered in the limit $\varepsilon \to 0^+$. More precisely, following \cite{FG12}
(see also \cite{EG96}), we consider a family of regular potentials based upon the second-order Taylor's expansion of $\widehat{S}_\rho$. Recalling \textbf{(H3)},
for any sufficiently small $\varepsilon \in (0,\epsilon_1)$, let $\widehat{S}_{\rho,\varepsilon}: \mathbb{R} \to \mathbb{R}$ be a globally defined approximation of $\widehat{S}_\rho$ given by
	\begin{equation}
		\label{eq:taylor}
		\widehat{S}_{\rho,\varepsilon}(s) =
		\begin{cases}
			\displaystyle\sum_{i=0}^{2} \dfrac{\widehat{S}_\rho^{(i)}(\varepsilon)}{i!}(s-\varepsilon)^i & \quad \text{if } s \leq \varepsilon,\\[0.2cm]
			\widehat{S}_\rho(s) & \quad \text{if } \varepsilon < s < 1-\varepsilon,\\
			\displaystyle\sum_{i=0}^{2} \dfrac{\widehat{S}_\rho^{(i)}(1-\varepsilon)}{i!}[s-(1-\varepsilon)]^i& \quad \text{if } s \geq 1 - \varepsilon.
		\end{cases}
	\end{equation}
Set
	\begin{equation} \label{eq:approxpot}
		S_{\rho, \varepsilon}(s) := \widehat{S}_{\rho, \varepsilon}(s) + R_\rho(s),\quad \forall\,s \in \mathbb{R}.
	\end{equation}
Then for any $\varepsilon \in (0,\epsilon_1)$, there exist constants $\gamma_1, \gamma_2, \gamma_3 > 0$ such that
	\begin{equation} \label{eq:convexityapprox}
		\widehat{S}_{\rho, \varepsilon}(s)\geq -\gamma_1, \qquad -\gamma_2 \leq \widehat{S}''_{\rho, \varepsilon}(s) \leq \gamma_3, \quad \forall\, s\in \mathbb{R},
	\end{equation}
where $\gamma_1, \gamma_2$ are independent of $\varepsilon$, while the upper bound $\gamma_3$ may depend on $\varepsilon$.

Our strategy is as follows. We first find a global weak solution to a regularized system of problem \eqref{eq:strongNSCHpert}--\eqref{eq:bcspert} with $S_{\rho}$ replaced by the regularized potential \eqref{eq:approxpot}. Then we derive uniform estimates and pass to the limit first as $\varepsilon \to 0^+$ to obtain a solution to the penalized problem \eqref{eq:strongNSCHpert}--\eqref{eq:bcspert}. Finally, we will get rid of the penalization term by letting also $\omega\to 0^+$.

\subsection{The Galerkin scheme} \label{sub:galerkin}
Let us consider the regularized system
\begin{equation} \label{eq:strongNSCHpertap}
\begin{cases}
\partial_t\mathbf{u} + (\mathbf{u} \cdot \nabla)\mathbf{u} - \nabla \cdot (\nu(\phi,\rho) D\mathbf{u} ) + \nabla \pi =\mu \nabla \phi +  \psi\nabla \rho& \quad \text{in } \Omega \times (0,T),\\
\nabla \cdot \mathbf{u} = 0& \quad \text{in } \Omega \times (0,T),\\
\partial_t \phi + \mathbf{u} \cdot \nabla \phi = \Delta \mu & \quad \text{in } \Omega \times (0,T),\\
\mu = \alpha \Delta^2 \phi-\Delta \phi  + S_\phi'(\phi) + \theta\nabla \cdot (\rho \nabla \phi) - \textcolor{black}{\omega \nabla \cdot (|\nabla \phi|^2 \nabla \phi)}& \quad \text{in } \Omega \times (0,T),\\
\partial_t \rho + \mathbf{u} \cdot \nabla \rho = \Delta \psi & \quad \text{in } \Omega \times (0,T), \\
\psi = -\beta \Delta \rho + S'_{\rho, \varepsilon}(\rho) - \dfrac{\theta}{2}|\nabla \phi|^2 & \quad \text{in } \Omega \times (0,T), \\
\end{cases}
\end{equation}
subject to the  boundary and initial conditions \eqref{eq:bcspert}.
Its weak formulation reads essentially the same as \eqref {eq:weakNSCH} in Definition \ref{def:solution}
with obvious modifications. Observe that system \eqref{eq:strongNSCHpertap} is associated with the following energy functional
\begin{equation}
\label{eq:approxenerg}
\mathcal{E}_{\omega,\varepsilon}(\mathbf{u},\phi,\rho) :=\ii{\Omega}{}{ \left( \dfrac{1}{2}|\mathbf{u}|^2+ \dfrac{\alpha}{2} |\Delta \phi|^2 + \dfrac{1}{2}|\nabla \phi|^2  + S_\phi(\phi) + \dfrac{\beta}{2}|\nabla \rho|^2 +  S_{\rho, \varepsilon}(\rho) - \dfrac{\theta}{2} \rho |\nabla \phi|^2 + \frac{\omega}{4} |\nabla \phi|^4\right)}{x}.
\end{equation}

The existence of a global weak solution to the regularized problem \eqref{eq:strongNSCHpertap} with \eqref{eq:bcspert} can be proved by using a suitable Galerkin approximation scheme. Recall the countably many eigencouples of the (negative) Neumann--Laplace operator, denoted by $(\eta_n , w_n) \in \mathbb{R} \times L^2(\Omega)$, $n\in \mathbb{Z}^+$. We note that $\{w_n\}$ forms an orthonormal basis of $L^2(\Omega)$ and is also an orthogonal basis of $H^2_N(\Omega)$. Analogously, we set $(\zeta_n, \mathbf{w}_n) \in \mathbb{R} \times \mathbf{H}_\sigma$ to be the countably many eigencouples of the Stokes operator $\mathbf{A}$ and $\{\mathbf{w}_n\}$ forms an orthonormal basis of $\mathbf{H}_\sigma$ and also an orthogonal basis of $\mathbf{W}_\sigma$. We set $W_n := \spn \{w_1, ..., w_n\}\subset H^2_N(\Omega)$, $\mathbf{W}_n := \spn\{\mathbf{w}_1, ...,\mathbf{ w}_n\}\subset \mathbf{W}_\sigma$, with corresponding orthogonal projections $\Pi_n: L^2(\Omega) \to W_n$ (with respect to the inner product in $L^2(\Omega)$) and $P_n: \mathbf{H}_\sigma \to \mathbf{W}_n$ (with respect to the inner product in $\mathbf{H}_\sigma$). Then we consider the following Galerkin scheme that depends on three approximating parameters $n, \varepsilon$ and $\omega$. Namely, for $\omega\in (0,1]$, $\varepsilon \in (0,\epsilon_1)$ and $n\in \mathbb{Z}^+$, we look for functions $(\uu_\omega^{n,\varepsilon},\phi_\omega^{n,\varepsilon}, \rho_\omega^{n,\varepsilon}, \mu_\omega^{n,\varepsilon}, \psi_\omega^{n,\varepsilon})$ of the form:
\begin{align*}
\phi^{n, \varepsilon}_\omega(t) = \sum_{i = 1}^n a_{i}(t)w_i, & \quad \rho^{n, \varepsilon}_\omega(t) = \sum_{i = 1}^n b_{i}(t)w_i,\\
\mu^{n, \varepsilon}_\omega(t) = \sum_{i = 1}^n c_{i}(t)w_i, & \quad \psi^{n, \varepsilon}_\omega(t) = \sum_{i = 1}^n d_{i}(t)w_i, \\
\uu^{n, \varepsilon}_\omega(t) =  \sum_{i = 1}^n e_{i}(t)\mathbf{w}_i,&
\end{align*}
which solve the following problem:
\begin{equation} \label{eq:weakrega}
\begin{cases}
\left\langle \partial_t\mathbf{u}^{n, \varepsilon}_\omega, \mathbf{v} \right\rangle_{\mathbf{V}^*_\sigma, \mathbf{V}_\sigma} + ((\mathbf{u}^{n, \varepsilon}_\omega \cdot \nabla)\mathbf{u}^{n, \varepsilon}_\omega,\mathbf{v}) + (\nu(\phi^{n, \varepsilon}_\omega,\rho^{n, \varepsilon}_\omega) D\mathbf{u}^{n, \varepsilon}_\omega, D\mathbf{v}) &\\
\qquad \qquad = (\mu^{n, \varepsilon}_\omega\nabla \phi^{n, \varepsilon}_\omega, \mathbf{v}) + (\psi^{n, \varepsilon}_\omega\nabla \rho^{n, \varepsilon}_\omega, \mathbf{v}) & \quad \forall \: \mathbf{v} \in \mathbf{W}_n, \aevv,\\
\left\langle \partial_t\phi^{n, \varepsilon}_\omega, v \right\rangle_{V^*,V} + (\uu^{n, \varepsilon}_\omega \cdot \nabla \phi^{n, \varepsilon}_\omega, v) + (\nabla \mu^{n, \varepsilon}_\omega, \nabla v) = 0 & \quad \forall \: v \in W_n, \aevv,\\
\mu^{n, \varepsilon}_\omega = \Pi_n \big(\alpha \Delta^2 \phi^{n, \varepsilon}_\omega  -\Delta \phi^{n, \varepsilon}_\omega +  S'_{\phi}(\phi^{n, \varepsilon}_\omega) + \theta\nabla \cdot (\rho^{n, \varepsilon}_\omega \nabla \phi^{n, \varepsilon}_\omega)\big) &\\
\qquad \quad - \Pi_n \big(\textcolor{black}{\omega \nabla \cdot\left( |\nabla \phi^{n, \varepsilon}_\omega|^2\nabla \phi^{n, \varepsilon}_\omega \right) }
\big)& \quad \text{a.e. in } \Omega \times (0,T),\\
\left\langle \partial_t\rho^{n, \varepsilon}_\omega, v \right\rangle_{V^*,V} + (\uu^{n, \varepsilon}_\omega \cdot \nabla \rho^{n, \varepsilon}_\omega, v) + (\nabla \psi^{n, \varepsilon}_\omega, \nabla v) = 0& \quad \forall \: v \in W_n, \aevv,\\
\psi^{n, \varepsilon}_\omega = \Pi_n \Big( -\beta \Delta \rho^{n, \varepsilon}_\omega + S'_{\rho,\varepsilon}(\rho^{n, \varepsilon}_\omega) - \dfrac{\theta}{2}|\nabla \phi^{n, \varepsilon}_\omega|^2 \Big) & \quad \text{a.e. in } \Omega \times (0,T), \\
\uu^{n, \varepsilon}_\omega(\cdot,0) = P_n(\uu_0) =: \uu_{0}^n & \quad \text{in } \Omega,\\
\phi^{n, \varepsilon}_\omega(\cdot, 0) = \Pi_n(\phi_{0}) =: \phi_{0}^n, \quad
\rho^{n, \varepsilon}_\omega(\cdot, 0) = \Pi_n(\rho_{0}) =: \rho_{0}^n & \quad \text{in } \Omega.
\end{cases}
\end{equation}
Inserting the expressions of those approximate solutions into the above weak formulation, we arrive at a system consisting of $5n$ ordinary differential equations in the unknowns $\big(a_i(t), b_i(t), c_i(t), d_i(t), e_i(t)\big)$, $i=1,...,n$. Recalling the assumptions \textbf{(H1)}--\textbf{(H4)}, an application of the Cauchy--Lipschitz theorem entails
\begin{proposition}\label{wellODE}
For any positive integer $n$, there exists $T_n \in(0,T]$ such that  problem \eqref{eq:weakrega} admits a unique local solution $(\uu_\omega^{n,\varepsilon},\phi_\omega^{n,\varepsilon}, \rho_\omega^{n,\varepsilon}, \mu_\omega^{n,\varepsilon}, \psi_\omega^{n,\varepsilon})$ on $[0,T_n]$, which is given by the functions $a_i, b_i, c_i, d_i, e_i\in \mathcal{C}^1([0,T_n])$, $i=1, \dots, n$.
\end{proposition}

\subsection{Uniform estimates}
Here we proceed to derive some bounds of the local approximating solutions that are uniform with respect to $n$, $\varepsilon$, and $\omega$.
The first one is the following energy estimate (cf. Proposition \ref{MEconv})
\begin{lemma} \label{prop:enebound1}
For every $t \in (0,T_n]$, it holds
\begin{align*}
&\ene(\mathbf{u}^{n, \varepsilon}_\omega(t), \phi^{n, \varepsilon}_\omega(t),\rho^{n, \varepsilon}_\omega(t)) + \ii{0}{t}{ \|\sqrt{\nu(\phi^{n, \varepsilon}_\omega(\tau),\rho^{n, \varepsilon}_\omega(\tau))}D \mathbf{u}^{n, \varepsilon}_\omega(\tau) \|^2 + \| \nabla \mu^{n, \varepsilon}_\omega(\tau) \|^2 + \| \nabla \psi^{n, \varepsilon}_\omega(\tau)\|^2 }{\tau} \leq C_1,
\end{align*}
and
\begin{align*}
\ene(\mathbf{u}^{n, \varepsilon}_\omega(t),\phi^{n, \varepsilon}_\omega(t),\rho^{n, \varepsilon}_\omega(t))
&\geq \dfrac{1}{2}\|\uu_{\omega}^{n,\varepsilon}(t)\|^2
+ \dfrac{\alpha}{4} \|\Delta \phi^{n, \varepsilon}_\omega(t)\|^2
+ \dfrac{1}{2} \|\nabla \phi^{n, \varepsilon}_\omega(t)\|^2
+ \dfrac{\beta}{2}\|\nabla \rho^{n, \varepsilon}_\omega(t)\|^2
\\
&\quad + \frac{c_3}{2} \|\phi^{n, \varepsilon}_\omega(t)\|^4_{L^4(\Omega)}
+ \dfrac{\omega}{8} \|\nabla\phi^{n, \varepsilon}_\omega(t)\|^4_{\mathbf{L}^4(\Omega)} -C_2,
\end{align*}
where $C_1>0$ is independent of $n$ and $\omega$, while $C_2>0$ is independent of $n$, $\varepsilon$, and $\omega$.
\end{lemma}
\begin{proof}
Arguing as in the proof of Proposition \ref{MEconv}, in \eqref{eq:weakrega} we can take the test functions $\mathbf{v} = \uu^{n, \varepsilon}_\omega$, $v = \mu^{n, \varepsilon}_\omega$ and $v = \psi^{n, \varepsilon}_\omega$ in the equations for $\uu^{n, \varepsilon}_\omega$, $\phi^{n, \varepsilon}_\omega$ and $\rho^{n, \varepsilon}_\omega$, respectively, while multiplying the equations for the chemical potentials by $\partial_t \phi^{n, \varepsilon}_\omega$ and $\partial_t \rho^{n, \varepsilon}_\omega$ accordingly. Combining all the resulting equalities, we find
  \[
		\td{}{t}\ene(\mathbf{u}^{n, \varepsilon}_\omega, \phi^{n, \varepsilon}_\omega,\rho^{n, \varepsilon}_\omega)
+  \|\sqrt{\nu(\phi^{n, \varepsilon}_\omega,\rho^{n, \varepsilon}_\omega)}D \mathbf{u}^{n, \varepsilon}_\omega\|^2
+ \| \nabla \mu^{n, \varepsilon}_\omega\|^2
+ \| \nabla \psi^{n, \varepsilon}_\omega \|^2= 0, \quad \forall\, t\in (0,T_n).
  \]
For any $t \in (0,T_n]$, integrating the above identity over $[0,t]$, we obtain
\begin{align}
&\ene(\mathbf{u}^{n, \varepsilon}_\omega(t), \phi^{n, \varepsilon}_\omega(t),\rho^{n, \varepsilon}_\omega(t)) + \ii{0}{t}{\|\sqrt{\nu(\phi^{n, \varepsilon}_\omega(\tau),\rho^{n, \varepsilon}_\omega(\tau))}D \mathbf{u}^{n, \varepsilon}_\omega(\tau)\|^2
+\| \nabla \mu^{n, \varepsilon}_\omega(\tau) \|^2 + \| \nabla \psi^{n, \varepsilon}_\omega(\tau)\|^2 }{\tau} \notag\\
 &\quad = \ene(\uu^{\varepsilon}_\omega(0), \phi^{n, \varepsilon}_\omega(0), \rho^{n, \varepsilon}_\omega(0)).
\label{eq:energyn}
\end{align}
Concerning the initial energy, we have
\begin{align*}
&\ene(\uu^{n, \varepsilon}_\omega(0), \phi^{n, \varepsilon}_\omega(0), \rho^{n, \varepsilon}_\omega(0)) = \ene(\uu_{0}^n , \phi_{0}^n, \rho_{0}^n)\\
&\qquad  = \ii{\Omega}{}{\left(\dfrac{1}{2}|\uu_{0}^n|^2 + \dfrac{\alpha}{2} |\Delta \phi_{0}^n|^2 + \dfrac{1}{2}|\nabla \phi_{0}^n|^2 + {S}_{\phi}(\phi_{0}^n) + \dfrac{\beta}{2}|\nabla \rho_{0}^n|^2 + \widehat{S}_{\rho, \varepsilon}(\rho_{0}^n)
+ R_\rho(\rho_{0}^n) \right.\\
&\qquad\qquad \ \  \left. - \dfrac{\theta}{2} \rho_{0}^n |\nabla \phi_{0}^n|^2 + \frac{\omega}{4}|\nabla \phi_{0}^n|^4\right)}{x}.
\end{align*}
We easily obtain the bounds
\begin{align*}
& \| \uu^n_0\|^2 \leq \| \uu_0\|^2,\quad \| \nabla \phi_{0}^n\|^2 \leq \hh{\phi_{0}^n}^2 \leq \hh{\phi_{0}}^2, \quad 	\| \nabla \rho_{0}^n\|^2 \leq \hh{\rho_{0}^n}^2 \leq \hh{\rho_{0}}^2.
\end{align*}
Moreover, we notice that since $\phi^n_0 \to \phi_0$ in $H^2(\Omega)$, there exists $n^{*} \in \mathbb{N}$ such that for all $n > n^{*}$
		\[
		\| \Delta \phi_{0}^n\|^2 \leq \|\phi_{0}^n\|^2_{H^2(\Omega)} \leq C\big(1 + \|\phi_{0}\|^2_{H^2(\Omega)}\big),
		\]
where $C$ is independent of $n$.
Also, we infer from \textbf{(H3)} and \eqref{eq:taylor} that
$$
\widehat{S}_{\rho,\varepsilon}(s) \leq C(\varepsilon)\left(1+ s^2\right)\quad\text{and}\quad  |R_\rho(s)|\leq C_R(1+s^2),\quad \forall\,s\in \mathbb{R},
$$
where $C(\varepsilon)>0$ may depend on $\varepsilon$ and $C_R>0$ is a constant only depending on $R_\rho(0)$, $R'_\rho(0)$ and $L_1$.
On account of \textbf{(H2)}, we deduce that
\begin{align}
\left| \ii{\Omega}{}{{S}_{\phi}(\phi_{0}^n) + \widehat{S}_{\rho,\varepsilon}(\rho_{0}^n) + R_\rho(\rho_{0}^n)}{x} \right|
&\leq C(\|\phi_0^n\|_{H^2(\Omega)})+C(\varepsilon) (\|\rho_0^n\|^2 + 1) \notag\\
&\leq C(\|\phi_0\|_{H^2(\Omega)})+ C(\varepsilon)(\|\rho_0\|^2 + 1),
\label{eq:epsilonestimate}
\end{align}
where we have also used the Sobolev embedding $H^2(\Omega) \hookrightarrow L^\infty(\Omega)$ ($d=2,3$). The symbol $C(\|\phi_0\|_{H^2(\Omega)})$ denotes a positive constant depending on $\|\phi_0\|_{H^2(\Omega)}$ and $\Omega$ but not on $n,\omega, \varepsilon$, while $C(\varepsilon)$ is independent of $n$ and $\omega$. The remaining two terms in the approximate initial energy are treated by using the Cauchy--Schwarz as well as Young's inequalities:
\begin{align*}
\left| \ii{\Omega}{}{ - \dfrac{\theta}{2} \rho_{0}^n |\nabla \phi_{0}^n|^2 + \frac{\omega}{4}|\nabla \phi_{0}^n|^4}{x} \right|
& \leq \frac12 \|\rho_{0}^n\|^2+ \left(\frac{\omega}{4}+ \frac{\theta^2}{8}\right) \int_\Omega |\nabla \phi_{0}^n|^4\,\mathrm{d}x \\
& \leq C\big( \|\nabla \phi_{0}^n\|_{\mathbf{L}^4(\Omega)}^4 + \|\rho_{0}^n\|^2\big)\leq C\big( \|\phi_0^n\|_{H^2(\Omega)}^4 + \|\rho_0^n\|^2 \big) \\
& \leq C\big( \|\phi_0\|_{H^2(\Omega)}^4 + \|\rho_0\|^2 \big),
\end{align*}
where we have also used the Sobolev embedding $H^2(\Omega) \hookrightarrow W^{1,4}(\Omega)$ ($d=2,3$). Collecting the above estimates, we obtain the required upper bound by choosing a suitable constant $C_1>0$ depending on $\varepsilon$, $\| \uu_0 \|$, $\| \phi_0\|_{H^2(\Omega)}$, $ \| \rho_0 \|_{H^1(\Omega)}$ and $\Omega$, but independent of $n$ and $\omega$.

Concerning the lower bound, we exploit some observations made in \cite{FG12,Xu21}. Consider the energy functional for the approximate solution (recall \eqref{eq:approxenerg})
\begin{align*}
\ene(\mathbf{u}^{n, \varepsilon}_\omega, \phi^{n, \varepsilon}_\omega,\rho^{n, \varepsilon}_\omega)
& = \dfrac{1}{2}\|\uu_{\omega}^{n,\varepsilon}\|^2+ \dfrac{\alpha}{2} \|\Delta \phi^{n, \varepsilon}_\omega\|^2 + \dfrac{1}{2} \|\nabla \phi^{n, \varepsilon}_\omega\|^2
+ \dfrac{\beta}{2}\|\nabla \rho^{n, \varepsilon}_\omega\|^2
   \\
   &\quad  + \ii{\Omega}{}{\left(S_{\phi}(\phi^{n, \varepsilon}_\omega) + \widehat{S}_{\rho, \varepsilon}(\rho^{n, \varepsilon}_\omega)
   +R_\rho(\rho^{n, \varepsilon}_\omega) - \dfrac{\theta}{2} \rho^{n, \varepsilon}_\omega |\nabla \phi^{n, \varepsilon}_\omega|^2+\frac{\omega}{4} | \nabla \phi^{n, \varepsilon}_\omega|^4\right)}{x}.
\end{align*}
From \textbf{(H2)} we infer that
\begin{align*}
&\ii{\Omega}{}{S_{\phi}(\phi^{n, \varepsilon}_\omega) }{x}\geq c_3\int_\Omega |\phi^{n, \varepsilon}_\omega|^4\,\mathrm{d}x -c_4|\Omega|.
\end{align*}
Set
$$
\Omega_1=\{x\in \Omega: 0< \rho^{n, \varepsilon}_\omega(x)<4\}\quad \text{and}\quad \Omega_2=\Omega\backslash \Omega_1.
$$
Then we find
\begin{align*}
\ii{\Omega}{}{R_\rho(\rho^{n, \varepsilon}_\omega) }{x}
& \geq - C_R\big (|\Omega|+\|\rho^{n, \varepsilon}_\omega\|^2\big)\\
& = - C_R|\Omega| - C_R\int_{\Omega_1} |\rho^{n, \varepsilon}_\omega|^2\,\mathrm{d}x - C_R\int_{\Omega_2} |\rho^{n, \varepsilon}_\omega|^2\,\mathrm{d}x\\
& \geq -17C_R|\Omega| - C_R\int_{\Omega_2} |\rho^{n, \varepsilon}_\omega|^2\,\mathrm{d}x.
\end{align*}
Next, using the Cauchy--Schwarz inequality, H\"{o}lder's inequality and Young's inequality, we infer that
\begin{align*}
&-\dfrac{\theta}{2} \int_\Omega \rho^{n, \varepsilon}_\omega |\nabla\phi^{n, \varepsilon}_\omega|^2\, \mathrm{d}x \\
&\quad  = -\dfrac{\theta}{2} \int_{\Omega_1} \rho^{n, \varepsilon}_\omega |\nabla\phi^{n, \varepsilon}_\omega|^2\, \mathrm{d}x
-\dfrac{\theta}{2} \int_{\Omega_2} \rho^{n, \varepsilon}_\omega |\nabla\phi^{n, \varepsilon}_\omega|^2\, \mathrm{d}x\\
&\quad  \geq  -2\theta\|\nabla\phi^{n, \varepsilon}_\omega\|^2
- \dfrac{\omega}{8} \int_{\Omega_2}  |\nabla\phi^{n, \varepsilon}_\omega|^4 \,\mathrm{d}x
- \dfrac{\theta^2}{2\omega}\int_{\Omega_2}  |\rho^{n, \varepsilon}_\omega|^2\, \mathrm{d}x\\
&\quad  = 2\theta\int_\Omega \phi^{n, \varepsilon}_\omega \Delta\phi^{n, \varepsilon}_\omega\mathrm{d}{x}
- \dfrac{\omega}{8} \int_{\Omega_2}  |\nabla\phi^{n, \varepsilon}_\omega|^4 \,\mathrm{d}x
-  \dfrac{\theta^2}{2\omega} \int_{\Omega_2} |\rho^{n, \varepsilon}_\omega|^2\, \mathrm{d}x\\
&\quad  \geq -\frac{\alpha}{4}\|\Delta \phi^{n, \varepsilon}_\omega\|^2
- \frac{4\theta^2}{\alpha} \|\phi^{n, \varepsilon}_\omega\|^2
- \dfrac{\omega}{8} \int_{\Omega}   |\nabla\phi^{n, \varepsilon}_\omega|^4\, \mathrm{d}x - \dfrac{\theta^2}{2\omega} \int_{\Omega_2} |\rho^{n, \varepsilon}_\omega|^2\, \mathrm{d}x\\
&\quad  \geq  -\frac{\alpha}{4}\|\Delta \phi^{n, \varepsilon}_\omega\|^2
- \frac{c_3}{2}\int_\Omega |\phi^{n, \varepsilon}_\omega|^4\,\mathrm{d}x
- \frac{8\theta^4}{c_3\alpha^2} |\Omega|
- \dfrac{\omega}{8} \int_{\Omega}  |\nabla\phi^{n, \varepsilon}_\omega|^4\, \mathrm{d}x
- \dfrac{\theta^2}{2\omega} \int_{\Omega_2} |\rho^{n, \varepsilon}_\omega|^2\, \mathrm{d}x.
\end{align*}
Therefore, from the above estimates we deduce that
\begin{align*}
\ene(\mathbf{u}^{n, \varepsilon}_\omega, \phi^{n, \varepsilon}_\omega, \rho^{n, \varepsilon}_\omega)
&\geq  \dfrac{1}{2}\|\uu_{\omega}^{n,\varepsilon}\|^2
+ \dfrac{\alpha}{4} \|\Delta \phi^{n, \varepsilon}_\omega\|^2
+ \dfrac{1}{2} \|\nabla \phi^{n, \varepsilon}_\omega\|^2
+ \dfrac{\beta}{2}\|\nabla \rho^{n, \varepsilon}_\omega\|^2
+ \frac{c_3}{2} \int_\Omega |\phi^{n, \varepsilon}_\omega|^4\,\mathrm{d}x
\\
&\quad + \int_{\Omega} \dfrac{\omega}{8} |\nabla\phi^{n, \varepsilon}_\omega|^4 \, \mathrm{d}x
- \left(c_4+17C_R+\frac{8\theta^4}{c_3\alpha^2}\right)|\Omega|
- \left(C_R+\dfrac{\theta^2}{2\omega}\right)\int_{\Omega_2} |\rho^{n, \varepsilon}_\omega|^2\,\mathrm{d}x \\
&\quad +  \ii{\Omega}{}{ \widehat{S}_{\rho, \varepsilon}(\rho^{n, \varepsilon}_\omega)}{x}.
\end{align*}
Recalling now \eqref{eq:taylor} and \textbf{(H3)}, we see that $\widehat{S}_\rho'(\varepsilon)<0$,  $\widehat{S}_\rho'(1-\varepsilon)>0$, $\widehat{S}_\rho''(\varepsilon)>0$,  $\widehat{S}_\rho''(1-\varepsilon)>0$, when $\varepsilon\in (0,\epsilon_2)$, for some sufficiently small $\epsilon_2\in (0,1)$. Then we have
 $$
 \widehat{S}_{\rho, \varepsilon}(s)\geq \widehat{S}_{\rho, \varepsilon}(\varepsilon) +\dfrac{\widehat{S}_\rho''(\varepsilon)}{2}(s-\varepsilon)^2,\quad \forall\, s\leq \varepsilon,
 $$
 which implies
 $$
 \widehat{S}_{\rho, \varepsilon}(s)\geq -\gamma_1 +\dfrac{\widehat{S}_\rho''(\varepsilon)}{2}s^2,\quad \forall\, s\leq 0.
 $$
 On the other hand, we have
 $$
 \widehat{S}_{\rho, \varepsilon}(s)\geq \widehat{S}_{\rho, \varepsilon}(1-\varepsilon) +  \dfrac{\widehat{S}_\rho''(1-\varepsilon)}{2}[s-(1-\varepsilon)]^2,\quad \forall\, s\geq 1-\varepsilon,
 $$
 so that
 $$
 \widehat{S}_{\rho, \varepsilon}(s)\geq -\gamma_1 +  \dfrac{\widehat{S}_\rho''(1-\varepsilon)}{2}(s-1)^2,\quad \forall\, s\geq 1.
 $$
 Observing that
 $$
 (s-1)^2=\frac12 s^2+\frac12(s-2)^2-1\geq \frac12 s^2,\quad \forall\, s\geq 2+\sqrt{2},
 $$
 we then obtain
 $$
 \widehat{S}_{\rho, \varepsilon}(s)\geq -\gamma_1 +  \dfrac{\widehat{S}_\rho''(1-\varepsilon)}{4}s^2,\quad \forall\, s\geq 4.
 $$
 From the above observations and \textbf{(H3)}, we conclude that there exists $k_1>0$ depending on $\varepsilon$ such that for any $\varepsilon \in (0, \epsilon_2)$
\begin{align*}
\displaystyle\ii{\Omega}{}{\widehat{S}_{\rho, \varepsilon}(\rho^{n, \varepsilon}_\omega)}{x}
& =	\displaystyle\ii{\Omega_1}{}{\widehat{S}_{\rho, \varepsilon}(\rho^{n, \varepsilon}_\omega)}{x}
    + \displaystyle\ii{\Omega_2}{}{\widehat{S}_{\rho, \varepsilon}(\rho^{n, \varepsilon}_\omega)}{x}\\
&\geq  - \gamma_1|\Omega| + k_1 \int_{\Omega_2} |\rho^{n, \varepsilon}_\omega |^{2}\mathrm{d}x.
\end{align*}
The constant $k_1=k_1(\varepsilon)$ can be taken arbitrarily large, as long as $\epsilon_2$ is sufficiently small. Thus, by choosing $0<\epsilon_2 = \epsilon_2(\omega)<<1$ such that
$$
k_1-\left(C_R+\dfrac{\theta^2}{2\omega}\right)\geq 0,\quad \varepsilon\in (0,\epsilon_2),
$$
we infer that
\begin{align*}
\ene(\mathbf{u}^{n, \varepsilon}_\omega, \phi^{n, \varepsilon}_\omega, \rho^{n, \varepsilon}_\omega)
&\geq  \dfrac{1}{2}\|\uu_{\omega}^{n,\varepsilon}\|^2
+ \dfrac{\alpha}{4} \|\Delta \phi^{n, \varepsilon}_\omega\|^2
+ \dfrac{1}{2} \|\nabla \phi^{n, \varepsilon}_\omega\|^2
+ \dfrac{\beta}{2}\|\nabla \rho^{n, \varepsilon}_\omega\|^2
+ \frac{c_3}{2} \int_\Omega |\phi^{n, \varepsilon}_\omega|^4\,\mathrm{d}x
\\
&\quad + \int_{\Omega} \dfrac{\omega}{8} |\nabla\phi^{n, \varepsilon}_\omega|^4\, \mathrm{d}x
- \left(c_4+17C_R+\frac{8\theta^4}{c_3\alpha^2}+\gamma_1\right)|\Omega|\\
&\quad + \left[k_1-\left(C_R+\dfrac{\theta^2}{2\omega}\right)\right]\int_{\Omega_2} |\rho^{n, \varepsilon}_\omega|^2\,\mathrm{d}x\\
&\geq \dfrac{1}{2}\|\uu_{\omega}^{n,\varepsilon}\|^2
+ \dfrac{\alpha}{4} \|\Delta \phi^{n, \varepsilon}_\omega\|^2
+ \dfrac{1}{2} \|\nabla \phi^{n, \varepsilon}_\omega\|^2
+ \dfrac{\beta}{2}\|\nabla \rho^{n, \varepsilon}_\omega\|^2
+ \frac{c_3}{2} \int_\Omega |\phi^{n, \varepsilon}_\omega|^4\,\mathrm{d}x
\\
&\quad +\dfrac{\omega}{8}  \int_{\Omega} |\nabla\phi^{n, \varepsilon}_\omega|^4\, \mathrm{d}x - C_2,
\end{align*}
for all $\varepsilon \in (0, \min(\epsilon_1, \epsilon_2))$,
where
$C_2=\big(c_4+17C_R+\frac{8\theta^4}{c_3\alpha^2}+\gamma_1\big)|\Omega|$
is independent of $n$, $\varepsilon$, and $\omega$.

The proof is complete.
\end{proof}

We can now obtain some uniform estimates for the approximate solutions $(\uu_\omega^{n,\varepsilon},\phi_\omega^{n,\varepsilon}, \rho_\omega^{n,\varepsilon}, \mu_\omega^{n,\varepsilon}, \psi_\omega^{n,\varepsilon})$.
\begin{lemma} \label{prop:ub1}
The sequence $\{\uu^{n, \varepsilon}_\omega\}$ is uniformly bounded in $L^\infty(0,T_n;\mathbf{H}_\sigma) \cap L^2(0,T_n; \mathbf{V}_\sigma)$.
The sequence $\{\phi^{n, \varepsilon}_\omega\}$ is uniformly bounded in $L^\infty(0,T_n;H^2(\Omega))$.
The sequence $\{\rho^{n, \varepsilon}_\omega\}$ is uniformly bounded in $L^\infty(0,T_n;H^1(\Omega))$. The bounds are independent of $\omega$ and $n$, but may depend on $\varepsilon$.
\end{lemma}
\begin{proof}
It follows from Lemma \ref{prop:enebound1} that
\begin{align}
 C_3\left(\| \mathbf{u}^{n, \varepsilon}_\omega(t) \|^2 + \| \Delta \phi^{n, \varepsilon}_\omega(t)\|^2 + \| \nabla \phi^{n, \varepsilon}_\omega(t) \|^2  + \|\nabla \rho^{n, \varepsilon}_\omega(t)\|^2 \right)
\leq \ene(\uu^{n, \varepsilon}_\omega(t), \phi^{n, \varepsilon}_\omega(t), \rho^{n, \varepsilon}_\omega(t)) + C_2.
\label{eq:aux1}
\end{align}
where $C_3=\frac{1}{8}\min\{1,\alpha,\beta,c_3\}>0$. Besides, we note that the averages of $\phi^{n, \varepsilon}_\omega$ and $\rho^{n, \varepsilon}_\omega$ are independent of $n$ and $t$ (by orthogonality of the eigenvectors), that is,
$$
\int_\Omega \phi^{n, \varepsilon}_\omega(x,t)\,\mathrm{d} x
= \int_\Omega \phi_0^n(x) \,\mathrm{d} x
= \int_\Omega \phi_0(x) \,\mathrm{d} x,\quad
\int_\Omega \rho^{n, \varepsilon}_\omega(x,t)\,\mathrm{d} x
= \int_\Omega \rho_0^n(x) \,\mathrm{d} x
= \int_\Omega \rho_0(x) \,\mathrm{d} x, \quad \forall\, t\in [0,T_n].
$$
Therefore, by the triangle inequality and the Poincar\'{e}--Wirtinger inequality, we have
    \[
    \|\phi^{n, \varepsilon}_\omega\|^2 \leq 2\|\phi^{n, \varepsilon}_\omega - \overline{\phi^{n, \varepsilon}_\omega}\|^2 + 2\| \overline{\phi^{n, \varepsilon}_\omega} \|^2 \leq C \| \nabla \phi^{n, \varepsilon}_\omega \|^2
      + C |\overline{\phi^{n, \varepsilon}_\omega}|^2
    =  C\left( \| \nabla \phi^{n, \varepsilon}_\omega \|^2
      + |\overline{\phi_0}|^2\right),
    \]
		since $\phi^{n, \varepsilon}_\omega \in H^1(\Omega)$. Thus, a uniform bound of $\phi^{n, \varepsilon}_\omega$ in $L^\infty(0,T_n;H^1(\Omega))$ is obtained. The  $L^\infty(0,T_n;H^2(\Omega))$-bound then follows from the standard elliptic regularity theory.
		A similar argument yields a bound for $\rho^{n, \varepsilon}_\omega$ in $L^\infty(0,T_n;H^1(\Omega))$. Concerning the velocity field $\uu^{n, \varepsilon}_\omega$, it follows from \eqref{eq:aux1}, Lemma \ref{prop:enebound1} and Korn's inequality that
		\[
		\ii{0}{t}{\|\sqrt{\nu(\phi^{n, \varepsilon}_\omega(\tau),\rho^{n, \varepsilon}_\omega(t))}D \mathbf{u}^{n, \varepsilon}_\omega(t)\|^2}{\tau} \geq \ii{0}{t}{\dfrac{\nu_*}{2}\| \nabla \uu^{n, \varepsilon}_\omega(\tau)\|^2}{\tau},
\quad \forall\, t\in (0,T_n].
		\]
The proof is complete.
	\end{proof}

We now prove a priori bounds for the chemical potentials.
\begin{lemma} \label{prop:ub2}
The sequences $\{\mu^{n, \varepsilon}_\omega\}$ and $\{\psi^{n, \varepsilon}_\omega\}$  are uniformly bounded in $L^2(0,T_n;H^1(\Omega))$. The bounds are independent of $\omega$, $n$, but may depend on $\varepsilon$.
\end{lemma}
\begin{proof}
From Lemma \ref{prop:enebound1} we infer that
$$
\int_0^{t} \left(\| \nabla \mu^{n, \varepsilon}_\omega(\tau)\|^2 + \| \nabla \psi^{n, \varepsilon}_\omega(\tau)\|^2\right)\,\mathrm{d}\tau \leq C_1+C_2,\quad \forall\, t\in (0,T_n].
$$
Let us first consider the estimate for $\{\mu^{n, \varepsilon}_\omega\}$. Since $\mu^{n, \varepsilon}_\omega \in H^1(\Omega)$, then arguing as in the proof of Lemma \ref{prop:ub1}, we obtain
    \[
\|\mu^{n, \varepsilon}_\omega\|^2 \leq 2\|\mu^{n, \varepsilon}_\omega - \overline{\mu^{n, \varepsilon}_\omega}\|^2 + 2\| \overline{\mu^{n, \varepsilon}_\omega} \|^2 \leq C \left( \| \nabla \mu^{n, \varepsilon}_\omega \|^2 + |\overline{\mu^{n, \varepsilon}_\omega}|^2 \right).
    \]
Then it remains to provide a uniform control of $|\overline{\mu^{n, \varepsilon}_\omega}|$. From the third equation in \eqref{eq:weakrega}, we see that
		\[
		\begin{split}
			|\overline{\mu^{n, \varepsilon}_\omega}| & = \left| \overline{\Pi_n \left(  \alpha \Delta^2 \phi^{n, \varepsilon}_\omega-\Delta \phi^{n, \varepsilon}_\omega  + S'_{\phi}(\phi^{n, \varepsilon}_\omega) + \theta\nabla \cdot (\rho^{n, \varepsilon}_\omega \nabla \phi^{n, \varepsilon}_\omega) -  \textcolor{black}{\omega \nabla \cdot\left( |\nabla \phi^{n, \varepsilon}_\omega|^2\nabla \phi^{n, \varepsilon}_\omega \right)} \right)} \right| \\
			& = \dfrac{1}{|\Omega|}\left| \left( 1, \Pi_n \left(  \alpha \Delta^2 \phi^{n, \varepsilon}_\omega-\Delta \phi^{n, \varepsilon}_\omega  + S'_{\phi}(\phi^{n, \varepsilon}_\omega) + \theta\nabla \cdot (\rho^{n, \varepsilon}_\omega \nabla \phi^{n, \varepsilon}_\omega) - \textcolor{black}{ \omega \nabla \cdot( |\nabla \phi^{n, \varepsilon}_\omega|^2\nabla \phi^{n, \varepsilon}_\omega ) } \right) \right) \right|\\
			& = \dfrac{1}{|\Omega|} \left| \left( 1,  \alpha \Delta^2 \phi^{n, \varepsilon}_\omega - \Delta \phi^{n, \varepsilon}_\omega  + S'_{\phi}(\phi^{n, \varepsilon}_\omega) + \theta\nabla \cdot (\rho^{n, \varepsilon}_\omega \nabla \phi^{n, \varepsilon}_\omega) - \textcolor{black}{\omega \nabla \cdot ( |\nabla \phi^{n, \varepsilon}_\omega|^2\nabla \phi^{n, \varepsilon}_\omega ) } \right) \right|.
		\end{split}
		\]
Thanks to the homogeneous Neumann boundary conditions $\partial_\mathbf{n}\phi^{n, \varepsilon}_\omega=\partial_\mathbf{n} \Delta \phi^{n, \varepsilon}_\omega=0$ on $\partial \Omega$, using integration by parts, we find
		\[
		|\overline{\mu^{n, \varepsilon}_\omega}|
       = \dfrac{1}{|\Omega|}\left| \big( 1, {S}'_{\phi}(\phi^{n, \varepsilon}_\omega) \big)\right|.
		\]
From Lemma \ref{prop:ub1} and the embedding $H^2(\Omega) \hookrightarrow L^\infty(\Omega)$ ($d=2,3$) we obtain that $\phi^{n, \varepsilon}_\omega \in L^\infty(0,T_n;L^\infty(\Omega))$. Then by \textbf{(H2)} we get
		\[
		\sup_{t\in[0,T_n]}|\overline{\mu^{n, \varepsilon}_\omega}(t)| \leq C,
		\]
where $C$ is independent of $\omega$ and $n$, but it may depend on $\varepsilon$. As a consequence, we have
	\[
      \int_0^t \|\mu^{n, \varepsilon}_\omega(\tau)\|_{H^1(\Omega)}^2 \, \mathrm{d}t
      \leq C \int_0^t \|\nabla \mu^{n, \varepsilon}_\omega(\tau)\|^2\, \mathrm{d}t
         + C\int_0^t |\overline{\mu^{n, \varepsilon}_\omega(\tau)}|^2 \ \mathrm{d}t\leq C(1+T_n)\leq C(1+T),
    \]
for all $t\in (0,T_n]$.
		
In a similar manner, for $\psi^{n, \varepsilon}_\omega$ we have
\begin{align*}
		|\overline{\psi^{n, \varepsilon}_\omega}|
& = \dfrac{1}{|\Omega|}\left| \left( 1, \widehat{S}'_{\rho, \varepsilon}(\rho^{n, \varepsilon}_\omega)+ R'_{\rho}(\rho^{n, \varepsilon}_\omega) - \dfrac{\theta}{2}|\nabla \phi^{n, \varepsilon}_\omega|^2 \right)\right|\\
& \leq  \dfrac{1}{|\Omega|}\| {S}'_{\rho, \varepsilon}(\rho^{n, \varepsilon}_\omega) \|_{L^1(\Omega)} + \dfrac{1}{|\Omega|}\| R'_{\rho}(\rho^{n, \varepsilon}_\omega) \|_{L^1(\Omega)}+ \dfrac{\theta}{2|\Omega|} \| \nabla \phi^{n, \varepsilon}_\omega\|^2\\
& \leq \dfrac{1}{|\Omega|}\| {S}'_{\rho, \varepsilon}(\rho^{n, \varepsilon}_\omega) \|_{L^1(\Omega)} +C \big(1+ \|\phi^{n, \varepsilon}_\omega\|^2_{H^1(\Omega)}\big).
\end{align*}
We recall a well-known result for the approximating potential $\widehat{S}_{\rho, \varepsilon}$ (see e.g., \cite{KNP95,MZ04}), namely, there exists $C > 0$ independent of $\varepsilon$, such that
\begin{equation}
\label{eq:uniformhatS}
\| S'_{\rho, \varepsilon} (\rho^{n, \varepsilon}_\omega)\|_{L^1(\Omega)} \leq C \ii{\Omega}{}{(\rho^{n, \varepsilon}_\omega - \overline{\rho_0^n})\big(S'_{\rho, \varepsilon}(\rho^{n, \varepsilon}_\omega) - \overline{S'_{\rho, \varepsilon}(\rho^{n, \varepsilon}_\omega)}\big)}{x} + C.
\end{equation}
Then, choosing $\rho^{n, \varepsilon}_\omega - \overline{\rho_0^n}$ as a test function in the equation for $\psi^{n, \varepsilon}_\omega$ in \eqref{eq:weakrega}, we get
		\[
		\beta\| \nabla \rho^{n, \varepsilon}_\omega \|^2 + \big(S'_{\rho, \varepsilon}(\rho^{n, \varepsilon}_\omega), \rho^{n, \varepsilon}_\omega - \overline{\rho_0^n}\big)
= (\psi^{n, \varepsilon}_\omega - \overline{\psi^{n, \varepsilon}_\omega}, \rho^{n, \varepsilon}_\omega - \overline{\rho_0^n}) + \dfrac{\theta}{2}(|\nabla \phi^{n, \varepsilon}_\omega|^2, \rho^{n, \varepsilon}_\omega - \overline{\rho_0^n}).
		\]
It follows from \eqref{eq:uniformhatS}, using the Poincar\'{e}--Wirtinger and the Cauchy--Schwarz inequalities, that
		\[
		\| S'_{\rho, \varepsilon} (\rho^{n, \varepsilon}_\omega)\|_{L^1(\Omega)} \leq C\big( \| \nabla \psi^{n, \varepsilon}_\omega\|\|\nabla \rho^{n, \varepsilon}_\omega \| + \| |\nabla \phi^{n, \varepsilon}_\omega|^2 \|\|\nabla \rho^{n, \varepsilon}_\omega\| + 1\big)
\leq C\big(\| \nabla \psi^{n, \varepsilon}_\omega\| + 1\big),
		\]
		where we recall that $\||\nabla \phi^{n, \varepsilon}_\omega|^2 \| \leq C\| \nabla \phi^{n, \varepsilon}_\omega \|^2_{\mathbf{L}^4(\Omega)} \leq C$ thanks to Lemma  \ref{prop:ub1}. Hence, we deduce that
\begin{equation} \label{eq:potentialest}
\|\psi^{n, \varepsilon}_\omega\|_{H^1(\Omega)}
\leq C\big( \|\nabla \psi^{n, \varepsilon}_\omega\|+|\overline{\psi^{n, \varepsilon}_\omega(\tau)}|\big)
\leq C\big( \|\nabla \psi^{n, \varepsilon}_\omega\|+1\big),
\end{equation}
which implies
  \[
  \int_0^t \|\psi^{n, \varepsilon}_\omega(\tau)\|_{H^1(\Omega)}^2 \, \mathrm{d}t
     \leq C \int_0^t \|\nabla \psi^{n, \varepsilon}_\omega(\tau)\|^2\, \mathrm{d}t
         + C\int_0^t |\overline{\psi^{n, \varepsilon}_\omega(\tau)}|^2 \ \mathrm{d}t\leq C(1+T_n)\leq C(1+T),
  \]
for all $t\in (0,T_n]$.

The proof is complete.
\end{proof}
The estimates obtained in Lemma \ref{prop:ub1} and Lemma \ref{prop:ub2} allow us to extend the local solution $\{a_i, b_i, c_i, d_i, e_i\}_{i = 1}^n$ to the
full time interval $[0, T]$. As a consequence, the approximate solution $(\uu_\omega^{n,\varepsilon},\phi_\omega^{n,\varepsilon}, \rho_\omega^{n,\varepsilon}, \mu_\omega^{n,\varepsilon}, \psi_\omega^{n,\varepsilon})$ is well defined in $[0,T]$.

We now need to derive some uniform estimates on the time derivatives of the approximate solutions in order to apply a compactness argument.
\begin{lemma}	\label{prop:ub3}
The sequence $\{\partial_t \uu^{n, \varepsilon}_\omega\}$ is uniformly bounded in $L^\frac{4}{d}(0,T; \mathbf{V}_\sigma^*)$, $d=2,3$. The sequences $\{\partial_t \phi^{n, \varepsilon}_\omega\}$, $\{\partial_t \rho^{n, \varepsilon}_\omega\}$ are uniformly bounded in $L^2(0,T; (H^1(\Omega))^*)$. The bounds are independent of $\omega$, $n$, but may depend on $T$ and $\varepsilon$.
\end{lemma}
\begin{proof}
Consider the first equation in \eqref{eq:weakrega} and let $\mathbf{w} \in \mathbf{V}_\sigma$ be such that $\| \mathbf{w}\|_{\mathbf{V}_\sigma} = 1$. On account of Lemmas \ref{prop:ub1} and \ref{prop:ub2}, we find that
\begin{align}
\left|\left\langle \partial_t\uu^{n, \varepsilon}_\omega, \mathbf{w} \right\rangle_{\mathbf{V}^*,\mathbf{V}}\right|
& \leq |(\uu^{n, \varepsilon}_\omega \cdot \nabla)\uu^{n, \varepsilon}_\omega, \mathbf{w})| + |(\nu(\phi^{n, \varepsilon}_\omega, \rho^{n, \varepsilon}_\omega)D\uu^{n, \varepsilon}_\omega, D\mathbf{w})| + |(\mu^{n, \varepsilon}_\omega\nabla\phi^{n, \varepsilon}_\omega, \mathbf{w})| \notag \\
&\qquad + |(\psi^{n, \varepsilon}_\omega\nabla\rho^{n, \varepsilon}_\omega, \mathbf{w})| \notag\\
& \leq |(\uu^{n, \varepsilon}_\omega \cdot \nabla)\mathbf{w}, \uu^{n, \varepsilon}_\omega)| + \nu^*\|D\uu^{n, \varepsilon}_\omega\|\|D\mathbf{w}\| + \|\mu^{n, \varepsilon}_\omega\|\|\nabla \phi^{n, \varepsilon}_\omega\|_{\LL^4(\Omega)}
\|\mathbf{w}\|_{\LL^4(\Omega)} \notag\\
&\qquad + \|\psi^{n, \varepsilon}_\omega\|_{L^4(\Omega)}\|\nabla \rho^{n, \varepsilon}_\omega\|\|\mathbf{w}\|_{\LL^4(\Omega)} \notag\\
& \leq \|\uu^{n, \varepsilon}_\omega\|_{\LL^4(\Omega)}^2 + C\left( \|\uu^{n, \varepsilon}_\omega\|_{\mathbf{V}_\sigma} + \|\nabla \mu^{n, \varepsilon}_\omega\| + \|\nabla\psi^{n, \varepsilon}_\omega\| +1 \right) \notag\\
& \leq \|\uu^{n, \varepsilon}_\omega\|_{\mathbf{V}_\sigma}^\frac{d}{2}\|\uu^{n, \varepsilon}_\omega\|^{2-\frac{d}{2}}
+ C\left(\|\uu^{n, \varepsilon}_\omega\|_{\mathbf{V}_\sigma}
+ \|\nabla \mu^{n, \varepsilon}_\omega\| + \|\nabla\psi^{n, \varepsilon}_\omega\|  +1 \right)\notag\\
& \leq C\big(\|\uu^{n, \varepsilon}_\omega\|_{\mathbf{V}_\sigma}^\frac{d}{2} +  \|\uu^{n, \varepsilon}_\omega\|_{\mathbf{V}_\sigma} + \|\nabla \mu^{n, \varepsilon}_\omega\| + \|\nabla\psi^{n, \varepsilon}_\omega\| +1\big).
\label{eq:timederivativesU}
\end{align}
From \eqref{eq:timederivativesU}, taking the supremum over all functions $\mathbf{w}$, raising the inequality to the $\frac{4}{d}$-th power and integrating on $[0,T]$, we obtain the desired bound.

We now perform a comparison argument in the second and fourth equations in \eqref{eq:weakrega}.
Let $w \in H^1(\Omega)$ such that $\|w\|_{H^1(\Omega)} = 1$ be given. Using Lemmas \ref{prop:ub1} and \ref{prop:ub2}, we get
\begin{equation} \label{eq:timederivatives}
\begin{cases}
|\left\langle \partial_t\phi^{n, \varepsilon}_\omega, w \right\rangle_{(H^1(\Omega))^*,H^1(\Omega)}| \leq C\left( \| \nabla \mu^{n, \varepsilon}_\omega \| + \|\nabla \uu^{n, \varepsilon}_\omega\|\right),	\\
|\left\langle \partial_t\rho^{n, \varepsilon}_\omega, w \right\rangle_{(H^1(\Omega))^*,H^1(\Omega)}| \leq C\left( \| \nabla \psi^{n, \varepsilon}_\omega \| + \|\nabla \uu^{n, \varepsilon}_\omega\|\right).	
\end{cases}	
\end{equation}
where we have used the following estimate
$$
|(\uu^{n, \varepsilon}_\omega\cdot \nabla \phi^{n, \varepsilon}_\omega, w)|
=|(\uu^{n, \varepsilon}_\omega\cdot \nabla w, \phi^{n, \varepsilon}_\omega)|
\leq C\|\uu^{n, \varepsilon}_\omega\|_{\LL^4(\Omega)}\|\phi^{n, \varepsilon}_\omega\|_{L^4(\Omega)}\leq C\|\nabla \uu^{n, \varepsilon}_\omega\|\|\phi^{n, \varepsilon}_\omega\|_{H^1(\Omega)},
$$
and a similar one for $\rho^{n, \varepsilon}_\omega$.
Taking the supremum over all functions $w$, squaring the inequality and integrating in $[0,T]$, we arrive at the desired bound thanks to Lemma \ref{prop:ub1}.

The proof is complete. 		
\end{proof}
Finally, some additional estimates for $\phi^{n, \varepsilon}_\omega$ and $\rho^{n, \varepsilon}_\omega$ can also be deduced, that is,
\begin{lemma}	\label{prop:ub4}
The sequence $\{\phi^{n, \varepsilon}_\omega\}$ is uniformly bounded in $L^2(0,T;H^4(\Omega))$. The sequence $\{\rho^{n, \varepsilon}_\omega\}$ is uniformly bounded in $L^4(0,T;H^2(\Omega))$. The bounds are independent of $\omega$, $n$, but may depend on $T$ and $\varepsilon$.
\end{lemma}
\begin{proof}
Let us first consider $\rho^{n, \varepsilon}_\omega$. Multiplying the equation for $\psi^{n, \varepsilon}_\omega$ by $-\Delta \rho^{n, \varepsilon}_\omega$ and integrating over $\Omega$, we find
   \[
		(\psi^{n, \varepsilon}_\omega, -\Delta \rho^{n, \varepsilon}_\omega) = \beta\|\Delta \rho^{n, \varepsilon}_\omega\|^2
+ \big(S''_{\rho,\varepsilon}(\rho^{n, \varepsilon}_\omega)\nabla \rho^{n, \varepsilon}_\omega, \nabla \rho^{n, \varepsilon}_\omega\big)
+ \dfrac{\theta}{2}(|\nabla \phi^{n, \varepsilon}_\omega|^2, \Delta \rho^{n, \varepsilon}_\omega).
	\]
Recalling \eqref{eq:convexityapprox} and using H\"{o}lder's as well as Young's inequalities, we get
\begin{align*}
\beta\|\Delta \rho^{n, \varepsilon}_\omega\|^2
 &\leq \|\nabla \psi^{n, \varepsilon}_\omega \| \|\nabla \rho^{n, \varepsilon}_\omega\| + \gamma_2 \|\nabla \rho^{n, \varepsilon}_\omega\|^2
 + C\|\nabla \phi^{n, \varepsilon}_\omega\|^2_{\LL^4(\Omega)}\|\Delta \rho^{n, \varepsilon}_\omega\|\\
 &\leq C \big(\|\nabla \psi^{n, \varepsilon}_\omega \|+1\big)
      + C \|\nabla \phi^{n, \varepsilon}_\omega\|^\frac32_{\mathbf{H}^1(\Omega)}\|\nabla \phi^{n, \varepsilon}_\omega\|^\frac12_{\mathbf{L}^2(\Omega)} \|\Delta \rho^{n, \varepsilon}_\omega\|\\
      &\leq C \big(\|\nabla \psi^{n, \varepsilon}_\omega \|+1\big) + C \big(\textcolor{black}{\|\Delta \phi^{n, \varepsilon}_\omega\|^\frac32}+1 \big)\|\Delta \rho^{n, \varepsilon}_\omega\|\\
      &\leq \frac{\beta}{2} \|\Delta \rho^{n, \varepsilon}_\omega\|^2 +C (\|\nabla \psi^{n, \varepsilon}_\omega \|+1),
\end{align*}
which implies
\begin{align}
\beta\|\Delta \rho^{n, \varepsilon}_\omega\|^2 \leq C(\|\nabla \psi^{n, \varepsilon}_\omega \|+1).
\label{rhoL4h2}
\end{align}
From the above estimate, the elliptic regularity theory and Lemma \ref{prop:ub2}, we get $\rho^{n, \varepsilon}_\omega\in L^4(0,T;H^2(\Omega))$.

Concerning $\phi^{n, \varepsilon}_\omega$, we multiply the equation for $\mu^{n, \varepsilon}_\omega$ by $\Delta^2\phi^{n, \varepsilon}_\omega$, integrating over $\Omega$. This yields
\begin{align*}
(\mu^{n, \varepsilon}_\omega, \Delta^2 \phi^{n, \varepsilon}_\omega)
&= \alpha\| \Delta^2 \phi^{n, \varepsilon}_\omega\|^2 + \| \nabla \Delta \phi^{n, \varepsilon}_\omega \|^2
- (S''_\phi(\phi^{n, \varepsilon}_\omega)\nabla\phi^{n, \varepsilon}_\omega, \nabla \Delta \phi^{n, \varepsilon}_\omega)
+ \theta(\rho^{n, \varepsilon}_\omega\Delta\phi^{n, \varepsilon}_\omega, \Delta^2\phi^{n, \varepsilon}_\omega) \\
&\quad + \theta (\nabla \rho^{n, \varepsilon}_\omega \cdot \nabla \phi^{n, \varepsilon}_\omega, \Delta^2\phi^{n, \varepsilon}_\omega)
- \textcolor{black}{\omega(\nabla \cdot\left( |\nabla \phi^{n, \varepsilon}_\omega|^2\nabla \phi^{n, \varepsilon}_\omega \right), \Delta^2\phi^{n, \varepsilon}_\omega).}	
\end{align*}
Exploiting the fact that $\phi^{n, \varepsilon}_\omega \in L^\infty(0,T;L^\infty(\Omega))$ and H\"{o}lder's inequality,
we infer that
\begin{align*}
& \alpha\| \Delta^2 \phi^{n, \varepsilon}_\omega\|^2
+\| \nabla \Delta \phi^{n, \varepsilon}_\omega \|^2\\
&\quad  \leq |(\nabla \mu^{n, \varepsilon}_\omega, \nabla \Delta \phi^{n, \varepsilon}_\omega)|
 + \|S''_\phi(\phi^{n, \varepsilon}_\omega)\|_{L^\infty(\Omega)}\|\nabla\phi^{n, \varepsilon}_\omega\|\| \nabla \Delta \phi^{n, \varepsilon}_\omega\| \\
 &\qquad + C \|\rho^{n, \varepsilon}_\omega\|_{L^\infty(\Omega)} \|\Delta \phi^{n, \varepsilon}_\omega\|\|\Delta^2\phi^{n, \varepsilon}_\omega\|
 + C\|\nabla \rho^{n, \varepsilon}_\omega\|_{\LL^3(\Omega)}\|\nabla \phi^{n, \varepsilon}_\omega\|_{\LL^6(\Omega)}\|\Delta^2 \phi^{n, \varepsilon}_\omega\|
 \\
 &\qquad + C \|\nabla \cdot\left( |\nabla \phi^{n, \varepsilon}_\omega|^2\nabla \phi^{n, \varepsilon}_\omega \right) \| \|\Delta^2 \phi^{n, \varepsilon}_\omega\|.
\end{align*}
Exploiting the Cauchy--Schwarz and Young inequalities, as well as the embeddings $H^1(\Omega) \hookrightarrow L^6(\Omega)$ and $H^2(\Omega) \hookrightarrow L^\infty(\Omega)$ ($d=2,3$), we have
\begin{align*}
& |(\nabla \mu^{n, \varepsilon}_\omega, \nabla \Delta \phi^{n, \varepsilon}_\omega)|
+ \|S''_\phi(\phi^{n, \varepsilon}_\omega)\|_{L^\infty(\Omega)}\|\nabla\phi^{n, \varepsilon}_\omega\|\| \nabla \Delta \phi^{n, \varepsilon}_\omega\| \\
 &\qquad + C \|\rho^{n, \varepsilon}_\omega\|_{L^\infty(\Omega)} \|\Delta \phi^{n, \varepsilon}_\omega\|\|\Delta^2\phi^{n, \varepsilon}_\omega\|
 + C\|\nabla \rho^{n, \varepsilon}_\omega\|_{\LL^3(\Omega)}\|\nabla \phi^{n, \varepsilon}_\omega\|_{\LL^6(\Omega)}\|\Delta^2 \phi^{n, \varepsilon}_\omega\|
 \\
& \leq \dfrac{\alpha}{4}\|\Delta^2\phi^{n, \varepsilon}_\omega\|^2 +\frac12 \|\nabla \Delta\phi^{n, \varepsilon}_\omega\|^2
+ C\big( 1+ \| \nabla \mu^{n, \varepsilon}_\omega\|^2+  \|\rho^{n, \varepsilon}_\omega\|_{L^\infty(\Omega)}^2 + \|\nabla \rho^{n, \varepsilon}_\omega\|_{\LL^3(\Omega)}^2\big) \\
& \leq \dfrac{\alpha}{4}\|\Delta^2\phi^{n, \varepsilon}_\omega\|^2 +\frac12 \|\nabla \Delta\phi^{n, \varepsilon}_\omega\|^2 + C\big(1 +\|\nabla \mu^{n, \varepsilon}_\omega\|^2+ \|\rho^{n, \varepsilon}_\omega\|_{H^2(\Omega)}^2 \big).		
\end{align*}
Next, it follows from H\"{o}lder's, Agmon's and Young's inequalities that
\begin{align*}
\| \nabla \cdot\left( |\nabla \phi^{n, \varepsilon}_\omega|^2\nabla \phi^{n, \varepsilon}_\omega \right) \| \|\Delta^2 \phi^{n, \varepsilon}_\omega\|
& \leq C\left( \||\nabla \phi^{n, \varepsilon}_\omega|^2\Delta\phi^{n, \varepsilon}_\omega\|
+ \|\nabla (|\nabla\phi^{n, \varepsilon}_\omega|^2) \cdot \nabla \phi^{n, \varepsilon}_\omega\|\right)\|\Delta^2 \phi^{n, \varepsilon}_\omega\| \\
& \leq C\|\nabla \phi^{n, \varepsilon}_\omega\|_{\LL^\infty(\Omega)}^2
\left( \|\Delta \phi^{n, \varepsilon}_\omega\| + \|\phi^{n, \varepsilon}_\omega\|_{H^2(\Omega)}\right) \|\Delta^2 \phi^{n, \varepsilon}_\omega\|\\
& \leq C\|\nabla \phi^{n, \varepsilon}_\omega\|_{\LL^\infty(\Omega)}^2 \|\Delta^2 \phi^{n, \varepsilon}_\omega\|\\
& \leq C\|\nabla \phi^{n, \varepsilon}_\omega\|_{\mathbf{H}^1(\Omega)}  \|\nabla \phi^{n, \varepsilon}_\omega\|_{\mathbf{H}^2(\Omega)} \|\Delta^2 \phi^{n, \varepsilon}_\omega\|\\
& \leq C\|\phi^{n, \varepsilon}_\omega\|_{H^3(\Omega)} \|\Delta^2 \phi^{n, \varepsilon}_\omega\|\\
& \leq C\big(\|\Delta \phi^{n, \varepsilon}_\omega\|_{H^1(\Omega)}+\|\phi^{n, \varepsilon}_\omega\|\big) \|\Delta^2 \phi^{n, \varepsilon}_\omega\|\\
& \leq C\big(\|\Delta^2 \phi^{n, \varepsilon}_\omega\|^\frac12\|\Delta \phi^{n, \varepsilon}_\omega\|^\frac12
+\|\phi^{n, \varepsilon}_\omega\|_{H^2(\Omega)}\big)\|\Delta^2 \phi^{n, \varepsilon}_\omega\| \\
& \leq \dfrac{\alpha}{4}\|\Delta^2\phi^{n, \varepsilon}_\omega\| + C.
\end{align*}
As a consequence, it holds
\begin{align*}
\alpha\| \Delta^2 \phi^{n, \varepsilon}_\omega\|^2
+ \| \nabla \Delta \phi^{n, \varepsilon}_\omega \|^2
& \leq  \dfrac{\alpha}{2}\|\Delta^2\phi^{n, \varepsilon}_\omega\|^2 +\frac12 \|\nabla \Delta\phi^{n, \varepsilon}_\omega\|^2 + C\big(1 +\|\nabla \mu^{n, \varepsilon}_\omega\|^2+ \|\rho^{n, \varepsilon}_\omega\|_{H^2(\Omega)}^2 \big),
\end{align*}
which implies
\begin{align*}
\alpha\| \Delta^2 \phi^{n, \varepsilon}_\omega\|^2
+ \| \nabla \Delta \phi^{n, \varepsilon}_\omega \|^2
& \leq C\big(1 +\|\nabla \mu^{n, \varepsilon}_\omega\|^2+ \|\rho^{n, \varepsilon}_\omega\|_{H^2(\Omega)}^2 \big).
\end{align*}
From \eqref{rhoL4h2}, the elliptic regularity theory and Lemma \ref{prop:ub2}, we can conclude that $\phi^{n, \varepsilon}_\omega\in L^2(0,T; H^4(\Omega))$.

The proof is complete.
\end{proof}

\subsection{Existence of weak solutions for the penalized problem}
We can now prove the existence of a global weak solution to the penalized problem  \eqref{eq:strongNSCHpert}--\eqref{eq:bcspert} on $[0,T]$.

Let us first pass to the limit as $n\to+\infty$.
In view of Lemmas \ref{prop:ub1}--\ref{prop:ub4}, for fixed $\omega \in (0,1]$, $\varepsilon \in (0, \min(\epsilon_1, \epsilon_2))$, we have proved that
\begin{align*}
        \uu^{n, \varepsilon}_\omega & \text{ is uniformly bounded in } L^\infty(0,T;\mathbf{H}_\sigma) \cap L^2(0,T;\mathbf{V}_\sigma) \cap W^{1, \frac{4}{d}}(0,T;\mathbf{V}_\sigma^*), \\
		\phi^{n, \varepsilon}_\omega & \text{ is uniformly bounded in } L^\infty(0,T;H^2(\Omega)) \cap L^2(0,T; H^4(\Omega)) \cap H^1(0,T;(H^1(\Omega))^*), \\
		\rho^{n, \varepsilon}_\omega & \text{ is uniformly bounded in } L^\infty(0,T;H^1(\Omega)) \cap L^4(0,T; H^2(\Omega)) \cap H^1(0,T;(H^1(\Omega))^*), \\
		\mu^{n, \varepsilon}_\omega & \text{ is uniformly bounded in } L^2(0,T;H^1(\Omega)),  \\
		\psi^{n, \varepsilon}_\omega & \text{ is uniformly bounded in } L^2(0,T;H^1(\Omega)).		
\end{align*}
Besides the obvious weak and weak star convergence (up to a subsequence), the above uniform bounds and the Aubin--Lions lemma allow us to find $(\uu_\omega^{\varepsilon},\phi_\omega^{\varepsilon}, \rho_\omega^{\varepsilon}, \mu_\omega^{\varepsilon}, \psi_\omega^{\varepsilon})$ such that, up to subsequences (not relabelled hereafter)
\begin{align*}
&\uu_\omega^{n,\varepsilon} \to \uu_\omega^{\varepsilon}, \quad \text{strongly in}\ \ L^2(0,T;\mathbf{H}^{1-r}(\Omega)),\\
&\phi^{n, \varepsilon}_\omega \to \phi_\omega^{\varepsilon}, \quad \text{strongly in}\ \ C([0,T];H^{2-r}(\Omega))\cap L^2(0,T; H^{4-r}(\Omega)),\\
& \rho_\omega^{n,\varepsilon} \to \rho_\omega^{\varepsilon},\quad \text{strongly in}\ \ C([0,T];H^{1-r}(\Omega))\cap L^4(0,T; H^{2-r}(\Omega)),
\end{align*}
for $r\in (0,1/2)$, which further imply the a.e. convergence of $(\uu_\omega^{n,\varepsilon},\phi_\omega^{n,\varepsilon}, \rho_\omega^{n,\varepsilon})$ in $\Omega\times (0,T)$. Then we can deduce the further convergences
\begin{align*}
&\big(\nu(\phi_\omega^{n,\varepsilon}, \rho_\omega^{n,\varepsilon})\big)^s D\uu_\omega^{n,\varepsilon} \to \big(\nu(\phi_\omega^{\varepsilon}, \rho_\omega^{\varepsilon})\big)^s D \uu_\omega^{\varepsilon}\quad \text{weakly in}\ \ L^2(0,T;\mathbf{L}^2(\Omega)),\quad s=\frac12, 1,\\
& \mu_\omega^{n,\varepsilon}\nabla \phi_\omega^{n,\varepsilon} \to \mu_\omega^{\varepsilon}\nabla \phi_\omega^{\varepsilon},\quad \text{weakly in}\ \ L^2(0,T;\mathbf{L}^\frac{3}{2}(\Omega)),\\
& \psi_\omega^{n,\varepsilon}\nabla \rho_\omega^{n,\varepsilon} \to \psi_\omega^{\varepsilon}\nabla \rho_\omega^{\varepsilon},\quad \text{weakly in}\ \ L^2(0,T;\mathbf{L}^\frac{3}{2}(\Omega)),\\
&\nabla \cdot(\rho_\omega^{n,\varepsilon}\nabla \phi_\omega^{n,\varepsilon} ) \to \nabla \cdot(\rho_\omega^{\varepsilon}\nabla \phi_\omega^{\varepsilon} ),\quad \text{weakly in}\ \ L^2(0,T;L^2(\Omega)).
\end{align*}
Thus, using well-known arguments (see e.g., \cite{Ab09,Bo99}), we can show that $(\uu_\omega^{\varepsilon},\phi_\omega^{\varepsilon}, \rho_\omega^{\varepsilon}, \mu_\omega^{\varepsilon}, \psi_\omega^{\varepsilon})$ is a global weak solution to \textcolor{black}{the penalized system \eqref{eq:strongNSCHpertap} with the regularized potential $S_{\rho,\varepsilon}$ } subject to \eqref{eq:bcspert}. Moreover, using the weak and weak-$^*$ lower semicontinuity of the norms, from \eqref{eq:energyn} we also deduce
\begin{align}
&\ene(\mathbf{u}^{\varepsilon}_\omega(t), \phi^{\varepsilon}_\omega(t),\rho^{n, \varepsilon}_\omega(t)) + \ii{0}{t}{\|\sqrt{\nu(\phi^{\varepsilon}_\omega(\tau),\rho^{\varepsilon}_\omega(\tau))}D \mathbf{u}^{\varepsilon}_\omega(\tau)\|^2
+\| \nabla \mu^{\varepsilon}_\omega(\tau) \|^2 + \| \nabla \psi^{\varepsilon}_\omega(\tau)\|^2 }{\tau} \notag\\
 &\quad\leq \ene(\uu_0, \phi_0, \rho_0),\quad \forall\,t\in (0,T].
\label{eq:energye}
\end{align}

Next, we pass to the limit as $\varepsilon \to 0^+$. We can still use the a priori bounds obtained in the Galerkin scheme except \eqref{eq:epsilonestimate}. Concerning this bound, on account of \textbf{(H3)}, arguing as in \cite{FG12} (see also \cite[Lemma 3.1]{GGW18}) we find
\begin{align}
\widehat{S}_{\rho, \varepsilon}(s) \leq \widehat{S}_\rho(s) \leq C,\quad \forall\, s \in [0,1].
\label{Srho1}
\end{align}
Thus the updated estimate \eqref{eq:epsilonestimate} for the weak solution yields a bound that is independent of $\varepsilon$.
More precisely, we have
\begin{lemma} \label{prop:enebound2}
For any fixed $\omega \in (0,1]$ and $\varepsilon \in (0, \min(\epsilon_1, \epsilon_2))$ (recall that $\epsilon_2$ depends on $\omega$), let $(\uu_\omega^{\varepsilon},\phi_\omega^{\varepsilon}$, $\rho_\omega^{\varepsilon}$, $\mu_\omega^{\varepsilon}, \psi_\omega^{\varepsilon})$ be a global weak solution solving system \eqref{eq:strongNSCHpertap} with the regularized potential $S_{\rho,\varepsilon}$ subject to \eqref{eq:bcspert} on $[0,T]$. Then, for every $t \in (0,T]$, we have
\begin{align*}
&\ene(\mathbf{u}^{\varepsilon}_\omega(t), \phi^{\varepsilon}_\omega(t),\rho^{ \varepsilon}_\omega(t)) + \ii{0}{t}{\|\sqrt{\nu(\phi^{ \varepsilon}_\omega(\tau),\rho^{\varepsilon}_\omega(\tau))}D \mathbf{u}^{ \varepsilon}_\omega(t)\|^2 + \| \nabla \mu^{\varepsilon}_\omega(\tau) \|^2 + \| \nabla \psi^{\varepsilon}_\omega(\tau)\|^2 }{\tau} \leq C_4,
\end{align*}
and
\begin{align*}
\ene(\mathbf{u}^{\varepsilon}_\omega(t),\phi^{\varepsilon}_\omega(t),\rho^{ \varepsilon}_\omega(t))
&\geq \dfrac{1}{2}\|\uu_{\omega}^{\varepsilon}(t)\|^2
+ \dfrac{\alpha}{4} \|\Delta \phi^{\varepsilon}_\omega(t)\|^2
+ \dfrac{1}{2} \|\nabla \phi^{\varepsilon}_\omega(t)\|^2
+ \dfrac{\beta}{2}\|\nabla \rho^{\varepsilon}_\omega(t)\|^2
\\
&\quad + \frac{c_3}{2} \|\phi^{\varepsilon}_\omega(t)\|^4_{L^4(\Omega)}
+ \dfrac{\omega}{8} \|\nabla\phi^{\varepsilon}_\omega(t)\|^4_{\mathbf{L}^4(\Omega)} -C_5,
\end{align*}
where $C_4,C_5>0$ are independent of $\omega$, $\varepsilon$, and $T$.
\end{lemma}
\begin{proof}
Recalling the energy inequality \eqref{eq:energye} for the approximate solution $(\uu_\omega^{\varepsilon},\phi_\omega^{\varepsilon}, \rho_\omega^{\varepsilon}, \mu_\omega^{\varepsilon}, \psi_\omega^{\varepsilon})$, we can obtain the lower bound with a constant $C_5>0$ that is independent of $\varepsilon$ following the same argument as in Lemma \ref{prop:enebound1}. In order to find $C_4$, we just need to observe that (see also \eqref{Srho1})
		\[
		\begin{split}
			\ene(\phi_0, \rho_0, \uu_0) \leq C\Big(\| \uu_0 \|, \|\phi_0\|_{H^2(\Omega)}, \|\rho_0\|_{H^1(\Omega)}, \int_\Omega S_\rho(\rho_0)\,\mathrm{d} x, \Omega, \alpha, \beta, \theta \Big),
		\end{split}
		\]
thanks to the Sobolev embeddings $H^2(\Omega)\hookrightarrow L^\infty(\Omega)$, $H^2(\Omega)\hookrightarrow W^{1,4}(\Omega)$ ($d=2,3$), \eqref{Srho1} and the fact that $0 \leq \rho_0 \leq 1$ almost everywhere in $\Omega$ (cf. Remark \ref{rem:rho1}).

The proof is complete.
\end{proof}

On account of Lemma \ref{prop:enebound2}, we can now argue as in the proofs of Lemmas \ref{prop:ub1}--\ref{prop:ub4} to deduce a series of uniform estimates with respect to $\varepsilon$ for the approximate solution $(\uu_\omega^{\varepsilon},\phi_\omega^{\varepsilon}, \rho_\omega^{\varepsilon}, \mu_\omega^{\varepsilon}, \psi_\omega^{\varepsilon})$.
The main novelty with respect to the Galerkin scheme is the following estimate which can be deduced from \eqref{eq:strongNSCHpert}$_6$ (see, for instance, \cite[Section 3]{MT16})
\begin{equation}
\label{singbound}
\Vert \widehat{S}'_{\rho, \varepsilon}(\rho^{ \varepsilon}_\omega) \Vert_{L^2(0,T;L^2(\Omega))} \leq C
\end{equation}
for some $C>0$ independent of $\varepsilon$.
These bounds combined with compactness arguments (see e.g., \cite{MT16}) allow us to find $(\uu_\omega,\phi_\omega, \rho_\omega, \mu_\omega, \psi_\omega)$ which is a global weak (or finite energy) solution to the penalized problem \eqref{eq:strongNSCHpert}--\eqref{eq:bcspert}. More precisely, we have

\begin{proposition}
\label{th:weaksolutionomega}
Let \textbf{(H1)}--\textbf{H4} hold and $\omega\in (0,1]$ be given. Then, for any $\uu_0 \in \mathbf{H}_\sigma$, $\phi_0 \in H^2_N(\Omega)$, $\rho_0 \in H^1(\Omega)$ be such that $S_\rho(\rho_0) \in L^1(\Omega)$ and $\overline{\rho_0} \in (0,1)$, problem \eqref{eq:strongNSCHpert}--\eqref{eq:bcspert} admits at least one global weak solution $(\uu_\omega$, $\phi_\omega$, $\rho_\omega$, $\mu_\omega$, $\psi_\omega)$ satisfying
\begin{align*}
& \uu_\omega \in L^{\infty}(0,T; \mathbf{H}_\sigma) \cap L^2(0,T;\mathbf{V}_\sigma)\cap W^{1,\frac{4}{d}}(0,T; \mathbf{V}^*_\sigma),\\
& \phi_\omega \in L^{\infty}(0,T; H^2_N(\Omega)) \cap L^{2}(0,T; H^4_N(\Omega))\cap H^1(0,T; (H^1(\Omega))^*),\\
& \rho_\omega \in L^{\infty}(0,T; H^1(\Omega)) \cap L^{4}(0,T;H^2_N(\Omega)) \cap H^1(0,T; (H^1(\Omega))^*),\\
& \mu_\omega, \psi_\omega \in L^2(0,T; H^1(\Omega)),\\
& \rho_\omega\in L^\infty(\Omega\times(0,T)),\ \ \text{and}\ \ 0 < \rho_\omega (x, t) < 1,\ \ \text{for a.a. }(x,t) \in \Omega \times (0,T),
\end{align*}
and the initial conditions. Moreover, the following energy inequality holds
\begin{align}
& \eneg(\phi_\omega(t),\rho_\omega(t),\mathbf{u}_\omega(t))
+ \ii{0}{t}{ \|\sqrt{\nu(\phi_\omega(\tau),\rho_\omega(\tau))}D \mathbf{u}_\omega(\tau)\|^2
+ \| \nabla \mu_\omega(\tau)\|^2
+ \| \nabla \psi_\omega(\tau) \|^2 }{\tau} \notag \\
&\quad \leq \eneg(\phi_0,\rho_0,\mathbf{u}_0),
\label{weak-IENo}
\end{align}
for every $t \in (0,T]$.
\end{proposition}
\begin{remark}
Observe that  $0 < \rho_\omega(x, t) < 1$ almost everywhere in $\Omega \times (0,T)$ can be achieved by means of an argument based on \eqref{singbound} (see, for instance, \cite{FG12,MZ04}).
Recalling Remark \ref{rem:reg3}, we also have $\| \rho_\omega \|_{L^\infty(0,T;L^\infty(\Omega))}\leq 1$. On the other hand, the same property cannot be deduced for $\phi_\omega$. However, the regularity $\phi_\omega \in L^\infty(0,T; H^2(\Omega))$ ensures that $\phi_\omega \in L^\infty(0,T;L^\infty(\Omega))$ thanks to the embedding $H^2(\Omega)\hookrightarrow L^\infty(\Omega)$.
\end{remark}\medskip

\subsection{Existence of weak solutions to the original problem}

The final step is based on uniform estimates that are independent of $\omega\in (0,1]$. First, we have energy bounds
\begin{lemma} \label{prop:enebound3}
For every $\omega \in (0,1]$, let $(\uu_\omega,\phi_\omega, \rho_\omega, \mu_\omega, \psi_\omega)$ be a global weak solution solving the penalized problem \eqref{eq:strongNSCHpert}--\eqref{eq:bcspert}. Then, for every $t \in (0,T]$, we have
\begin{align*}
&\mathcal{E}_\omega(\mathbf{u}_\omega(t), \phi_\omega(t),\rho_\omega(t)) + \ii{0}{t}{\|\sqrt{\nu(\phi_\omega(\tau),\rho_\omega(\tau))}D \mathbf{u}_\omega(t)\|^2 + \| \nabla \mu_\omega(\tau) \|^2 + \| \nabla \psi_\omega(\tau)\|^2 }{\tau} \leq C_6,
\end{align*}
and
\begin{align*}
\mathcal{E}_\omega(\mathbf{u}_\omega(t),\phi_\omega(t),\rho_\omega(t))
&\geq \dfrac{1}{2}\|\uu_{\omega}(t)\|^2
+ \dfrac{\alpha}{4} \|\Delta \phi_\omega(t)\|^2
+ \dfrac{1}{2} \|\nabla \phi_\omega(t)\|^2
+ \dfrac{\beta}{2}\|\nabla \rho_\omega(t)\|^2
\\
&\quad + \frac{c_3}{2} \|\phi_\omega(t)\|^4_{L^4(\Omega)}
 -C_7,
\end{align*}
where $C_6,C_7>0$ do not depend on $\omega$ and $T$.
\end{lemma}
\begin{proof}
The upper bound is straightforward (cf. the proof of Lemma \ref{prop:enebound2}).
Concerning the lower bound, we shall essentially make use of the estimate $\|\rho_\omega\|_{L^\infty(0,T;L^\infty(\Omega))}\leq 1$. Recalling \textbf{(H2)}, \textbf{(H3)} and arguing as in Remark \ref{Rm-ill}, we can achieve the conclusion by noting that the perturbation involving $\omega$ is nonnegative.

The proof is complete.
\end{proof}

Thanks to Lemma \ref{prop:enebound3} and reasoning as above, we can derive a number of uniform estimates that are independent of $\omega$. In particular we find again
(see \eqref{singbound})
\begin{equation}
\label{singbound2}
\Vert \widehat{S}'_{\rho}(\rho_\omega) \Vert_{L^2(0,T;L^2(\Omega))} \leq C
\end{equation}
for some $C>0$ independent of $\omega$.  This is enough to find, through compactness arguments by taking $\omega\to 0^+$ (up to a convergent subsequence), a quintuplet $(\uu, \phi, \rho, \mu, \psi)$ which is a global weak solution to problem \eqref{eq:strongNSCH}--\eqref{eq:bcs} in the sense of Definition \ref{def:solution}, provided we establish some additional spatial regularity. First, we show that $\rho \in L^2(0,T; W^{2,p}(\Omega))$ for every finite $p > 2$ if $d = 2$ and for every $2 < p \leq 6$ if $d = 3$. To this end, we observe that $\rho$ solves the semilinear problem with a singular nonlinearity
\begin{equation}
\label{eq:nnlp1}
\begin{cases}
-\beta\Delta \rho + \widehat{S}_\rho'(\rho) = \psi + \dfrac{\theta}{2}|\nabla \phi|^2 - R'_\rho(\rho) & \quad \text{a.e. in } \Omega,\\
\partial_\mathbf{n}\rho = 0 & \quad \text{a.e. on } \partial\Omega.
\end{cases}
\end{equation}
We know that $\psi\in L^2(0,T;H^1(\Omega))$. Moreover, we have
\[
\begin{split}
\| |\nabla \phi|^2\|_{H^1(\Omega)} \leq C\|\phi\|_{W^{1,4}(\Omega)}\|\phi\|_{W^{2,4}(\Omega)} & \leq C\|\phi\|_{H^{2}(\Omega)}\|\phi\|_{H^{3}(\Omega)}
\end{split}
\]
and
$$
\|R'_\rho(\rho)\|_{H^1(\Omega)}\leq C(1+\|\rho\|_{H^1(\Omega)}).
$$
Thus the right-hand side of \eqref{eq:nnlp1} belongs to $L^2(0,T;H^1(\Omega))$. Therefore, recalling \cite[Theorem 6]{Ab09} (see also \cite[Theorem A.2]{GMT18}) , we conclude that $\rho \in L^2(0,T; W^{2,p}(\Omega))$. Consider now the elliptic problem
\begin{equation*}
\label{eq:nnp2}
\begin{cases}
 \alpha \Delta^2 \phi = \mu + \Delta \phi -S_\phi'(\phi) - \theta\nabla \cdot (\rho \nabla \phi) & \quad \text{a.e. in } \Omega,\\
\partial_\mathbf{n}\phi = 0 & \quad \text{a.e. on } \partial\Omega.
\end{cases}
\end{equation*}
A well-known elliptic estimate yields
\begin{align*}
\|\phi\|_{H^5(\Omega)}&\leq C( \|\mu + \Delta \phi -S_\phi'(\phi) - \theta\nabla \cdot (\rho \nabla \phi)\|_{H^1(\Omega)}+\|\phi\|)\\
&\leq C\big(\|\mu\|_{H^1(\Omega)}+\|\phi\|_{H^3(\Omega)}+ \|\nabla \cdot (\rho \nabla \phi)\|_{H^1(\Omega)}\big).
\end{align*}
On the other hand, we have
\begin{align*}
\|\nabla \cdot (\rho \nabla \phi)\|_{H^1(\Omega)}
&\leq C\big(\|\rho\|_{W^{2,4}(\Omega)}\|\phi\|_{W^{1,4}(\Omega)}+ \|\rho\|_{H^1(\Omega)}\|\phi\|_{W^{2,\infty}(\Omega)}+\|\rho\|_{L^\infty(\Omega)}\|\phi\|_{H^3(\Omega)}\big)\\
&\leq C\big(\|\rho\|_{W^{2,4}(\Omega)}\|\phi\|_{H^2(\Omega)}+ \|\rho\|_{H^1(\Omega)} \|\phi\|_{H^4(\Omega)}+\|\phi\|_{H^3(\Omega)}\big),
\end{align*}
and recalling that $\phi\in L^{\infty}(0,T; H^2(\Omega)) \cap L^{2}(0,T; H^4(\Omega))$ and $\rho \in L^{\infty}(0,T; H^1(\Omega)) \cap L^{2}(0,T;W^{2,4}(\Omega))$, we infer that $\phi\in L^2(0,T; H^5(\Omega))$.

Finally, through a semicontinuity argument applied to \eqref{weak-IENo}, we can also recover the energy inequality \eqref{weak-IEN}. If $d=2$, the regularity of weak solutions allow us to derive an energy equality by arguing as in the proof of Proposition \ref{MEconv} (see also \cite{Ab09}).

The existence part of Theorem \ref{th:weaksolution} is now proved.
\hfill $\square$

\begin{remark}
\label{FHphi}
As we mentioned in the Introduction, it would be physically reasonable to take a Flory--Huggins potential for $\phi$ as well. From the mathematical point of view, this case is highly non-trivial since $\phi$ satisfies a sixth order Cahn--Hilliard type equation with a singular potential (cf. \cite{AM15}). In the approximation scheme, the essential bound \eqref{singbound2} (see also \eqref{singbound}) cannot be recovered anymore because of the fourth-order term in the chemical potential. Thus it is not clear how to establish the existence of a weak solution in the usual sense. On one hand, maybe one could show the existence of a weaker solution like the one obtained for a single Cahn--Hilliard equation in \cite{AM15}. See also \cite{SP13b,SW20} for alternative approaches to handle singular potentials. On the other hand, one may want to consider a standard fourth-order Cahn--Hilliard equation for $\phi$ (i.e., taking $\alpha=0$ in \eqref{total}). In this case, the existence of a weak solution might be provable provided that $\theta\in (0,1)$ (see \eqref{theta}). However, other results (e.g., uniqueness and regularity in two dimensions) could be rather challenging because of the couplings between the two Cahn--Hilliard equations.
\end{remark}

\subsection{Uniqueness of weak solutions when $d=2$}
 \label{sec:proof2}

Suppose that $(\uu_{0},\phi_{0}, \rho_{0})\in \mathbf{H}_\sigma\times H^2_N(\Omega)\times H^1(\Omega)$ is a set of initial data satisfying the assumptions of Theorem \ref{th:weaksolution}. Denote by $(\uu_1,\phi_1, \rho_1)$ and $(\uu_2,\phi_2, \rho_2)$ two global weak solutions to problem \eqref{eq:strongNSCH}--\eqref{eq:bcs} departing from $(\uu_{0},\phi_{0}, \rho_{0})$, with corresponding chemical potentials $\mu_i$ and $\psi_i$, for $i = 1,2$. Set (see Remark \ref{rem:reg1} for pressures)
\begin{equation}
\label{initial}
	\begin{cases}
      	\uu = \uu_1 - \uu_2, & \quad \pi = \pi_1 - \pi_2,\\
		\phi = \phi_1 - \phi_2, & \quad \rho = \rho_1 - \rho_2, \\
		\mu = \mu_1 - \mu_2, & \quad \psi= \psi_1 - \psi_2.
	\end{cases}
\end{equation}
Writing down (formally) the system for $(\uu,\phi,\rho)$, we get
\begin{equation} \label{eq:stab0}
\begin{cases}
			\partial_t\mathbf{u} + (\mathbf{u}_1 \cdot \nabla)\mathbf{u}_1 - (\mathbf{u}_2 \cdot \nabla)\mathbf{u}_2 - \nabla \cdot \left( \nu(\phi_1,\rho_1) D\uu_1 - \nu(\phi_2, \rho_2)D\uu_2 \right)  + \nabla \pi\\
\qquad  =\mu_1 \nabla \phi_1 - \mu_2 \nabla \phi_2 +  \psi_1 \nabla \rho_1 - \psi_2\nabla \rho_2 ,\\
			\nabla \cdot \mathbf{u} = 0,\\
\partial_t\phi + \uu_1 \cdot \nabla \phi +\uu \cdot \nabla \phi_2 = \Delta \mu,\\
			\mu = \alpha \Delta^2 \phi -\Delta \phi + S'_\phi(\phi_1)-S'_\phi(\phi_2) + \theta\nabla \cdot (\rho_1 \nabla \phi + \rho \nabla \phi_2),\\
			\partial_t\rho +  \uu_1 \cdot \nabla \rho + \uu \cdot \nabla \rho_2 = \Delta \psi, \\
			\psi = -\beta \Delta \rho + S'_\rho(\rho_1)-S'_\rho(\rho_2) - \dfrac{\theta}{2}|\nabla \phi_1|^2 + \dfrac{\theta}{2} |\nabla \phi_2|^2,
\end{cases}
\end{equation}
in $\Omega\times (0,T)$, with
\begin{equation} \label{eq:bcstab}
\begin{cases}
			\mathbf{u} = \mathbf{0} & \quad \text{on } \partial \Omega \times (0,T), \\
			\partial_\mathbf{n} \phi = \partial_\mathbf{n}\Delta \phi = \partial_\mathbf{n}\mu =0 & \quad \text{on } \partial \Omega \times (0,T), \\
			\partial_\mathbf{n} \rho = \partial_\mathbf{n} \psi = 0 & \quad \text{on } \partial \Omega \times (0,T), \\
			\mathbf{u}|_{t=0}=\mathbf{0},\quad \phi|_{t=0} = 0,\quad \rho|_{t=0} = 0,	 & \quad \text{in } \Omega.
\end{cases}
\end{equation}

In the subsequent analysis, on account of \cite{GMT18}, we derive a differential inequality for problem \eqref{eq:stab0}--\eqref{eq:bcstab} involving weaker norms with respect to the energy ones (cf. \eqref{total}). The proof consists of several steps. We indicate by $C$ a generic positive constant depending on known quantities.
\medskip

		\textbf{Step 1.} Testing the evolution equation for $\phi$ in \eqref{eq:stab0} by $\phi$, using integration by parts, we get
		\[
		\left\langle \partial_t\phi, \phi \right\rangle_{H^1(\Omega))^*,H^1(\Omega)}
       - \left( \phi\uu_1, \nabla \phi \right) -  \left( \phi_2\uu, \nabla \phi \right) = - (\nabla \mu, \nabla \phi).
		\]
		Using the equation for $\mu$ and  a further integration by parts, we obtain
		\begin{equation} \label{eq:stab1}
			\dfrac{1}{2}\td{}{t}\|\phi\|^2+ \alpha\| \nabla \Delta \phi\|^2 + \|\Delta \phi\|^2    = \mathcal{I}_1 + \mathcal{I}_2 + \mathcal{I}_3,
		\end{equation}
		where
		\begin{align*}
			\mathcal{I}_1 & := \left( \phi\uu_1, \nabla \phi \right) + \left( \phi_2\uu , \nabla \phi \right),\\
			\mathcal{I}_2 & := \left( S'_\phi(\phi_1) - S'_\phi(\phi_2), \Delta \phi \right), \\
			\mathcal{I}_3 & := \theta\left( \nabla \cdot(\rho_1 \nabla \phi), \Delta \phi\right) + \theta\left( \nabla \cdot(\rho \nabla \phi_2), \Delta \phi\right).
		\end{align*}
\textcolor{black}{Concerning  $\mathcal{I}_1$, observe that $\left( \phi\uu_1, \nabla \phi \right)=0$ since $\nabla \cdot \uu_1=0$ and $\|\nabla \phi\|^2=-(\Delta\phi,\phi)$. Thus we have}
\begin{align}
\mathcal{I}_1 & = \left( \phi_2\uu , \nabla \phi \right) \leq  \|\phi_2\|_{L^\infty(\Omega)}\|\uu\|\|\nabla \phi\| \notag\\
& \leq C\|\uu\|\|\phi\|^\frac{1}{2}\|\Delta \phi\|^\frac{1}{2} \notag\\
& \leq \dfrac{1}{18}\|\Delta \phi\|^2 + \dfrac{\nu_*}{20}\|\uu\|^2 + C\|\phi\|^2.		
\label{eq:stab11}
\end{align}
 Next, for $\mathcal{I}_2$ it holds
\begin{align}
\mathcal{I}_2 & = \left(S'_\phi(\phi_1) - S'_\phi(\phi_2), \Delta \phi \right) \notag\\
& \leq \left|\left(\int_0^1S''_\phi(\tau\phi_1+(1-\tau)\phi_2)\phi\,\mathrm{d}\tau,\, \Delta\phi\right)\right|  \notag \\
& \leq C\left\|\int_0^1S''_\phi(\tau\phi_1+(1-\tau)\phi_2)\,\mathrm{d}\tau \right\|_{L^\infty(\Omega)} \|\phi\|\| \Delta \phi\| \notag \\
& \leq \dfrac{1}{18}\|\Delta \phi\|^2 + C\|\phi\|^2.
\label{eq:stab12}
\end{align}
Finally, owing to standard Sobolev embeddings, the Poincar\'{e}--Wirtinger inequality and the elliptic estimate, we have
\begin{align}
\mathcal{I}_3
& = \theta(\rho_1 \Delta\phi, \Delta \phi) + \theta(\rho \Delta\phi_2, \Delta \phi) + \theta(\nabla \rho_1 \cdot \nabla \phi, \Delta \phi) + \theta(\nabla \rho \cdot \nabla \phi_2, \Delta \phi) \notag  \\
& \leq C\|\Delta \phi\|_{L^3(\Omega)}\big( \|\rho_1\|_{L^6(\Omega)}\|\Delta\phi\|
+ \|\rho\|_{L^6(\Omega)}\|\Delta\phi_2\|
+ \|\nabla \rho_1\| \|\nabla \phi\|_{\LL^6(\Omega)}
+ \|\nabla \rho\|\|\nabla \phi_2\|_{\mathbf{L}^6(\Omega)} \big) \notag \\
& \leq C\|\phi\|^\frac{2}{9}\|\nabla \Delta \phi\|^\frac{7}{9} \left( \|\Delta\phi\| + \|\nabla \rho\| \right)\notag \\
& \leq \dfrac{1}{18}\|\Delta \phi\|^2 + \dfrac{\alpha}{6}\|\nabla \Delta \phi\|^2 + \dfrac{\beta}{12}\| \nabla \rho\|^2  + C\| \phi\|^2.
\label{eq:stab13}
\end{align}
Collecting the estimates \eqref{eq:stab11}--\eqref{eq:stab13}, we deduce from  \eqref{eq:stab1} that
\begin{equation} \label{eq:stab1final}
\dfrac{1}{2}\td{}{t}\| \phi\|^2 + \dfrac{5\alpha}{6}\| \nabla \Delta \phi\|^2+ \frac{5}{6}\|\Delta \phi\|^2
\leq \dfrac{\nu_*}{20}\|\uu\|^2 + \dfrac{\beta}{12}\| \nabla \rho\|^2  + C\|\phi\|^2.		
\end{equation}

\textbf{Step 2.} We now apply a similar argument for $\rho$, but the presence of the singular potential forces us to take $\mathcal{N}\rho \in V_0$ as test function in \eqref{eq:stab0}. This yields
	\[
		\left\langle \partial_t\rho, \mathcal{N}\rho \right\rangle_{H^1(\Omega))^*,H^1(\Omega)}  - \left(\rho\mathbf{u}_1, \nabla \mathcal{N}\rho\right) - \left(\rho_2\mathbf{u}, \nabla \mathcal{N}\rho\right) =- (\psi, \rho).
	\]
Like in Step 1, we compute the last scalar product using the equation for $\psi$ and obtain
\begin{equation} \label{eq:stab2}
\dfrac{1}{2}\td{}{t}\| \rho \|^2_{V_0^*} + \beta\|\nabla\rho\|^2 + \mathcal{I}_4 = \mathcal{I}_5 + \mathcal{I}_6,
\end{equation}
where
\begin{align*}
    \mathcal{I}_4 & := \big(S'_\rho(\rho_1)-S'_\rho(\rho_2), \rho\big), \\
	\mathcal{I}_5 & := \left(\rho\mathbf{u}_1, \nabla \mathcal{N}\rho\right) + \left(\rho_2\mathbf{u}, \nabla \mathcal{N}\rho\right), \\
	\mathcal{I}_6 & := \dfrac{\theta}{2}\big(\nabla \phi\cdot(\nabla\phi_1+\nabla \phi_2), \rho\big).
\end{align*}
Recalling  \textbf{(H3)}, we easily obtain
\begin{equation}
\label{eq:stab22}
\mathcal{I}_4  \geq -L_1 \| \rho \|^2 \geq -\frac{\beta}{12}\Vert \nabla \rho\|^2 - C\|\rho\|^2_{V_0^*}.
\end{equation}
Arguing as for $\mathcal{I}_1$, we get
\begin{align}
\mathcal{I}_5
& \leq \|\rho\|_{L^6(\Omega)}\|\uu_1\|_{\LL^3(\Omega)}\|\rho\|_{V_0^*}
+ \|\rho_2\|_{L^\infty(\Omega)}\|\uu\|\|\rho\|_{V_0^*} \notag \\
& \leq \dfrac{\beta}{12}\| \nabla \rho \|^2  + \dfrac{\nu_*}{20}\|\uu\|^2
+ C\big( 1 + \|\uu_1\|_{\LL^3(\Omega)}^2 \big) \|\rho\|_{V_0^*}^2.
\label{eq:stab21}
\end{align}
Then, Sobolev embeddings and the Poincar\'{e}--Wirtinger inequality give
\begin{align}
	\mathcal{I}_6 & \leq \big(\|\nabla \phi_1\|_{\mathbf{L}^4(\Omega)}
    + \|\nabla \phi_2\|_{\mathbf{L}^4(\Omega)}\big) \|\nabla \phi\| \|\rho\|_{L^4(\Omega)} \notag \\
	& \leq C\|\nabla \phi\|\|\rho\|_{L^4(\Omega)} \notag \\
	& \leq C\|\phi\|^\frac{1}{2}\|\Delta \phi\|^\frac{1}{2}\|\nabla \rho\| \notag \\
    & \leq  \dfrac{\beta}{12}\|\nabla \rho\|^2+ \dfrac{1}{18}\|\Delta \phi\|^2 + C\|\phi\|^2.
    \label{eq:stab23}
\end{align}
Thus, we can conclude from estimates \eqref{eq:stab21}--\eqref{eq:stab23} and  \eqref{eq:stab2} that
\begin{equation} \label{eq:stab2final}
			\dfrac{1}{2}\td{}{t}\| \rho \|^2_{V_0^*}
+ \frac{3\beta}{4}\|\nabla \rho\|^2 \leq \dfrac{\nu_*}{20}\|\uu\|^2 + \dfrac{1}{18}\| \nabla \phi\|^2
+  C\big( 1 + \|\uu_1\|_{\LL^3(\Omega)}^2 \big) \|\rho\|_{V_0^*}^2 + C\|\phi\|^2.
\end{equation}

\textbf{Step 3.} Now we consider the Navier--Stokes system. For the sake of convenience, we make use of the vectorial identity
		\begin{equation*} \label{eq:identity1}
			(\uu_i \cdot \nabla)\uu_i = \nabla \cdot (\uu_i \otimes \uu_i), \qquad i = 1, 2,
		\end{equation*}
which holds thanks to $\nabla\cdot \uu_i=0$. Besides, we recast the Korteweg forces by using the equations for $\mu_i$, $\psi_i$, $i = 1, 2$,
and we write
\begin{align*}
    \mu_i\nabla \phi_i + \psi_i \nabla \rho_i
& = \nabla \left( \dfrac{1}{2}(1-\theta\rho_i)|\nabla \phi_i|^2 + \dfrac{\beta}{2}|\nabla \rho_i|^2 + S_\phi(\phi_i) + S_\rho(\rho_i)\right) \\
&\quad - \nabla \cdot \big( (1-\theta\rho_i) \nabla \phi_i \otimes \nabla \phi_i
+ \beta \nabla \rho_i \otimes \nabla \rho_i - \alpha \nabla\Delta\phi_i \otimes \nabla \phi_i \big) \\
&\quad - \alpha(\nabla\Delta\phi_i \cdot \nabla)\nabla\phi_i.
\end{align*}
In this way, we get rid of the chemical potentials by considering extra pressure terms.
After introducing these modifications, we test the equation for $\uu$ by $\mathbf{A}^{-1}\uu \in \mathbf{W}_\sigma$, which yields
\begin{equation} \label{eq:stab3pre}
			\dfrac{1}{2}\td{}{t}\|\uu\|_{\mathbf{V}_\sigma^*}^2 + (\nu(\phi_1, \rho_1)D\uu, \nabla \mathbf{A}^{-1}\uu) = \sum_{j = 7}^{13} \mathcal{I}_j,
\end{equation}
where we set (using integrations by parts and adding/subtracting suitable quantities)
		\begin{align*}
			\mathcal{I}_7 & := (\uu_1 \otimes \uu, \nabla \mathbf{A}^{-1}\uu) + (\uu \otimes \uu_2, \nabla \mathbf{A}^{-1}\uu)  \\
			\mathcal{I}_8 & := (\nabla \phi_1 \otimes \nabla \phi, \nabla \mathbf{A}^{-1}\uu) + (\nabla \phi \otimes \nabla \phi_2, \nabla \mathbf{A}^{-1}\uu), \\
			\mathcal{I}_9 & := \beta(\nabla \rho_1 \otimes \nabla \rho, \nabla \mathbf{A}^{-1}\uu) + \beta(\nabla \rho \otimes \nabla \rho_2, \nabla \mathbf{A}^{-1}\uu), \\
			\mathcal{I}_{10} & := -\theta(\rho_1\nabla \phi_1 \otimes \nabla \phi, \nabla \mathbf{A}^{-1}\uu) - \theta(\rho_1\nabla \phi \otimes \nabla \phi_2, \nabla \mathbf{A}^{-1}\uu) - \theta(\rho\nabla \phi_2 \otimes \nabla \phi_2, \nabla \mathbf{A}^{-1}\uu), \\
			\mathcal{I}_{11} & := -\alpha(\nabla\Delta\phi_1 \otimes \nabla \phi, \nabla \mathbf{A}^{-1}\uu) - \alpha(\nabla \Delta \phi \otimes \nabla \phi_2, \nabla \mathbf{A}^{-1}\uu),\\
			\mathcal{I}_{12} & := -\alpha\big((\nabla\Delta\phi_1 \cdot \nabla)\nabla\phi,  \mathbf{A}^{-1}\uu\big) - \alpha\big((\nabla\Delta\phi \cdot \nabla)\nabla\phi_2,  \mathbf{A}^{-1}\uu\big),\\
			\mathcal{I}_{13} & := - \big((\nu(\phi_1, \rho_1) - \nu(\phi_2, \rho_2))D\uu_2, \nabla \mathbf{A}^{-1}\uu\big).
		\end{align*}
We analyze the remainder terms on the left-hand side of \eqref{eq:stab3pre} by using the argument in \cite{GMT18}. Since $\nabla \cdot \nabla \uu^\mathrm{T} = \nabla \nabla \cdot \uu =  0$, we deduce that
\begin{align} \label{eq:stab31}
\big(\nu(\phi_1, \rho_1)D\uu, \nabla \mathbf{A}^{-1}\uu\big)
    & = \big(\nu(\phi_1, \rho_1) D\uu, D\mathbf{A}^{-1}\uu\big) \notag\\
    & = \big(\nabla \uu, \nu(\phi_1, \rho_1) D\mathbf{A}^{-1}\uu\big) \notag\\
	& = - \big(\uu, \nabla \cdot [\nu(\phi_1, \rho_1)D\mathbf{A}^{-1}\uu] \big) \notag \\
	& = -  \big(\uu, D\mathbf{A}^{-1}\uu\nabla \nu(\phi_1, \rho_1)\big) - \dfrac{1}{2}\big(\uu, \nu(\phi_1, \rho_1)\Delta\mathbf{A}^{-1}\uu\big).
\end{align}
From the definition of the Stokes operator, we find that there exists $q \in L^2(0,T;H^1(\Omega))$ satisfying $-\Delta\mathbf{A}^{-1}\uu + \nabla q = \uu$ almost everywhere in $\Omega \times (0,T)$ (cf. Lemma \ref{stokes}). Moreover, it holds
\begin{equation} \label{eq:qest}
			\|q\| \leq C\|\nabla \mathbf{A}^{-1}\uu\|^\frac{1}{2}\|\uu\|^\frac{1}{2}, \qquad \|\nabla q\| \leq C \|\uu\|.
\end{equation}
Therefore, the second term on the right-hand side of \eqref{eq:stab31} can be estimated as follows:
\begin{align*}
				- \dfrac{1}{2}\big(\uu, \nu(\phi_1,\rho_1)\Delta\mathbf{A}^{-1}\uu\big)
 & = \dfrac{1}{2}\big(\uu, \nu(\phi_1,\rho_1)\uu\big)
 - \dfrac{1}{2}\big(\uu, \nu(\phi_1,\rho_1)\nabla q\big) \\
 & \geq \dfrac{\nu_*}{2}\|\uu\|^2 - \dfrac{1}{2}\big(\uu, \nu(\phi_1,\rho_1)\nabla q\big).
\end{align*}	
Setting
\begin{align*}
\mathcal{I}_{14} := \big(\uu, D\mathbf{A}^{-1}\uu\nabla \nu(\phi_1, \rho_1) \big),\qquad \mathcal{I}_{15} := \dfrac{1}{2}\big(\uu, \nu(\phi_1,\rho_1)\nabla q\big),
\end{align*}
we then recast \eqref{eq:stab3pre} as
\begin{equation} \label{eq:stab3}
			\dfrac{1}{2}\td{}{t}\|\uu\|_{\mathbf{V}_\sigma^*}^2 + \dfrac{\nu_*}{2}\|\uu\|^2 \leq \sum_{j = 7}^{15} \mathcal{I}_j.
\end{equation}
Next, we estimate all the $\mathcal{I}_j$ terms defined above.
Using the Ladyzhenskaya inequality \eqref{eq:gn1} and Young's inequality, we can deduce that  (see \cite{GMT18})
\begin{align}
\mathcal{I}_7
& \leq \left( \|\uu_1\|_{\LL^4(\Omega)} + \|\uu_2\|_{\LL^4(\Omega)} \right)\|\uu\|\|\nabla \mathbf{A}^{-1}\uu\| \notag \\
& \leq C\left( \|\uu_1\|^\frac12\|\uu_1\|_{\mathbf{V}_\sigma}^\frac12
               + \|\uu_2\|^\frac12\|\uu_2\|_{\mathbf{V}_\sigma}^\frac12 \right)\|\uu\|_{\mathbf{V}_\sigma^*}^\frac{1}{2}\|\uu\|^\frac{3}{2} \notag \\
& \leq \dfrac{\nu_*}{20}\|\uu\|^2 + C\left( \|\uu_1\|_{\mathbf{V}_\sigma}^2 + \|\uu_2\|_{\mathbf{V}_\sigma}^2 \right) \|\uu\|_{\mathbf{V}_\sigma^*}^2,
	\label{eq:stab32}
\end{align}
\begin{align}
\mathcal{I}_8
& \leq \left( \| \nabla \phi_1 \|_{\LL^\infty(\Omega)} + \| \nabla \phi_2 \|_{\LL^\infty(\Omega)} \right) \| \nabla \phi \| \|\nabla \mathbf{A}^{-1} \uu\| \notag \\
& \leq \dfrac{1}{18}\|\Delta \phi\|^2 + C\|\phi\|^2 + C\left(  \| \nabla \phi_1 \|_{\LL^\infty(\Omega)}^2 + \| \nabla \phi_2 \|_{\LL^\infty(\Omega)}^2 \right)\|\uu\|^2_{\mathbf{V}_\sigma^*},
\label{eq:stab33}
\end{align}
\begin{align}
\mathcal{I}_9 & \leq \left( \| \nabla \rho_1 \|_{\LL^\infty(\Omega)} + \| \nabla \rho_2 \|_{\LL^\infty(\Omega)} \right) \| \nabla \rho \| \|\nabla \mathbf{A}^{-1} \uu\| \notag \\
& \leq \dfrac{\beta}{12}\| \nabla \rho\|^2 + C\big(  \| \nabla \rho_1 \|_{\LL^\infty(\Omega)}^2 + \| \nabla \rho_2 \|_{\LL^\infty(\Omega)}^2 \big)\|\uu\|^2_{\mathbf{V}_\sigma^*}.
\label{eq:stab34}
\end{align}
The estimate for $\mathcal{I}_{10}$ is slightly more involved. Indeed we have
\begin{align}
\mathcal{I}_{10}
& \leq C\|\rho_1\|_{L^\infty(\Omega)}\left( \| \nabla \phi_1 \|_{\LL^\infty(\Omega)} + \| \nabla \phi_2 \|_{\LL^\infty(\Omega)} \right) \| \nabla \phi \| \|\nabla \mathbf{A}^{-1} \uu\| \notag\\
&\quad +  C \| \rho\|_{L^4(\Omega)} \| \nabla \phi_2 \|_{\LL^4(\Omega)}\| \nabla \phi_2 \|_{\LL^\infty(\Omega)} \| \nabla \mathbf{A}^{-1}  \uu\| \notag \\
& \leq C\left( \| \nabla \phi_1 \|_{\LL^\infty(\Omega)} + \| \nabla \phi_2 \|_{\LL^\infty(\Omega)} \right) \| \Delta \phi \|^\frac12\|\phi\|^\frac12 \|\nabla \mathbf{A}^{-1} \uu\| \notag\\
&\quad +  C \| \nabla \rho\|^\frac34\|\rho\|_{V_0^*}^\frac14 \| \nabla \phi_2 \|_{\LL^\infty(\Omega)} \| \nabla \mathbf{A}^{-1}  \uu\| \notag \\
& \leq \dfrac{1}{18}\|\Delta \phi\|^2 + \dfrac{\beta}{12}\| \nabla \rho\|^2
+ C\big(  \| \nabla \phi_1 \|_{\LL^\infty(\Omega)}^2 + \| \nabla \phi_2 \|_{\LL^\infty(\Omega)}^2 \big)\|\uu\|^2_{\mathbf{V}_\sigma^*} + C\big( \|\phi\|^2 + \|\rho\|^2_{V_0^*}\big).
\label{eq:stab35}
\end{align}
Using now Sobolev embeddings, we deduce that
\begin{align}
\mathcal{I}_{11}
& \leq C\left( \|\nabla \Delta \phi_1\|_{\LL^4(\Omega)}\|\nabla \phi\|_{\LL^4(\Omega)}\|\nabla \mathbf{A}^{-1}\uu\| + \|\nabla\Delta \phi\|\|\nabla \phi_2\|_{\LL^\infty(\Omega)}\|\nabla \mathbf{A}^{-1}\uu\|\right) \notag \\
& \leq \dfrac{\alpha}{6}\|\nabla \Delta \phi\|^2 + \dfrac{1}{18}\|\Delta \phi\|^2 +  C\big(\|\nabla \Delta \phi_1\|_{\LL^4(\Omega)}^2 + \|\nabla \phi_2\|_{\LL^\infty(\Omega)}^2 \big)\|\uu\|_{\mathbf{V}_\sigma^*}^2.
\label{eq:stab36}
\end{align}
Recalling that $\mathbf{V}_\sigma \hookrightarrow \LL^r(\Omega)$ for every $r > 0$, we find
\begin{align}
\mathcal{I}_{12}
    & \leq C\left( \|\nabla \Delta \phi_1\|_{\LL^4(\Omega)}\|\phi\|_{H^2(\Omega)}\| \mathbf{A}^{-1}\uu\|_{\LL^4(\Omega)}
        + \|\nabla\Delta \phi\|\|\phi_2\|_{W^{2,4}(\Omega)}\| \mathbf{A}^{-1}\uu\|_{\LL^4(\Omega)}\right)\notag \\
	& \leq \dfrac{\alpha}{6}\|\nabla \Delta \phi\|^2 + \dfrac{1}{18}\|\Delta \phi\|^2
        +  C\big(\|\nabla \Delta \phi_1\|_{\LL^4(\Omega)}^2 +\|\phi_2\|_{W^{2,4}(\Omega)}^2 \big)\|\uu\|_{\mathbf{V}_\sigma^*}^2.
 \label{eq:stab37}
\end{align}
Let us now handle the terms involving viscosity. Consider $\mathcal{I}_{13}$.
Making use of Agmon's inequality \eqref{eq:agmon1}, the Poincar\'{e}--Wirtinger inequality and \cite[Proposition C.1]{GMT18}, we obtain
\begin{align}
\mathcal{I}_{13} & = - ((\nu(\phi_1, \rho_1) - \nu(\phi_1, \rho_2))D\uu_2, \nabla \mathbf{A}^{-1}\uu) - ((\nu(\phi_1, \rho_2) - \nu(\phi_2, \rho_2))D\uu_2, \nabla \mathbf{A}^{-1}\uu) \notag \\
				& = -\left( \ii{0}{1}{\partial_\rho\nu(\phi_1, s\rho_1 + (1-s)\rho_2) \rho }{s}D\uu_2, \nabla \mathbf{A}^{-1}\uu \right)\notag\\
&\quad - \left( \ii{0}{1}{\partial_\phi\nu(s\phi_1 + (1-s)\phi_2, \rho_2)\phi}{s} D\uu_2, \nabla \mathbf{A}^{-1}\uu \right) \notag \\
				& \leq C\|D\uu_2\|\|\rho \nabla \mathbf{A}^{-1}\uu\|
                + C\|D\uu_2\| \|\phi\|_{L^\infty(\Omega)} \|\nabla \mathbf{A}^{-1}\uu\| \notag \\
				& \leq C\|D \uu_2\| \|\nabla \rho\| \|\uu\|_{\mathbf{V}_\sigma^*} \left[ \ln \left(\dfrac{e\|\nabla \mathbf{A}^{-1}\uu\|_{\mathbf{H}^1(\Omega)}}{\|\uu\|_{\mathbf{V}_\sigma^*}} \right) \right]^\frac{1}{2}
+ \dfrac{1}{18}\|\Delta \phi\|^2 + C\|\phi\|^2 + \|D\uu_2\|^2\|\uu\|_{\mathbf{V}_\sigma^*}^2.
\label{eq:stab38}
\end{align}		
Next, we see that
\begin{align}
\mathcal{I}_{14}
& \leq C\left( \|\nabla \phi_1\|_{\LL^\infty(\Omega)} + \|\nabla \rho_1\|_{\LL^\infty(\Omega)}\right)\|\uu\|\|D\mathbf{A}^{-1}\uu\| \notag \\
& \leq \dfrac{\nu_*}{20}\|\uu\|^2 +  C\big(\|\nabla \phi_1\|_{\LL^\infty(\Omega)}^2 + \|\nabla \rho_1\|_{\LL^\infty(\Omega)}^2 \big)\|\uu\|_{\mathbf{V}_\sigma^*}^2,
\label{eq:stab39}
\end{align}
and exploiting \eqref{eq:gn1} and \eqref{eq:agmon1}, jointly with \eqref{eq:qest}, we get
\begin{align}
\mathcal{I}_{15}
& = -\dfrac{1}{2}(\uu \cdot \nabla \nu(\phi_1, \rho_1), q) \notag \\
				& \leq C\left( \|\nabla \phi_1\|_{\LL^4(\Omega)} + \|\nabla \rho_1\|_{\LL^4(\Omega)}\right)\|\uu\|\|q\|_{L^4(\Omega)} \notag \\
				& \leq C\big( \|\phi_1\|_{H^2(\Omega)} + \|\rho_1\|^\frac{1}{2}_{L^\infty(\Omega)}\|\rho_1\|_{H^2(\Omega)}^\frac{1}{2}\big)\|\uu\|\|q\|^\frac{1}{2}\|q\|_{H^1(\Omega)}^\frac{1}{2} \notag \\
				& \leq C\big( 1 + \|\rho_1\|_{H^2(\Omega)}^\frac{1}{2}\big)\|\uu\|^\frac{7}{4}\|\nabla \mathbf{A}^{-1}\uu\|^\frac{1}{4} \notag \\
				& \leq \dfrac{\nu_*}{20}\|\uu\|^2 + C\left( 1 + \|\rho_1\|_{H^2(\Omega)}^4\right)\|\uu\|_{\mathbf{V}_\sigma^*}^2.
	\label{eq:stab310}
\end{align}
Collecting the estimates \eqref{eq:stab31}--\eqref{eq:stab310}, we infer from  \eqref{eq:stab3} that
\begin{align}
 \dfrac{1}{2}\td{}{t}\|\uu\|^2_{\mathbf{V}_\sigma^*} + \dfrac{7\nu_*}{20} \|\uu\|^2
& \leq  \dfrac{\alpha}{3}\|\nabla \Delta \phi\|^2+ \dfrac{5}{18}\|\Delta \phi\|^2 + \dfrac{\beta}{6}\|\nabla \rho\|^2
+  C\big( \|\phi\|^2 + \|\rho\|_{V_0^*}^2\big) \notag \\
& \quad  + \mathcal{H}\|\uu\|^2_{\mathbf{V}_\sigma^*} + C\|D \uu_2\| \|\nabla \rho\| \|\uu\|_{\mathbf{V}_\sigma^*} \left[ \ln \left(\dfrac{e\|\nabla \mathbf{A}^{-1}\uu\|_{\mathbf{H}^1(\Omega)}}{\|\uu\|_{\mathbf{V}_\sigma^*}} \right) \right]^\frac{1}{2},
\label{eq:stab3final}
\end{align}
where we set
\begin{align}
\mathcal{H}(t) &:= C\left( 1 + \|\uu_1(t)\|_{\mathbf{V}_\sigma}^2 + \|\uu_2(t)\|_{\mathbf{V}_\sigma}^2 + \| \nabla \phi_1(t) \|_{\LL^\infty(\Omega)}^2 + \| \nabla \phi_2(t) \|_{\LL^\infty(\Omega)}^2 + \|\nabla \Delta \phi_1\|_{\LL^4(\Omega)}^2 \right.\notag\\
&\qquad  \left.  +\|\phi_2\|_{W^{2,4}(\Omega)}^2 + \| \nabla \rho_1(t) \|_{\LL^\infty(\Omega)}^2 + \| \nabla \rho_2(t) \|_{\LL^\infty(\Omega)}^2 + \|\rho_1(t)\|_{H^2(\Omega)}^4 \right). \label{FF1}
\end{align}

\textbf{Step 4.} Collecting \eqref{eq:stab1final},  \eqref{eq:stab2final} and \eqref{eq:stab3final}, we arrive at the differential inequality
\begin{align*}
&\frac{\mathrm{d}\mathcal{Y} }{\mathrm{d}t} + \dfrac{\nu_*}{2} \|\uu\|^2 + \alpha\| \nabla \Delta \phi\|^2+ \|\Delta \phi\|^2
                + \beta \|\nabla \rho\|^2 \notag\\
&\quad \leq \mathcal{H} \mathcal{Y} + C\|D \uu_2\| \|\nabla \rho\| \|\uu\|_{\mathbf{V}_\sigma^*} \left[ \ln \left(\dfrac{e\|\nabla \mathbf{A}^{-1}\uu\|_{\mathbf{H}^1(\Omega)}}{\|\uu\|_{\mathbf{V}_\sigma^*}} \right) \right]^\frac{1}{2},
\end{align*}
where
\begin{align}
			\mathcal{Y}(t) & := \|\uu(t)\|^2_{\mathbf{V}_\sigma^*} + \|\phi(t)\|^2 + \|\rho(t)\|^2_{V_0^*}, \notag
\end{align}
and the function $\mathcal{H}$ is given by \eqref{FF1} with a suitably enlarged $C>0$.

We now analyze the logarithmic term on the right-hand side by using the fact that on any interval $(0,M]$ the function $s \ln \left(\frac{C}{s}\right)$ is increasing provided that $C>eM$.
Recalling that $\|\uu\|_{L^\infty(0,T;\mathbf{H}_\sigma)}$, $\|\phi\|_{L^\infty(0,T;H^2(\Omega))}$ and $\|\rho\|_{L^\infty(0,T;H^1(\Omega))}$ are bounded, we have
\begin{equation}
\|\mathcal{Y}\|_{L^{\infty}\left(0, T\right)}\le K_1,\notag
\end{equation}
where the constant $K_1>0$ depends on norms of the initial data, $\Omega$, $T$, and coefficients of the system. Let $K_2$ be a sufficiently large constant that may depend on $K_1$. Then we deduce that
\begin{align*}
\|D \uu_2\| \|\nabla \rho\| \|\uu\|_{\mathbf{V}_\sigma^*} \left[ \ln \left(\dfrac{e\|\nabla \mathbf{A}^{-1}\uu\|_{\mathbf{H}^1(\Omega)}}{\|\uu\|_{\mathbf{V}_\sigma^*}} \right) \right]^\frac{1}{2}
& \leq  \|D \uu_2\| \|\nabla \rho\| \mathcal{Y}(t)^\frac12 \left[ \ln \left(\dfrac{K_2^\frac12}{\mathcal{Y}(t)^\frac12} \right) \right]^\frac{1}{2}\\
& \leq \frac{\beta}{2} \|\nabla \rho\|^2+ C \|D \uu_2\|^2  \mathcal{Y}(t) \ln \left(\dfrac{K_2}{\mathcal{Y}(t)} \right).
\end{align*}
As a consequence, we obtain
\begin{align}
&\frac{\mathrm{d}\mathcal{Y}}{\mathrm{d}t}  + \dfrac{\nu_*}{2} \|\uu\|^2 + \alpha\| \nabla \Delta \phi\|^2+ \|\Delta \phi\|^2
                + \frac{\beta}{2} \|\nabla \rho\|^2
                \leq \mathcal{H} \mathcal{Y} \ln \left(\dfrac{K_2}{\mathcal{Y}} \right),
                \label{uniA}
\end{align}
where, again, we possibly enlarge $C$ in $\mathcal{H}$. Integrating \eqref{uniA} on $[0,t]\subset [0,T]$, we get
\begin{align}
\mathcal{Y}(t)\leq \mathcal{Y}(0)+ \int_0^t \mathcal{H}(\tau)\mathcal{Y}(\tau)\ln \left(\frac{K_2}{\mathcal{Y}(\tau)}\right)\mathrm{d}\tau,\quad \text{for a.a.}\, t\in [0,T].
\label{uniA1}
\end{align}
Since $\mathcal{H}\in L^1(0,T)$ and $\mathcal{Y}(0)=0$, from \eqref{uniA1} and using the Osgood lemma (see, e.g., \cite[Lemma 3.4]{BCD}), we can conclude that  $\mathcal{Y}(t)=0$ for all $t\in [0,T]$. Hence, the global weak solution to problem \eqref{eq:strongNSCH}--\eqref{eq:bcs} is  unique.

This completes the proof of Theorem \ref{th:weaksolution}.
\hfill $\square$

\begin{remark}
If $\mathcal{Y}(0)>0$, the inequality \eqref{uniA1} yields a continuous dependence estimate with respect to the initial data (cf. \cite{GMT18,He21}). Indeed, choosing a
sufficiently large $K_2$ in order to have
$$
\ln\left[\ln\left(\frac{K_2}{\mathcal{Y}(0)}\right)\right]\geq \int_0^T \mathcal{H}(t)\,\mathrm{d}t,
$$
we infer from \eqref{uniA1} and the Osgood lemma \cite[Lemma 3.4]{BCD} that
$$
\ln\left[\ln\left(\frac{K_2}{\mathcal{Y}(0)}\right)\right]-
\int_0^t \mathcal{H}(\tau)\,\mathrm{d}\tau \leq \ln\left[\ln\left(\frac{K_2}{\mathcal{Y}(t)}\right)\right],\quad \ \forall\, t\in[0,T].
$$
Thus, after taking the double exponential, we find
$$
\mathcal{Y}(t)\le K_2\left(\frac{\mathcal{Y}(0)}{K_2}\right)^{\exp\left(-\int_{0}^{t}\mathcal{H}(\tau)\, \mathrm{d}\tau\right)},\quad \ \forall\, t\in[0,T].
$$
Nevertheless, in the above argument, we should assume that either the initial data for $\rho$ have the same mean value, or take $\mathcal{N}(\rho - \overline{\rho})$ as a test function in Step 2.
\end{remark}

\section{Proof of  Theorem \ref{th:wellposedlocal}} \label{sec:proof3}
In this section, we prove the existence of strong solutions to problem \eqref{eq:strongNSCH}--\eqref{eq:bcs}. Following the approach devised in \cite{GMT18}, we first construct a proper approximation of the initial datum $\rho_0$ (\textcolor{black}{which is indeed not necessary for the logarithmic potential \eqref{eq:mixentropy} when $d=2$, as pointed out in \cite{HeWu}}). Then, using the same approximating scheme as in Section \ref{sec:proof1}, we derive higher-order uniform bounds
which allow us to pass to the limit with respect to the approximation parameters.

\subsection{The approximating scheme} \label{subs:initialdata}

\textbf{Approximating $\rho_0$}. Recalling \cite[Section 4]{GMT18} (see also \cite{GGM}), we consider the family of cutoff functions $h_k:\mathbb{R} \to \mathbb{R}$, $k \in \mathbb{N}$, defined by
\begin{equation*} \label{eq:cutoff}
		h_k(s) := \begin{cases}
			-k, & \quad s < -k, \\
			s, & \quad |s| \leq k, \\
			k, & \quad s > k.
		\end{cases}
\end{equation*}
Observe that $h_k$ is globally Lipschitz continuous. For $\widehat{\psi}_0 := -\Delta \rho_0 + \widehat{S}'_\rho(\rho_0) \in H^1(\Omega)$, we have
$\widehat{\psi}_{0,k} := h_k \circ\widehat{\psi}_0 \in H^1(\Omega)$ for any $k \in \mathbb{N}$. Moreover, the weak chain rule implies
$\nabla \widehat{\psi}_{0,k} = \nabla {\psi}_0 \cdot \chi_{[-k,k]}(\widehat{\psi}_0)$, and thus
\begin{equation} \label{eq:approxinitial}
		\|\widehat{\psi}_{0,k}\|_{H^1(\Omega)} \leq \|\widehat{\psi}_{0}\|_{H^1(\Omega)}.
\end{equation}
	We now approximate the initial condition $\rho_0 \in H^2(\Omega)$. For any $k \in \mathbb{N}$, consider the following elliptic problem
\begin{equation*} \label{eq:neumanninitial}
\begin{cases}
	-\beta \Delta \rho_{0,k} + \widehat{S}'_\rho(\rho_{0,k}) = \widehat{\psi}_{0,k} & \quad \text{ a.e. in } \Omega,\\
	\partial_\mathbf{n}\rho_{0,k} = 0, & \quad \text{ a.e. on } \partial\Omega,
\end{cases}
\end{equation*}
which admits a unique solution $\rho_{0,k} \in H^2_N(\Omega)$ satisfying
\begin{equation} \label{eq:estimateinitdata0}
		\|\rho_{0,k}\|_{H^2(\Omega)} + \|\widehat{S}_\rho'(\rho_{0,k})\| \leq C\big( 1 + \|\widehat{\psi}_{0,k}\|\big)
\leq  C\big( 1 + \|\widehat{\psi}_{0}\|\big).
\end{equation}
Besides, owing to the strong convergence $\widehat{\psi}_{0,k} \to \widehat{\psi}_0$ in $L^2(\Omega)$, it holds $\rho_{0,k} \to \rho_0$ in $H^1(\Omega)$, see \cite[Lemma A.1]{GMT18}. Then there exists some $m_1 \in (0,1/2)$ independent of $k$ and some $k^* \in \mathbb{N}$ such that for every $k > k^*$
\begin{equation} \label{eq:estimateinitdata2}
		\|\rho_{0,k}\|_{H^1(\Omega)} \leq 1 + \|\rho_0\|_{H^1(\Omega)}, \qquad m_1\leq \overline{\rho_{0,k}} \leq 1-m_1.
\end{equation}
Moreover, a regularity estimate (see \cite[Theorem A.2]{GMT18}) implies that
	\[
	\|\widehat{S}'_\rho(\rho_{0,k})\|_{L^\infty(\Omega)} \leq \|\widehat{\psi}_{0,k}\|_{L^\infty(\Omega)} \leq k,
	\]
and thus the approximated initial data sequence is strictly separated from the pure states $0,1$, namely, there exists $\textcolor{black}{\widetilde{\eta} = \widetilde{\eta}(k)\ \in (0,1/2)}$ such that
	\[
	\textcolor{black}{\widetilde{\eta} \leq \|\rho_{0,k}\|_{L^\infty(\Omega)} \leq 1-\widetilde{\eta}.}
	\]
Hence, it holds $\widehat{S}'_\rho(\rho_{0,k}) \in H^1(\Omega)$ and furthermore $\rho_{0,k} \in H^3(\Omega)$. \medskip

\textbf{The Galerkin scheme}. For the approximating system, we carry over the notation of Subsection \ref{sub:galerkin} with the additional modification of the initial datum $\rho_{0,k}$. However, for ease of notation, instead of $f^{n,\varepsilon}_{\omega, k}$ we only write $f^{n,\varepsilon}_\omega$. We look for functions
\begin{align*}
\phi^{n, \varepsilon}_\omega(t) = \sum_{i = 1}^n a_{i}(t)w_i, & \quad \rho^{n, \varepsilon}_\omega(t) = \sum_{i = 1}^n b_{i}(t)w_i,\\ \mu^{n, \varepsilon}_\omega(t) = \sum_{i = 1}^n c_{i}(t)w_i, & \quad \psi^{n, \varepsilon}_\omega(t) = \sum_{i = 1}^n d_{i}(t)w_i, \\ \uu^{n, \varepsilon}_\omega(t) =  \sum_{i = 1}^n e_{i}(t)\mathbf{w}_i,&
\end{align*}
that solve the following problem
\begin{equation} \label{eq:weakreg2}
\begin{cases}
\left\langle \partial_t\mathbf{u}^{n, \varepsilon}_\omega, \mathbf{v} \right\rangle_{\mathbf{V}^*_\sigma, \mathbf{V}_\sigma} + ((\mathbf{u}^{n, \varepsilon}_\omega \cdot \nabla)\mathbf{u}^{n, \varepsilon}_\omega,\mathbf{v}) + (\nu(\phi^{n, \varepsilon}_\omega,\rho^{n, \varepsilon}_\omega) D\mathbf{u}^{n, \varepsilon}_\omega, D\mathbf{v}) &\\
\qquad \qquad = (\mu^{n, \varepsilon}_\omega\nabla \phi^{n, \varepsilon}_\omega, \mathbf{v}) + (\psi^{n, \varepsilon}_\omega\nabla \rho^{n, \varepsilon}_\omega, \mathbf{v}) & \quad \forall \: \mathbf{v} \in \mathbf{W}_n, \aevv,\\
\left\langle \partial_t\phi^{n, \varepsilon}_\omega, v \right\rangle_{V^*,V} + (\uu^{n, \varepsilon}_\omega \cdot \nabla \phi^{n, \varepsilon}_\omega, v) + (\nabla \mu^{n, \varepsilon}_\omega, \nabla v) = 0 & \quad \forall \: v \in W_n, \aevv,\\
\mu^{n, \varepsilon}_\omega = \Pi_n \Big(\alpha \Delta^2 \phi^{n, \varepsilon}_\omega  -\Delta \phi^{n, \varepsilon}_\omega +  S'_{\phi}(\phi^{n, \varepsilon}_\omega) + \theta\nabla \cdot (\rho^{n, \varepsilon}_\omega \nabla \phi^{n, \varepsilon}_\omega)\Big) &\\
\qquad \quad - \Pi_n \Big(\omega \nabla \cdot\left( |\nabla \phi^{n, \varepsilon}_\omega|^2\nabla \phi^{n, \varepsilon}_\omega \right) \Big)& \quad \text{a.e. in } \Omega \times (0,T),\\
\left\langle \partial_t\rho^{n, \varepsilon}_\omega, v \right\rangle_{V^*,V} + (\uu^{n, \varepsilon}_\omega \cdot \nabla \rho^{n, \varepsilon}_\omega, v) + (\nabla \psi^{n, \varepsilon}_\omega, \nabla v) = 0& \quad \forall \: v \in W_n, \aevv,\\
\psi^{n, \varepsilon}_\omega = \Pi_n \Big( -\beta \Delta \rho^{n, \varepsilon}_\omega + S'_{\rho,\varepsilon}(\rho^{n, \varepsilon}_\omega) - \dfrac{\theta}{2}|\nabla \phi^{n, \varepsilon}_\omega|^2 \Big) & \quad \text{a.e. in } \Omega \times (0,T), \\
\uu^{n, \varepsilon}_\omega(\cdot,0) = P_n(\uu_0) =: \uu_{0}^n & \quad \text{in } \Omega,\\
\phi^{n, \varepsilon}_\omega(\cdot, 0) = \Pi_n(\phi_{0}) =: \phi_{0}^n, \quad
\rho^{n, \varepsilon}_\omega(\cdot, 0) = \Pi_n(\rho_{0,k}) =: \rho_{0,k}^n & \quad \text{in } \Omega.
\end{cases}
\end{equation}
Since the singular and regular potentials coincide on compact subsets of the interval $(0,1)$, we notice that the following bound holds
\begin{equation} \label{eq:estimateinitdata}
		\|-\beta \Delta \rho_{0,k} + \widehat{S}'_{\rho,\varepsilon}(\rho_{0,k}) \|_{H^1(\Omega)}=\|-\beta \Delta \rho_{0,k} + \widehat{S}'_{\rho}(\rho_{0,k}) \|_{H^1(\Omega)} \leq \|	\widehat{\psi}_0 \|_{H^1(\Omega)},
\end{equation}
for sufficiently small $\varepsilon\in (0,\epsilon_3]$, where $\epsilon_3=\min\{\frac12 \widetilde{\eta}(k), \epsilon_1, \epsilon_2\}$. For the definition of $\epsilon_2$, we recall the proof of Lemma \ref{prop:enebound1}.

 Let us clarify how the parameters work. We fix $\omega\in (0, 1]$, then for any $k > k^*$ we take $\varepsilon \in (0,\epsilon_3(k))$ so that \eqref{eq:convexityapprox}, \eqref{eq:estimateinitdata2} and \eqref{eq:estimateinitdata} hold. Since $\rho^n_{0,k} \to \rho_{0,k}$ in $H^3(\Omega)$ as $n \to +\infty$ and thus in $L^\infty(\Omega)$, there exist $m_2 \in (0,m_1)$ independent of $k$ and some $n^* = n^*(k) \in \mathbb{N}$ such that
	\[
	m_2 < \overline{\rho_{0,k}^n}<1-m_2,\quad
\textcolor{black}{\frac{\widetilde{\eta}}{2} \leq \|\rho_{0,k}^n\|_{L^\infty(\Omega)} \leq 1-\frac{\widetilde{\eta}}{2},}
\qquad \forall\, n > n^*.
	\]

We can now apply the Cauchy--Lipschitz theorem to the Cauchy problem for the above system of $5n$ ordinary differential equations in the unknowns $a_i(t)$, $b_i(t)$, $c_i(t)$, $d_i(t)$ and $e_i(t)$. This gives
\begin{proposition}\label{localODEstr}
Let $\omega$, $k, \varepsilon$ be fixed as specified above.
For any positive integer $n > n^*(k)$, there exists $T_n > 0$ such that problem \eqref{eq:weakrega} admits a unique local solution $(\uu_\omega^{n,\varepsilon},\phi_\omega^{n,\varepsilon}, \rho_\omega^{n,\varepsilon}, \mu_\omega^{n,\varepsilon}, \psi_\omega^{n,\varepsilon})$ in $[0,T_n]$, which is given by the functions $a_i, b_i, c_i, d_i, e_i\in \mathcal{C}^1([0,T_n])$, $i=1, \dots, n$.
\end{proposition}
\subsection{Uniform estimates} \label{sub:unifestimates}
We now show uniform estimates with respect to the approximating parameters $\omega$, $k$, $\varepsilon$ and $n$. \medskip

\textbf{Lower order estimates}.
The presence of the additional parameter $k$ does not introduce any technical difficulty. Mimicking the proof of Lemma \ref{prop:enebound1} and using \eqref{Srho1} we deduce that
\begin{lemma} \label{prop:enebound4}
For every $t \in (0,T_n]$, it holds
\begin{align*}
&\ene(\mathbf{u}^{n, \varepsilon}_\omega(t), \phi^{n, \varepsilon}_\omega(t),\rho^{n, \varepsilon}_\omega(t)) + \ii{0}{t}{\|\sqrt{\nu(\phi^{n, \varepsilon}_\omega(\tau),\rho^{n, \varepsilon}_\omega(\tau))}D \mathbf{u}^{n, \varepsilon}_\omega(t)\|^2 + \| \nabla \mu^{n, \varepsilon}_\omega(\tau) \|^2 + \| \nabla \psi^{n, \varepsilon}_\omega(\tau)\|^2 }{\tau} \leq C_8,
\end{align*}
and
\begin{align*}
\ene(\mathbf{u}^{n, \varepsilon}_\omega(t),\phi^{n, \varepsilon}_\omega(t),\rho^{n, \varepsilon}_\omega(t))
&\geq \dfrac{1}{2}\|\uu_{\omega}^{n,\varepsilon}(t)\|^2
+ \dfrac{\alpha}{4} \|\Delta \phi^{n, \varepsilon}_\omega(t)\|^2
+ \dfrac{1}{2} \|\nabla \phi^{n, \varepsilon}_\omega(t)\|^2
+ \dfrac{\beta}{2}\|\nabla \rho^{n, \varepsilon}_\omega(t)\|^2
\\
&\quad + \frac{c_3}{2} \|\phi^{n, \varepsilon}_\omega(t)\|^4_{L^4(\Omega)}
+ \dfrac{\omega}{8} \|\nabla\phi^{n, \varepsilon}_\omega(t)\|^4_{\mathbf{L}^4(\Omega)} -C_9,
\end{align*}
where the constant $C_8, C_9>0$ are independent of $n$, $\omega$, $k$ and $\varepsilon$.
\end{lemma}

Then, we can follow line by line all the proofs of Lemmas \ref{prop:ub1}--\ref{prop:ub4} to derive uniform estimates for the approximate solutions with respect to $\omega$, $n$, $k$ and $\varepsilon$. In particular, we have $T_n=T$ and
	\begin{align*}
        \uu^{n, \varepsilon}_\omega & \text{ is uniformly bounded in } L^\infty(0,T; \mathbf{H}_\sigma) \cap L^2(0,T; \mathbf{V}_\sigma) \cap W^{1, \frac{4}{d}}([0,T];\mathbf{V}_\sigma^*), \\
		\phi^{n, \varepsilon}_\omega & \text{ is uniformly bounded in } L^\infty(0,T; H^2(\Omega)) \cap L^2(0,T; H^4(\Omega)) \cap H^1(0,T;(H^1(\Omega))^*), \\
		\rho^{n, \varepsilon}_\omega & \text{ is uniformly bounded in } L^\infty(0,T; H^1(\Omega)) \cap L^4(0,T; H^2(\Omega)) \cap H^1(0,T;(H^1(\Omega))^*), \\
		\mu^{n, \varepsilon}_\omega & \text{ is uniformly bounded in } L^2(0,T; H^1(\Omega)),  \\
		\psi^{n, \varepsilon}_\omega & \text{ is uniformly bounded in } L^2(0,T; H^1(\Omega)).	
	\end{align*}
The existence of strong solutions depends on higher-order estimates. The situation is different according to the spatial dimension.
\medskip

\textbf{Higher-order estimates in two dimensions}. We have
\begin{lemma} \label{prop:ub5}
Let $d=2$. The sequences $\{\mu^{n, \varepsilon}_\omega\}$ and $\{\psi^{n, \varepsilon}_\omega\}$ are uniformly bounded in $L^\infty(0,T;H^1(\Omega))$. The sequence $\{\uu^n_\omega\}$ is uniformly bounded in $L^\infty(0,T;\mathbf{V}_\sigma) \cap L^2(0,T; \mathbf{W}_\sigma)$.
\end{lemma}
\begin{proof}
The proof consists of several steps.

		\textbf{Step 1.} As in \cite{GMT18}, taking $\partial_t\psi^{n, \varepsilon}_\omega$ as a test function in \eqref{eq:weakreg2}$_2$  yields
		\begin{equation} \label{eq:step11}
			\dfrac{1}{2}\td{}{t}\|\nabla \psi^{n, \varepsilon}_\omega\|^2
+ (\partial_t \psi^{n, \varepsilon}_\omega, \partial_t\rho^{n, \varepsilon}_\omega)
+ (\uu^{n, \varepsilon}_\omega \cdot \nabla \rho^{n, \varepsilon}_\omega, \partial_t \psi^{n, \varepsilon}_\omega) = 0.
		\end{equation}
Then, recalling \eqref{eq:convexityapprox} and the fact that $\overline{\partial_t\rho^{n, \varepsilon}_\omega}(t) = 0$, we arrive at
\begin{align}
(\partial_t \psi^{n, \varepsilon}_\omega, \partial_t\rho^{n, \varepsilon}_\omega)
& = \beta(-\Delta \partial_t\rho^{n, \varepsilon}_\omega, \partial_t\rho^{n, \varepsilon}_\omega) + (S''_{\rho,\varepsilon}(\rho^{n, \varepsilon}_\omega)\partial_t\rho^{n, \varepsilon}_\omega,\partial_t\rho^{n, \varepsilon}_\omega) - \dfrac{\theta}{2}(\partial_t|\nabla \phi^{n, \varepsilon}_\omega|^2, \partial_t\rho^{n, \varepsilon}_\omega) \notag \\
				& \geq \beta\| \nabla \partial_t\rho^{n, \varepsilon}_\omega\|^2 -C\|\partial_t\rho^{n, \varepsilon}_\omega\|^2 - \theta|(\nabla\partial_t\phi^{n, \varepsilon}_\omega \cdot \nabla\phi^{n, \varepsilon}_\omega, \partial_t\rho^{n, \varepsilon}_\omega)| \notag \\
				& \geq \beta\| \nabla \partial_t\rho^{n, \varepsilon}_\omega\|^2 -C\|\partial_t\rho^{n, \varepsilon}_\omega\|^2 - C\|\nabla \phi^{n, \varepsilon}_\omega\|_{\LL^4(\Omega)}\|\nabla \partial_t\phi^{n, \varepsilon}_\omega\|\|\partial_t\rho^{n, \varepsilon}_\omega\|_{L^4(\Omega)} \notag \\
				& \geq \beta\| \nabla \partial_t\rho^{n, \varepsilon}_\omega\|^2 -C\|\partial_t\rho^{n, \varepsilon}_\omega\|^2 - \dfrac{1}{16}\|\nabla \partial_t\phi^{n, \varepsilon}_\omega\|^2 - C\|\partial_t\rho^{n, \varepsilon}_\omega\|_{L^4(\Omega)}^2 \notag \\
				& \geq \left( \beta - \dfrac{\beta}{12} \right)\| \nabla \partial_t\rho^{n, \varepsilon}_\omega\|^2 - \dfrac{1}{16}\|\nabla \partial_t\phi^{n, \varepsilon}_\omega\|^2 - C\|\partial_t\rho^{n, \varepsilon}_\omega\|^2 \notag \\
				& \geq \left( \beta - \dfrac{\beta}{6} \right)\| \nabla \partial_t\rho^{n, \varepsilon}_\omega\|^2 - \dfrac{1}{16}\|\nabla \partial_t\phi^{n, \varepsilon}_\omega\|^2 - C\|\partial_t\rho^{n, \varepsilon}_\omega\|_{V_0^*}^2,
\label{eq:step12}
\end{align}
whereas
\begin{align}
(\uu^{n, \varepsilon}_\omega \cdot \nabla \rho^{n, \varepsilon}_\omega, \partial_t \psi^{n, \varepsilon}_\omega) = \td{}{t}	(\uu^{n, \varepsilon}_\omega \cdot \nabla \rho^{n, \varepsilon}_\omega, \psi^{n, \varepsilon}_\omega)
- 	(\partial_t\uu^{n, \varepsilon}_\omega \cdot \nabla \rho^{n, \varepsilon}_\omega, \psi^{n, \varepsilon}_\omega)
- 	(\uu^{n, \varepsilon}_\omega \cdot \nabla \partial_t\rho^{n, \varepsilon}_\omega,  \psi^{n, \varepsilon}_\omega).
\label{eq:step13}
\end{align}
Moreover, recalling \eqref{eq:potentialest}, we have
\begin{align}
(\uu^{n, \varepsilon}_\omega \cdot \nabla\partial_t \rho^{n, \varepsilon}_\omega, \psi^{n, \varepsilon}_\omega)
& = \textcolor{black}{(\uu^{n, \varepsilon}_\omega \cdot \nabla\partial_t \rho^{n, \varepsilon}_\omega, \psi^{n, \varepsilon}_\omega-\overline{\psi^{n, \varepsilon}_\omega} )}\notag \\
& \leq \|\uu^{n, \varepsilon}_\omega\|_{\LL^3(\Omega)}\|\nabla \partial_t\rho^{n, \varepsilon}_\omega\|\|\psi^{n, \varepsilon}_\omega -\overline{\psi^{n, \varepsilon}_\omega}\|_{L^6(\Omega)}\notag  \\
& \leq \dfrac{\beta}{12}\| \nabla \partial_t\rho^{n, \varepsilon}_\omega\|^2 + C\|\uu^{n, \varepsilon}_\omega\|_{\LL^3(\Omega)}^2 \|\nabla \psi^{n, \varepsilon}_\omega\|^2.
\label{eq:step14}
\end{align}
Thus, from \eqref{eq:step11} and the estimates \eqref{eq:step12}--\eqref{eq:step14}, we infer that
\begin{align}
&\td{}{t}\left( \dfrac{1}{2}\|\nabla \psi^{n, \varepsilon}_\omega\|^2 +  (\uu^{n, \varepsilon}_\omega \cdot \nabla \rho^{n, \varepsilon}_\omega, \psi^{n, \varepsilon}_\omega) \right)  +  \dfrac{3\beta}{4} \| \nabla \partial_t\rho^{n, \varepsilon}_\omega\|^2
- \dfrac{1}{16}\| \nabla \partial_t\phi^{n, \varepsilon}_\omega\|^2 \notag \\
&\quad \leq (\partial_t\uu^{n, \varepsilon}_\omega \cdot \nabla \rho^{n, \varepsilon}_\omega, \psi^{n, \varepsilon}_\omega)
+\ \textcolor{black}{C\big( \|\partial_t \rho^{n, \varepsilon}_\omega\|_{V_0^*}^2 + \|\uu^{n, \varepsilon}_\omega\|_{\LL^3(\Omega)}^2
 \|\nabla \psi^{n, \varepsilon}_\omega\|^2 \big). }
\label{eq:step1final}
\end{align}

\textbf{Step 2.} Analogously, we take $\partial_t\mu^{n, \varepsilon}_\omega$ as a test function for \eqref{eq:weakreg2}$_2$ and obtain
\begin{equation} \label{eq:step21}
			\dfrac{1}{2}\td{}{t}\|\nabla  \mu^{n, \varepsilon}_\omega\|^2
+ (\partial_t \mu^{n, \varepsilon}_\omega, \partial_t\phi^{n, \varepsilon}_\omega)
+ (\uu^{n, \varepsilon}_\omega \cdot \nabla \phi^{n, \varepsilon}_\omega, \partial_t \mu^{n, \varepsilon}_\omega) = 0.
\end{equation}
Recalling that $S''_\phi(s) \geq -c_0$ for every $s \in \mathbb{R}$ and using
\begin{align*}
-\omega(\partial_t \nabla \cdot(|\nabla \phi^{n, \varepsilon}_\omega|^2\nabla \phi^{n, \varepsilon}_\omega), \partial_t \phi^{n, \varepsilon}_\omega)
& = \omega\left(\partial_t \left(|\nabla \phi^{n, \varepsilon}_\omega|^2\nabla \phi^{n, \varepsilon}_\omega\right), \nabla \partial_t \phi^{n, \varepsilon}_\omega \right) \notag\\
&= \omega \int_\Omega |\nabla \phi^{n, \varepsilon}_\omega|^2|\nabla \partial_t \phi^{n, \varepsilon}_\omega|^2\,\mathrm{d}x
+ 2\omega \| \nabla \phi^{n, \varepsilon}_\omega \cdot \nabla \partial_t \phi^{n, \varepsilon}_\omega \|^2 \geq 0,
\end{align*}
we get
\begin{align}
(\partial_t \mu^{n, \varepsilon}_\omega, \partial_t\phi^{n, \varepsilon}_\omega)
& \geq  \alpha \|\Delta \partial_t \phi^{n, \varepsilon}_\omega\|^2 + \|\nabla \partial_t \phi^{n, \varepsilon}_\omega\|^2  - c_0 \|\partial_t \phi^{n, \varepsilon}_\omega\|^2
+ \theta\big(\partial_t(\rho^{n, \varepsilon}_\omega\Delta\phi^{n, \varepsilon}_\omega), \partial_t \phi^{n, \varepsilon}_\omega\big) \notag\\
&\quad +  \theta\big(\partial_t(\nabla \rho^{n, \varepsilon}_\omega \cdot \nabla \phi^{n, \varepsilon}_\omega), \partial_t \phi^{n, \varepsilon}_\omega\big).
\label{eq:step22}
\end{align}
The last two terms on the right-hand side of \eqref{eq:step22} can be controlled as follows
\begin{align}
& \big|\theta\big(\partial_t(\rho^{n, \varepsilon}_\omega\Delta\phi^{n, \varepsilon}_\omega), \partial_t \phi^{n, \varepsilon}_\omega\big)\big| \notag \\
& \quad \leq \theta|(\partial_t\rho^{n, \varepsilon}_\omega\Delta\phi^{n, \varepsilon}_\omega,\partial_t \phi^{n, \varepsilon}_\omega)| + \theta|(\rho^{n, \varepsilon}_\omega\Delta\partial_t\phi^{n, \varepsilon}_\omega,\partial_t \phi^{n, \varepsilon}_\omega)| \notag \\
				& \quad \leq C\|\partial_t\rho^{n, \varepsilon}_\omega\|_{L^4(\Omega)}\|\Delta \phi^{n, \varepsilon}_\omega\|\|\partial_t\phi^{n, \varepsilon}_\omega\|_{L^4(\Omega)} + C\|\rho^{n, \varepsilon}_\omega\|_{L^4(\Omega)}\|\Delta \partial_t\phi^{n, \varepsilon}_\omega\|\|\partial_t\phi^{n, \varepsilon}_\omega\|_{L^4(\Omega)} \notag \\
				& \quad \leq \dfrac{\beta}{12}\|\nabla \partial_t\rho^{n, \varepsilon}_\omega\|^2
+ \dfrac{1}{8}\|\nabla \partial_t\phi^{n, \varepsilon}_\omega\|^2
+ C\|\partial_t \rho^{n, \varepsilon}_\omega\|^2
+ C\|\partial_t \phi^{n, \varepsilon}_\omega\|^2
+ \dfrac{\alpha}{4}\|\Delta \partial_t \phi^{n, \varepsilon}_\omega\|^2 \notag  \\
				& \quad \leq \dfrac{\beta}{6}\|\nabla \partial_t\rho^{n, \varepsilon}_\omega\|^2
+ \dfrac{3}{16}\|\nabla \partial_t\phi^{n, \varepsilon}_\omega\|^2
+ \dfrac{\alpha}{4}\|\Delta \partial_t \phi^{n, \varepsilon}_\omega\|^2
+ C\|\partial_t \rho^{n, \varepsilon}_\omega\|^2_{V_0^*} + C\|\partial_t \phi^{n, \varepsilon}_\omega\|^2_{V_0^*},
\label{eq:step23}
\end{align}
and
\begin{align}
	&			\big|\theta\big(\partial_t(\nabla \rho^{n, \varepsilon}_\omega \cdot \nabla \phi^{n, \varepsilon}_\omega), \partial_t \phi^{n, \varepsilon}_\omega\big)\big| \notag \\
                & \quad \leq \theta|(\nabla\partial_t\rho^{n, \varepsilon}_\omega \cdot \nabla \phi^{n, \varepsilon}_\omega,\partial_t \phi^{n, \varepsilon}_\omega)| + \theta|(\nabla\rho^{n, \varepsilon}_\omega \cdot \nabla\partial_t \phi^{n, \varepsilon}_\omega,\partial_t \phi^{n, \varepsilon}_\omega)| \notag \\
				& \quad \leq C\|\nabla \partial_t\rho^{n, \varepsilon}_\omega\|\|\nabla \phi^{n, \varepsilon}_\omega\|_{\LL^4(\Omega)}\|\partial_t\phi^{n, \varepsilon}_\omega\|_{L^4(\Omega)} + C\|\nabla \rho^{n, \varepsilon}_\omega\|\|\nabla \partial_t\phi^{n, \varepsilon}_\omega\|_{\LL^4(\Omega)}\|\partial_t\phi^{n, \varepsilon}_\omega\|_{L^4(\Omega)} \notag \\
				& \quad \leq \dfrac{\beta}{12}\|\nabla \partial_t\rho^{n, \varepsilon}_\omega\|^2
                + \dfrac{1}{8}\|\nabla \partial_t\phi^{n, \varepsilon}_\omega\|^2 + C\|\partial_t \phi^{n, \varepsilon}_\omega\|^2
                + C\|\nabla \partial_t\phi^{n, \varepsilon}_\omega\|_{\LL^4(\Omega)}^2 \notag \\
				& \quad \leq \dfrac{\beta}{12}\|\nabla \partial_t\rho^{n, \varepsilon}_\omega\|^2
                + \dfrac{1}{8}\|\nabla \partial_t\phi^{n, \varepsilon}_\omega\|^2 + C\|\partial_t \phi^{n, \varepsilon}_\omega\|^2
                + C\|\partial_t \phi^{n, \varepsilon}_\omega\|^\frac{1}{2}\|\Delta \partial_t \phi^{n, \varepsilon}_\omega\|^\frac{3}{2} \notag \\
				& \quad \leq \dfrac{\beta}{12}\|\nabla \partial_t\rho^{n, \varepsilon}_\omega\|^2
                + \dfrac{\alpha}{4}\|\Delta \partial_t \phi^{n, \varepsilon}_\omega\|^2
                + \dfrac{3}{16}\|\nabla \partial_t\phi^{n, \varepsilon}_\omega\|^2 + C\|\partial_t \phi^{n, \varepsilon}_\omega\|^2_{V_0^*}.
			\label{eq:step24}
\end{align}
Furthermore, the remaining term can be treated as in Step 1, namely,
\begin{align}
(\uu^{n, \varepsilon}_\omega \cdot \nabla \phi^{n, \varepsilon}_\omega, \partial_t \mu^{n, \varepsilon}_\omega)
= \td{}{t}	(\uu^{n, \varepsilon}_\omega \cdot \nabla \phi^{n, \varepsilon}_\omega, \mu^{n, \varepsilon}_\omega)
- 	(\partial_t\uu^{n, \varepsilon}_\omega \cdot \nabla \phi^{n, \varepsilon}_\omega, \mu^n_\omega)
- 	(\uu^{n, \varepsilon}_\omega \cdot \partial_t\nabla \phi^{n, \varepsilon}_\omega, \mu^{n, \varepsilon}_\omega).
\label{eq:step25}
\end{align}
Then, recalling again \eqref{eq:potentialest}, we get
\begin{align}
(\uu^{n, \varepsilon}_\omega \cdot \nabla\partial_t \phi^{n, \varepsilon}_\omega, \mu^{n, \varepsilon}_\omega)
&= \textcolor{black}{(\uu^{n, \varepsilon}_\omega \cdot \nabla\partial_t \phi^{n, \varepsilon}_\omega, \mu^{n, \varepsilon}_\omega-\overline{\mu^{n, \varepsilon}_\omega}) }\notag\\
& \leq \|\uu^{n, \varepsilon}_\omega\|_{\LL^3(\Omega)}\|\nabla \partial_t\phi^{n, \varepsilon}_\omega\|\|\mu^{n, \varepsilon}_\omega-\overline{\mu^{n, \varepsilon}_\omega}\|_{L^6(\Omega)} \notag \\
& \leq \dfrac{1}{16}\| \nabla \partial_t\phi^{n, \varepsilon}_\omega\|^2
+ C\|\uu^{n, \varepsilon}_\omega\|_{\LL^3(\Omega)}^2 \|\nabla \mu^{n, \varepsilon}_\omega\|^2.	
\label{eq:step26}	
\end{align}
Collecting the estimates \eqref{eq:step22}--\eqref{eq:step26}, we infer from \eqref{eq:step21} that
\begin{align}
& \td{}{t}\left( \dfrac{1}{2}\|\nabla  \mu^{n, \varepsilon}_\omega\|^2
+  (\uu^{n, \varepsilon}_\omega \cdot \nabla \phi^{n, \varepsilon}_\omega, \mu^{n, \varepsilon}_\omega) \right)
+ \dfrac{\alpha}{2} \|\Delta \partial_t \phi^{n, \varepsilon}_\omega\|^2
+ \dfrac{9}{16} \| \nabla \partial_t\phi^{n, \varepsilon}_\omega\|^2
- \dfrac{\beta}{4}\| \nabla \partial_t\rho^{n, \varepsilon}_\omega\|^2
\notag  \\
&\quad \leq (\partial_t\uu^{n, \varepsilon}_\omega \cdot \nabla \phi^{n, \varepsilon}_\omega, \mu^{n, \varepsilon}_\omega)
+ \textcolor{black}{	C\big( \|\partial_t \phi^{n, \varepsilon}_\omega\|_{V_0^*}^2 + \|\uu^{n, \varepsilon}_\omega\|_{\LL^3(\Omega)}^2 \|\nabla \mu^{n, \varepsilon}_\omega\|^2 \big). }
\label{eq:step2final}
\end{align}

\textbf{Step 3.} Testing  \eqref{eq:weakreg2}$_1$  by \textcolor{black}{ $\partial_t\uu^{n,\varepsilon}_\omega$,} we get
\begin{align}
&	\|\partial_t \uu^{n, \varepsilon}_\omega\|^2 + \big((\uu^{n, \varepsilon}_\omega \cdot \nabla)\uu^{n, \varepsilon}_\omega, \partial_t \uu^{n, \varepsilon}_\omega\big) + \big(\nabla \cdot(\nu(\phi^{n, \varepsilon}_\omega, \rho^{n, \varepsilon}_\omega) D\uu^{n, \varepsilon}_\omega), \partial_t \uu^{n, \varepsilon}_\omega\big) \notag\\
&\quad  =  (\mu^{n, \varepsilon}_\omega\nabla\phi^{n, \varepsilon}_\omega, \partial_t \uu^{n, \varepsilon}_\omega)
+(\psi^{n, \varepsilon}_\omega\nabla\rho^{n, \varepsilon}_\omega, \partial_t \uu^{n, \varepsilon}_\omega) .
\label{eq:step31}
\end{align}
The following estimates can be derived arguing as in \cite[Section 4]{GMT18}:
\begin{align}
	\big|\big((\uu^{n, \varepsilon}_\omega \cdot \nabla)\uu^{n, \varepsilon}_\omega, \partial_t \uu^{n, \varepsilon}_\omega\big)\big|
& \leq \|\uu^{n, \varepsilon}_\omega\|_{\LL^4(\Omega)}\|\nabla \uu^{n, \varepsilon}_\omega\|_{\LL^4(\Omega)}\|\partial_t \uu^{n, \varepsilon}_\omega\| \notag \\
	& \leq C\|\nabla \uu^{n, \varepsilon}_\omega\|\|\mathbf{A}\uu^{n, \varepsilon}_\omega\|^\frac{1}{2}\|\partial_t \uu^{n, \varepsilon}_\omega\| \notag \\
	& \leq \dfrac{1}{10}\|\partial_t \uu^{n, \varepsilon}_\omega\|^2 + C\left( \|\mathbf{A}\uu^{n, \varepsilon}_\omega\|^2 + \|\nabla \uu^{n, \varepsilon}_\omega\|^4 \right),
\label{eq:step32}
\end{align}
\begin{align}
			&	\big|\big(\nabla \cdot(\nu(\phi^{n, \varepsilon}_\omega, \rho^{n, \varepsilon}_\omega) D\uu^{n, \varepsilon}_\omega), \partial_t\uu^{n, \varepsilon}_\omega\big)\big| \notag\\
            &\quad  = \left|\dfrac{1}{2}\big(\nu(\phi^{n, \varepsilon}_\omega, \rho^{n, \varepsilon}_\omega)\Delta\uu^{n, \varepsilon}_\omega, \partial_t\uu^{n, \varepsilon}_\omega\big)
            + \big(D\uu^{n, \varepsilon}_\omega\nabla\nu(\phi^{n, \varepsilon}_\omega, \rho^{n, \varepsilon}_\omega), \partial_t \uu^{n, \varepsilon}_\omega \big) \right| \notag \\
				&\quad  \leq C\|\mathbf{A} \uu^{n, \varepsilon}_\omega\|\|\partial_t \uu^{n, \varepsilon}_\omega\| + C\left(\|\nabla \phi^{n, \varepsilon}_\omega\|_{\LL^4(\Omega)} + \|\nabla \rho^{n, \varepsilon}_\omega\|_{\LL^4(\Omega)}\right)\|D\uu^{n, \varepsilon}_\omega\|_{\LL^4(\Omega)}\|\partial_t \uu^{n, \varepsilon}_\omega\| \notag \\
				&\quad  \leq \dfrac{1}{5}\|\partial_t \uu^{n, \varepsilon}_\omega\|^2 + C \|\mathbf{A}\uu^{n, \varepsilon}_\omega\|^2 + C\big(1 + \| \rho^{n, \varepsilon}_\omega\|_{H^2(\Omega)}\big)\|\nabla\uu^{n, \varepsilon}_\omega\|\|\mathbf{A}\uu^{n, \varepsilon}_\omega\| \notag \\
				&\quad  \leq \dfrac{1}{5}\|\partial_t \uu^{n, \varepsilon}_\omega\|^2 + C\|\mathbf{A}\uu^{n, \varepsilon}_\omega\|^2
 + C\big(1+ \| \rho^{n, \varepsilon}_\omega\|_{H^2(\Omega)}^2\big)\|\nabla\uu^{n, \varepsilon}_\omega\|^2.
\label{eq:step322}
\end{align}
Concerning the Korteweg forces, we have
\begin{align}
	& (\mu^{n, \varepsilon}_\omega\nabla\phi^{n, \varepsilon}_\omega, \partial_t \uu^{n, \varepsilon}_\omega)
   + (\psi^{n, \varepsilon}_\omega\nabla\phi^{n, \varepsilon}_\omega, \partial_t \uu^{n, \varepsilon}_\omega) \notag \\
   &\quad \textcolor{black}{ = \big((\mu^{n, \varepsilon}_\omega - \overline{\mu^{n, \varepsilon}_\omega}) \nabla\phi^{n, \varepsilon}_\omega, \partial_t \uu^{n, \varepsilon}_\omega\big)
   + \big((\psi^{n, \varepsilon}_\omega - \overline{\psi^{n, \varepsilon}_\omega}) \nabla\phi^{n, \varepsilon}_\omega, \partial_t \uu^{n, \varepsilon}_\omega\big) } \notag \\
   &\quad  \leq   \|\mu^{n, \varepsilon}_\omega - \overline{\mu^{n, \varepsilon}_\omega}\|_{L^6(\Omega)}\|\nabla \phi^{n, \varepsilon}_\omega\|_{\LL^3(\Omega)}\|\partial_t \uu^{n, \varepsilon}_\omega\| +\|\psi^{n, \varepsilon}_\omega- \overline{\psi^{n, \varepsilon}_\omega}\|_{L^6(\Omega)}\|\nabla \rho^{n, \varepsilon}_\omega\|_{\LL^3(\Omega)}\|\partial_t \uu^{n, \varepsilon}_\omega\| \notag \\
	& \quad \leq \dfrac{1}{5}\|\partial_t \uu^{n, \varepsilon}_\omega\|^2 \textcolor{black}{\ + C\big(1 +\|\rho^{n, \varepsilon}_\omega\|_{H^2(\Omega)}^2\big) \big(\|\nabla \mu^{n, \varepsilon}_\omega\|^2+\|\nabla \psi^{n, \varepsilon}_\omega\|^2\big). }
 \label{eq:step33}
\end{align}
On account of \eqref{eq:step32}--\eqref{eq:step33}, from \eqref{eq:step31} we deduce
		\begin{equation} \label{eq:step3final}
		\|\partial_t \uu^{n, \varepsilon}_\omega\|^2
\leq C_{10}\|\mathbf{A}\uu^{n, \varepsilon}_\omega\|^2
\textcolor{black}{\ + C\|\nabla \uu^{n, \varepsilon}_\omega\|^4 + C\big(1 +\|\rho^{n, \varepsilon}_\omega\|_{H^2(\Omega)}^2\big) \big(\|\nabla \uu^{n, \varepsilon}_\omega\|^2+\|\nabla \mu^{n, \varepsilon}_\omega\|^2+\|\nabla \psi^{n, \varepsilon}_\omega\|^2\big). }
		\end{equation}

		\textbf{Step 4.} We now take $\mathbf{A}\uu^{n, \varepsilon}_\omega \in L^2(0,T;\mathbf{H}_\sigma)$ as test function in \eqref{eq:weakreg2}$_1$. This gives
		\begin{align}
	&		\dfrac{1}{2}\td{}{t}\|\nabla\uu^{n, \varepsilon}_\omega\|^2 + \big((\uu^{n, \varepsilon}_\omega \cdot \nabla)\uu^{n, \varepsilon}_\omega, \mathbf
			A\uu^{n, \varepsilon}_\omega\big) - \big(\nabla \cdot(\nu(\phi^{n, \varepsilon}_\omega, \rho^{n, \varepsilon}_\omega)D\uu^{n, \varepsilon}_\omega), \mathbf{A}\uu^{n, \varepsilon}_\omega\big)\notag\\
&\quad  =  (\mu^{n, \varepsilon}_\omega\nabla\phi^{n, \varepsilon}_\omega, \mathbf{A}\uu^{n, \varepsilon}_\omega) + (\psi^{n, \varepsilon}_\omega\nabla\rho^{n, \varepsilon}_\omega, \mathbf{A}\uu^{n, \varepsilon}_\omega).
\label{eq:step41}
		\end{align}
The trilinear form can be controlled as follows
\begin{align}
		\big|\big((\uu^{n, \varepsilon}_\omega \cdot \nabla)\uu^{n, \varepsilon}_\omega, \mathbf{A}\uu^{n, \varepsilon}_\omega \big)\big|
& \leq \|\uu^{n, \varepsilon}_\omega\|_{\LL^4(\Omega)}\|\nabla \uu^{n, \varepsilon}_\omega\|_{\LL^4(\Omega)}\|\mathbf{A}\uu^{n, \varepsilon}_\omega\| \notag \\
				& \leq C\|\nabla \uu^{n, \varepsilon}_\omega\|\|\mathbf{A}\uu^{n, \varepsilon}_\omega\|^\frac{3}{2} \notag \\
				& \leq \dfrac{\nu_*}{16}\|\mathbf{A}\uu^{n, \varepsilon}_\omega\|^2 + C\|\nabla \uu^{n, \varepsilon}_\omega\|^4.
	\label{eq:step43}
\end{align}
Following \cite{GMT18} once more, we find that there exists a $q^{n, \varepsilon}_\omega \in L^2(0,T;H^1(\Omega))$ such that $-\Delta \uu^{n, \varepsilon}_\omega + \nabla q^{n, \varepsilon}_\omega = \mathbf{A}\uu^{n, \varepsilon}_\omega$ almost everywhere in $\Omega \times (0,T)$. Thus we obtain
		\begin{align}
			&	- \big(\nabla \cdot(\nu(\phi^{n, \varepsilon}_\omega, \rho^{n, \varepsilon}_\omega)D\uu^{n, \varepsilon}_\omega), \mathbf{A}\uu^{n, \varepsilon}_\omega\big) \notag\\
& \quad = -\dfrac{1}{2}\big(\nu(\phi^{n, \varepsilon}_\omega, \rho^{n, \varepsilon}_\omega)\Delta\uu^{n, \varepsilon}_\omega, \mathbf{A}\uu^{n, \varepsilon}_\omega\big)
- \big(D\uu^{n, \varepsilon}_\omega\nabla\nu(\phi^{n, \varepsilon}_\omega, \rho^{n, \varepsilon}_\omega), \mathbf{A}\uu^n_\omega \big)\notag \\
				& \quad = \dfrac{1}{2}\big(\nu(\phi^{n, \varepsilon}_\omega, \rho^{n, \varepsilon}_\omega)\mathbf{A}\uu^{n, \varepsilon}_\omega, \mathbf{A}\uu^{n, \varepsilon}_\omega\big)
-  \dfrac{1}{2}\big(\nu(\phi^{n, \varepsilon}_\omega, \rho^{n, \varepsilon}_\omega)\nabla q^{n, \varepsilon}_\omega, \mathbf{A}\uu^{n, \varepsilon}_\omega\big)\notag \\
				& \qquad  - \big(D\uu^{n, \varepsilon}_\omega\nabla\nu(\phi^{n, \varepsilon}_\omega, \rho^{n, \varepsilon}_\omega), \mathbf{A}\uu^n_\omega \big) \notag \\
				& \quad \geq \dfrac{\nu_*}{2}\|\mathbf{A}\uu^{n, \varepsilon}_\omega\|^2 + \frac12 \big(q^{n, \varepsilon}_\omega  \nabla\nu(\phi^{n, \varepsilon}_\omega, \rho^{n, \varepsilon}_\omega), \mathbf{A}\uu^n_\omega \big)
- \big(D\uu^{n, \varepsilon}_\omega\nabla\nu(\phi^{n, \varepsilon}_\omega, \rho^{n, \varepsilon}_\omega), \mathbf{A}\uu^n_\omega \big),
\label{eq:step42}
		\end{align}
where the two terms on the right-hand side can be estimated in the following way
\begin{align}
 		& \left|\frac12 \big(q^{n, \varepsilon}_\omega  \nabla\nu(\phi^{n, \varepsilon}_\omega, \rho^{n, \varepsilon}_\omega), \mathbf{A}\uu^n_\omega \big)
  -\big(D\uu^{n, \varepsilon}_\omega\nabla\nu(\phi^{n, \varepsilon}_\omega, \rho^{n, \varepsilon}_\omega), \mathbf{A}\uu^n_\omega \big)\right|\notag  \\
				& \quad  \leq C\big( \|\nabla \phi^{n, \varepsilon}_\omega\|_{\LL^4(\Omega)} + \|\nabla \rho^{n, \varepsilon}_\omega\|_{\LL^4(\Omega)} \big)\big( \|q^{n, \varepsilon}_\omega\|_{L^4(\Omega)} + \|D\uu^{n, \varepsilon}_\omega\|_{\LL^4(\Omega)} \big) \|\mathbf{A}\uu^{n, \varepsilon}_\omega\|\notag \\
				& \quad  \leq C\big( 1 + \|\rho^{n, \varepsilon}_\omega\|^\frac{1}{2}_{H^2(\Omega)}\big) \big(\|q^{n, \varepsilon}_\omega\|^\frac{1}{2}\|q^{n, \varepsilon}_\omega\|_{H^1(\Omega)}^\frac{1}{2} + \|\nabla \uu^{n, \varepsilon}_\omega\|^\frac{1}{2}\|\mathbf{A}\uu^{n, \varepsilon}_\omega\|^\frac{1}{2} \big)\|\mathbf{A}\uu^{n, \varepsilon}_\omega\| \notag \\	
				& \quad  \leq C\big( 1 + \|\rho^{n, \varepsilon}_\omega\|^\frac{1}{2}_{H^2(\Omega)}\big) \big( \|\nabla \uu^{n, \varepsilon}_\omega\|^\frac{1}{4}\|\mathbf{A}\uu^{n, \varepsilon}_\omega\|^\frac{3}{4} + \|\nabla \uu^{n, \varepsilon}_\omega\|^\frac{1}{2}\|\mathbf{A}\uu^{n, \varepsilon}_\omega\|^\frac{1}{2} \big)\|\mathbf{A}\uu^{n, \varepsilon}_\omega\| \notag \\
				& \quad  \leq \frac{\nu_*}{16}\|\mathbf{A}\uu^{n, \varepsilon}_\omega\|^2  + C\big( 1 + \|\rho^{n, \varepsilon}_\omega\|^4_{H^2(\Omega)}\big)\|\nabla \uu^{n, \varepsilon}_\omega\|^2.
\label{eq:step422}
\end{align}
Besides, we have
\begin{align}
& (\mu^{n, \varepsilon}_\omega\nabla\phi^{n, \varepsilon}_\omega, \mathbf{A}\uu^{n, \varepsilon}_\omega)
+ (\psi^{n, \varepsilon}_\omega\nabla\rho^{n, \varepsilon}_\omega, \mathbf{A}\uu^{n, \varepsilon}_\omega) \notag\\
&\quad \textcolor{black}{ =
\big((\mu^{n, \varepsilon}_\omega - \overline{\mu^{n, \varepsilon}_\omega}) \nabla\phi^{n, \varepsilon}_\omega, \mathbf{A}\uu^{n, \varepsilon}_\omega)
+ \big((\psi^{n, \varepsilon}_\omega - \overline{\psi^{n, \varepsilon}_\omega}) \nabla\rho^{n, \varepsilon}_\omega, \mathbf{A}\uu^{n, \varepsilon}_\omega\big)} \notag\\
& \quad \leq \|\mu^{n, \varepsilon}_\omega- \overline{\mu^{n, \varepsilon}_\omega}\|_{L^6(\Omega)}\|\nabla \phi^{n, \varepsilon}_\omega\|_{\LL^3(\Omega)}\|\mathbf{A}\uu^{n, \varepsilon}_\omega\|
+ \|\psi^{n, \varepsilon}_\omega- \overline{\psi^{n, \varepsilon}_\omega}\|_{L^6(\Omega)}\|\nabla \rho^{n, \varepsilon}_\omega\|_{\LL^3(\Omega)}\|\mathbf{A}\uu^{n, \varepsilon}_\omega\|
\notag \\
& \quad \leq \dfrac{\nu_*}{16}\|\mathbf{A}\uu^{n, \varepsilon}_\omega\|^2
\textcolor{black}{\ + C\big(1  +\|\rho^{n, \varepsilon}_\omega\|_{H^2(\Omega)}^2\big)\big(\|\nabla \mu^{n, \varepsilon}_\omega\|^2+\|\nabla \psi^{n, \varepsilon}_\omega\|^2\big). }
\label{eq:step44}
\end{align}
Combining \eqref{eq:step42}--\eqref{eq:step44}, from  \eqref{eq:step41} we deduce that
\begin{align}
&	\dfrac{1}{2}\td{}{t}\|\nabla\uu^{n, \varepsilon}_\omega\|^2
+ \frac{5\nu_*}{16} \|\mathbf{A}\uu^{n, \varepsilon}_\omega\|^2 \notag \\
&\quad \leq\ \textcolor{black}{ C\|\nabla \uu^{n, \varepsilon}_\omega\|^4 + C \big( 1 + \|\rho^{n, \varepsilon}_\omega\|^4_{H^2(\Omega)}\big)\big(\|\nabla \uu^{n, \varepsilon}_\omega\|^2+  \|\nabla \mu^{n, \varepsilon}_\omega\|^2 +\|\nabla \psi^{n, \varepsilon}_\omega\|^2\big). }
\label{eq:step4final}
\end{align}

\textbf{Step 5.} We can now estimate the two scalar products on the right-hand side of \eqref{eq:step1final} and \eqref{eq:step2final} as follows
\begin{align}
& (\partial_t\uu^{n, \varepsilon}_\omega \cdot \nabla \rho^{n, \varepsilon}_\omega, \psi^{n, \varepsilon}_\omega) +(\partial_t\uu^{n, \varepsilon}_\omega \cdot \nabla \phi^{n, \varepsilon}_\omega, \mu^{n, \varepsilon}_\omega) \notag\\
&\quad \textcolor{black}{
=  (\partial_t\uu^{n, \varepsilon}_\omega \cdot \nabla \rho^{n, \varepsilon}_\omega, \psi^{n, \varepsilon}_\omega - \overline{\psi^{n, \varepsilon}_\omega} ) +(\partial_t\uu^{n, \varepsilon}_\omega \cdot \nabla \phi^{n, \varepsilon}_\omega, \mu^{n, \varepsilon}_\omega - \overline{\mu^{n, \varepsilon}_\omega} ) }
\notag\\
& \quad \leq \|\partial_t \uu^{n, \varepsilon}_\omega\|\|\nabla \rho^{n, \varepsilon}_\omega\|_{\LL^3(\Omega)}\|\psi^{n, \varepsilon}_\omega
- \overline{\psi^{n, \varepsilon}_\omega}\|_{L^6(\Omega)}
+ \|\partial_t \uu^{n, \varepsilon}_\omega\|\|\nabla \phi^{n, \varepsilon}_\omega\|_{\LL^3(\Omega)}\|\mu^{n, \varepsilon}_\omega- \overline{\mu^{n, \varepsilon}_\omega}\|_{L^6(\Omega)} \notag \\
& \quad \leq \dfrac{\nu_*}{64C_{10}}\|\partial_t \uu^{n, \varepsilon}_\omega\|^2 + C\big(1   +\|\rho^{n, \varepsilon}_\omega\|_{H^2(\Omega)}^2\big)\big(\|\nabla \mu^{n, \varepsilon}_\omega\|^2+\|\nabla \psi^{n, \varepsilon}_\omega\|^2\big).
\label{eq:step51}
\end{align}
Besides, we observe that
\begin{align}
&	|(\uu^{n, \varepsilon}_\omega \cdot \nabla \phi^{n, \varepsilon}_\omega, \mu^{n, \varepsilon}_\omega) + (\uu^{n, \varepsilon}_\omega \cdot \nabla \rho^{n, \varepsilon}_\omega, \psi^{n, \varepsilon}_\omega)| \notag\\
&\quad \leq \textcolor{black}{ |(\uu^{n, \varepsilon}_\omega \cdot \nabla \phi^{n, \varepsilon}_\omega, \mu^{n, \varepsilon}_\omega-\overline{\mu^{n, \varepsilon}_\omega}) |
            + |(\uu^{n, \varepsilon}_\omega \cdot \nabla \rho^{n, \varepsilon}_\omega, \psi^{n, \varepsilon}_\omega-\overline{\psi^{n, \varepsilon}_\omega})| }\notag\\
&\quad  \leq  \|\uu^{n, \varepsilon}_\omega\|_{\LL^4(\Omega)}\|\nabla \phi^{n, \varepsilon}_\omega\| \|\mu^{n, \varepsilon}_\omega-\overline{\mu^{n, \varepsilon}_\omega}\|_{L^4(\Omega)}
+ \|\uu^{n, \varepsilon}_\omega\|_{\LL^4(\Omega)}\|\nabla \rho^{n, \varepsilon}_\omega\| \|\psi^{n, \varepsilon}_\omega -\overline{\psi^{n, \varepsilon}_\omega}\|_{L^4(\Omega)}\notag\\
			&\quad  \leq  C\|\uu^{n, \varepsilon}_\omega\|^\frac12\|\uu^{n, \varepsilon}_\omega\|_{\mathbf{V}_\sigma}^\frac{1}{2} \|\nabla \mu^{n, \varepsilon}_\omega\|
+ C\|\uu^{n, \varepsilon}_\omega\|^\frac12\|\uu^{n, \varepsilon}_\omega\|_{\mathbf{V}_\sigma}^\frac{1}{2}\|\nabla \psi^{n, \varepsilon}_\omega\| \notag\\
			&\quad  \leq \dfrac{1}{4}\| \nabla \uu^{n, \varepsilon}_\omega\|^2 + \dfrac{1}{4}\| \nabla \mu^{n, \varepsilon}_\omega\|^2 + \dfrac{1}{4}\| \nabla \psi^{n, \varepsilon}_\omega\|^2 + C_{11}\|\uu^{n, \varepsilon}_\omega\|^2,
\label{equiL}
\end{align}
where $C_{11}>0$ only depends on $\Omega$. Set now
    \[
		\begin{split}
			\Lambda(t) &:= \dfrac{1}{2}\|\nabla\uu^{n, \varepsilon}_\omega(t)\|^2
+ \dfrac{1}{2}\|\nabla \mu^{n, \varepsilon}_\omega(t)\|^2 +\dfrac{1}{2}\|\nabla \psi^{n, \varepsilon}_\omega(t)\|^2 +(\uu^{n, \varepsilon}_\omega(t) \cdot \nabla \phi^{n, \varepsilon}_\omega(t), \mu^{n, \varepsilon}_\omega(t)) \\
&\qquad +  (\uu^{n, \varepsilon}_\omega(t) \cdot \nabla \rho^{n, \varepsilon}_\omega(t), \psi^{n, \varepsilon}_\omega(t))\ \textcolor{black}{+ C_{11}\|\uu^{n, \varepsilon}_\omega\|^2}, \\
			\mathcal{G}(t) & := \dfrac{\nu_*}{32C_{10}}\|\partial_t\uu^{n, \varepsilon}_\omega(t)\|^2 + \dfrac{\nu_*}{4}\|\mathbf{A}\uu^{n, \varepsilon}_\omega(t)\|^2 + \dfrac{\alpha}{2}\|\Delta \partial_t \phi^{n, \varepsilon}_\omega(t)\|^2 + \dfrac{1}{2}\|\nabla \partial_t \phi^{n, \varepsilon}_\omega(t)\|^2 + \dfrac{\beta}{2}\|\nabla \partial_t \rho^{n, \varepsilon}_\omega(t)\|^2,
		\end{split}
	\]
for all $t\in [0,T]$. Therefore, we have
		\begin{align*}
		&\Lambda(t)\leq \dfrac{3}{4}\|\nabla \uu^{n, \varepsilon}_\omega\|^2
        + \dfrac{3}{4}\|\nabla \mu^{n, \varepsilon}_\omega\|^2
        + \dfrac{3}{4}\|\nabla \psi^{n, \varepsilon}_\omega\|^2  + \textcolor{black}{2C_{11}\|\uu^{n, \varepsilon}_\omega\|^2},\\
        &\Lambda(t) \geq \dfrac{1}{4}\|\nabla \uu^{n, \varepsilon}_\omega\|^2
        + \dfrac{1}{4}\|\nabla \mu^{n, \varepsilon}_\omega\|^2
        + \dfrac{1}{4}\|\nabla \psi^{n, \varepsilon}_\omega\|^2\geq 0.
		\end{align*}
By Young's inequality, we get
\begin{align}
C_{11}\frac{\mathrm{d}}{\mathrm{d}t}\|\uu^{n, \varepsilon}_\omega\|^2 \leq  2C_{11}\|\partial_t \uu^{n, \varepsilon}_\omega\|\|\uu^{n, \varepsilon}_\omega\|
\leq \dfrac{\nu_*}{64C_{10}}\|\partial_t \uu^{n, \varepsilon}_\omega\|^2 + \frac{64}{\nu^*}C_{10}C_{11}^2\|\uu^{n, \varepsilon}_\omega\|^2.
\label{eq:lowdiffu}
\end{align}

Collecting \eqref{eq:step1final}, \eqref{eq:step2final}, \eqref{eq:step4final}, \eqref{eq:lowdiffu} and \eqref{eq:step3final} multiplied by $\dfrac{\nu_*}{16C_{10}}$, and taking  \eqref{eq:timederivatives}, \eqref{rhoL4h2} and \eqref{eq:step51} into account, we arrive at the differential inequality
		\begin{equation*}
			\td{\Lambda}{t} + \mathcal{G} \leq C_{12}(1 + \Lambda)\Lambda.
		\end{equation*}
Since
		\[
		\int_0^T \Lambda(t)\,  \mathrm{d}t\leq C_{13},
		\]
		where the constant $C_{13}>0$ does not depend on $n$, $k$, $\varepsilon$ and $\Omega$, an application of Gronwall's lemma yields
		\[
		\Lambda(t) \leq \Lambda(0)e^{C_{12}C_{13}} + C_{12}e^{C_{12}C_{13}}T, \qquad \forall \: t \in [0,T].
		\]
As a consequence, we also have
\begin{align*}
\int_0^t \mathcal{G}(\tau)\,\mathrm{d}\tau &\leq  \Lambda(0)+ C_{12}\int_0^t (1+\Lambda(\tau))\Lambda(\tau)\,\mathrm{d}\tau \\
&\leq  \Lambda(0)+2C_{12}T+ 2C_{12}T\big(\Lambda(0)e^{C_{12}C_{13}} + C_{12}e^{C_{12}C_{13}}T\big)^2 ,\quad \forall\, t\in [0,T].
\end{align*}
 We are left to control the initial quantity $\Lambda(0)$. To this end, we observe that, owing to \eqref{eq:estimateinitdata2},
\begin{align}
	\Lambda(0) & =
  \dfrac{1}{2}\|\nabla\uu^{n}_0\|^2
+ \dfrac{1}{2}\|\nabla \mu^{n,\varepsilon}_\omega(0)\|^2
+ \dfrac{1}{2}\|\nabla \psi^{n,\varepsilon}_\omega(0)\|^2
+ (\uu^n_0 \cdot \nabla \phi^n_{0}, \mu^{n,\varepsilon}_\omega(0))
+ (\uu^n_0 \cdot \nabla \rho^n_{0,k}, \psi^{n,\varepsilon}_\omega(0))\notag \\
&\quad + \textcolor{black}{C_{11}\|\uu^{n}_0\|^2 }
\notag\\
		& \leq \dfrac{1}{2}\|\nabla\uu^{n}_0\|^2
+ \dfrac{1}{2}\|\nabla \mu^{n,\varepsilon}_\omega(0)\|^2
+ \dfrac{1}{2}\|\nabla \psi^{n,\varepsilon}_\omega(0)\|^2  \notag\\
&\quad + \|\uu^n_0\|_{\LL^3(\Omega)}\left( \|\phi^n_{0}\|_{L^6(\Omega)} \|\mu^{n,\varepsilon}_\omega(0)\| +\|\rho^n_{0,k}\|_{L^6(\Omega)}\|\psi^{n,\varepsilon}_\omega(0)\| \right) + \textcolor{black}{C_{11}\|\uu^{n}_0\|^2}
\notag \\
		& \leq C\|\uu_0\|_{\mathbf{V}_\sigma}^2 + C\left( 1 + \|\phi_0\|^2_{H^1(\Omega)}  + \|\rho_0\|^2_{H^1(\Omega)} \right)\left(\|\mu^{n,\varepsilon}_\omega(0)\|_{H^1(\Omega)}^2 + \|\psi^{n,\varepsilon}_\omega(0)\|_{H^1(\Omega)}^2 \right).
\label{eq:step512}
	\end{align}
Moreover, we have
\begin{align}
		\|\mu^{n,\varepsilon}_\omega(0)\|_{H^1(\Omega)}
&  = \| \Pi_n\left( \alpha \Delta^2 \phi^n_{0} -\Delta \phi^n_{0} +  S'_\phi(\phi^n_{0}) + \nabla \cdot(\rho_{0,k}^n\nabla\phi^n_{0})
- \textcolor{black}{\omega \nabla \cdot\left( |\nabla \phi^n_{0}|^2\nabla \phi^n_{0} \right)} \right) \|_{H^1(\Omega)} \notag \\
		&  \leq \| \alpha \Delta^2 \phi^n_{0}-\Delta \phi^n_{0}  + S'_\phi(\phi^n_{0}) + \nabla \cdot(\rho_{0,k}^n\nabla\phi^n_{0}) - \textcolor{black}{\omega \nabla \cdot\left( |\nabla \phi^n_{0}|^2\nabla \phi^n_{0} \right)} \|_{H^1(\Omega)} \notag\\
		&  \leq  C\|\phi^n_{0}\|_{H^5(\Omega)} + \| S'_\phi(\phi^n_{0})\|_{H^1(\Omega)} + \|\nabla \cdot(\rho_{0,k}^n\nabla\phi^n_{0})\|_{H^1(\Omega)}+ \textcolor{black}{\omega \||\nabla \phi^n_{0}|^2\nabla \phi^n_{0} \|_{\mathbf{H}^2(\Omega)}.}
\label{eq:step52}
		\end{align}
Observe that
\begin{equation*}
\begin{split}
\|\nabla \cdot(\rho_{0,k}^n\nabla\phi^n_{0})\|_{H^1(\Omega)}
&  \leq \| \rho_{0,k}^n\Delta \phi_{0}^n\|_{H^1(\Omega)} + \| \nabla \rho_{0,k}^n \cdot \nabla \phi_{0}^n\|_{H^1(\Omega)},
\end{split}
\end{equation*}
where
\begin{equation*}
\begin{split}
				\| \rho_{0,k}^n\Delta \phi_{0}^n\|_{H^1(\Omega)} &  \leq \| \rho_{0,k}^n\Delta \phi_{0}^n\| + \| \nabla \rho_{0,k}^n \Delta \phi_{0}^n\| + \| \rho_{0,k}^n \nabla \Delta \phi_{0}^n\| \\
				& \leq\|\rho_{0,k}^n\|_{L^\infty(\Omega)}\left(\|\Delta \phi_{0}^n \| + \|\nabla \Delta \phi_{0}^n \|\right) + \|\nabla \rho_{0,k}^n\|_{\LL^4(\Omega)}\|\Delta \phi_{0}^n \|_{L^4(\Omega)} \\
				& \leq C\left(\|\rho_{0,k}^n\|_{H^2(\Omega)}^2 + \|\phi_{0}^n\|_{H^3(\Omega)}^2  \right),
\end{split}
\end{equation*}
and
\begin{equation*}
\begin{split}
	\| \nabla \rho_{0,k}^n \cdot \nabla \phi_{0}^n\|_{H^1(\Omega)}
    & \leq \|\nabla \rho_{0,k}^n\|_{\LL^4(\Omega)}\|\nabla \phi_{0}^n \|_{\LL^4(\Omega)} + \|\rho_{0,k}^n \|_{H^2(\Omega)}\|\nabla \phi_{0}^n\|_{\LL^\infty(\Omega)} + \|\nabla \rho_{0,k}^n\|_{\LL^4(\Omega)}\|\phi_{0}^n \|_{W^{2,4}(\Omega)} \\
	& \leq C\left(\|\rho_{0,k}^n\|_{H^2(\Omega)}^2 + \|\phi_{0}^n\|_{H^3(\Omega)}^2  \right).
\end{split}
\end{equation*}
Therefore, exploiting \eqref{eq:estimateinitdata0}, in light of the above controls, from \eqref{eq:step52} and Young's inequality we infer that
\begin{align} \label{eq:step53}
\|\mu^{n,\varepsilon}_\omega(0)\|_{H^1(\Omega)}
& \leq \textcolor{black}{C\|\phi^n_{0}-\phi_{0}\|_{H^5(\Omega)}^3} + \| S'_\phi(\phi^n_{0})-S'_\phi(\phi_{0})\|_{H^1(\Omega)}  + \| S'_\phi(\phi_{0})\|_{H^1(\Omega)}\notag\\
&\quad +C\|\rho_{0,k}^n-\rho_{0,k}\|_{H^2(\Omega)}^2 + C  (1+\textcolor{black}{\|\phi_{0}\|_{H^5(\Omega)}^3} + \|\widehat{\psi}_0 \|^2).
\end{align}
Recalling now \eqref{eq:estimateinitdata}, we find
\begin{align}
	\|\psi^{n,\varepsilon}_\omega(0)\|_{H^1(\Omega)}
    &  = \left\| \Pi_n\left( -\beta\Delta \rho^n_{0,k} + S'_{\rho,\varepsilon}(\rho^n_{0,k}) -\dfrac{\theta}{2} |\nabla\phi^n_{0}|^2 \right) \right\|_{H^1(\Omega)}  \notag\\
	& \leq \left\| -\beta\Delta \rho^n_{0,k} + \widehat{S}'_{\rho,\varepsilon}(\rho^n_{0,k}) +R'_\rho(\rho^n_{0,k}) - \dfrac{\theta}{2} |\nabla\phi^n_{0}|^2 \right\|_{H^1(\Omega)} \notag \\
	& \leq \beta\|\rho^n_{0,k}- \rho_{0,k}\|_{H^3(\Omega)} + \|\widehat{S}'_{\rho,\varepsilon}(\rho^n_{0,k})-\widehat{S}'_{\rho,\varepsilon}(\rho_{0,k}) \|_{H^1(\Omega)}  + \textcolor{black}{\|\widehat{\psi}_0\|_{H^1(\Omega)}}\notag\\
&\quad + \|R'_\rho(\rho^n_{0,k})\|_{H^1(\Omega)}+\dfrac{\theta}{2} \left\| |\nabla\phi^n_{0}|^2\right\|_{H^1(\Omega)}.
\label{eq:step54}
\end{align}
Assumption \textbf{(H3)} and H\"{o}lder's inequality entail that
$$
\|R'_\rho(\rho^n_{0,k})\|_{H^1(\Omega)}\leq C\left(1 + \|\rho^n_{0,k}\|_{H^1(\Omega)}  \right),
$$
	\[
	\dfrac{\theta}{2} \left\| |\nabla\phi^n_{0}|^2\right\|_{H^1(\Omega)} \leq C\left( \|\nabla \phi^n_{0}\|_{\LL^4(\Omega)}^2 + \|\phi^n_{0}\|_{W^{2,4}(\Omega)}\|\nabla\phi^n_{0}\|_{\LL^4(\Omega)}\right) \leq C\|\phi^n_{0}\|_{H^3(\Omega)}^2.
	\]
Therefore, on account of \eqref{eq:estimateinitdata2}, from \eqref{eq:step54} we deduce that
\begin{align}
\|\psi^{n,\varepsilon}_\omega(0)\|_{H^1(\Omega)}
&\leq C\|\rho^n_{0,k}- \rho_{0,k}\|_{H^3(\Omega)} + \|\widehat{S}'_{\rho,\varepsilon}(\rho^n_{0,k})-\widehat{S}'_{\rho,\varepsilon}(\rho_{0,k}) \|_{H^1(\Omega)}  + \textcolor{black}{\|\widehat{\psi}_0\|_{H^1(\Omega)}} \notag \\
&\quad +C\|\phi^n_{0}-\phi_{0}\|_{H^3(\Omega)}^2+ C\big(1 +\|\phi_{0}\|_{H^3(\Omega)}^2 +\|\rho_{0}\|_{H^1(\Omega)}\big).
\label{eq:step55}
		\end{align}
Exploiting the strong convergences  $\phi_{0}^n \to \phi_{0}$ and $\rho_{0,k}^n \to \rho_{0,k}$ in $H^5(\Omega)$ and $H^3(\Omega)$, respectively, we can find $n^{**}(k)>n^*(k)$ such that, for all $n>n^{**}(k)$ and $k > k^*$, it holds
\begin{align}
\|\phi_{0}^n-\phi_{0}\|_{H^5(\Omega)}\leq 1,\quad \|\rho^n_{0,k}-\rho_{0,k}\|_{H^3(\Omega)}\leq 1.
\label{nerr}
\end{align}
Thanks to \textbf{(H3)$'$}, arguing line by line exactly as in \cite[(4.39)]{GMT18}, we have
\begin{align}
\|\widehat{S}'_{\rho,\varepsilon}(\rho^n_{0,k}) - \widehat{S}'_{\rho,\varepsilon}(\rho_{0,k})\|_{H^1(\Omega)} \leq C\|\rho_{0,k}^n - \rho_{0,k}\|_{H^1(\Omega)},
\label{kkk}
\end{align}
		where $C>0$ may depend on $k$. In a similar manner, thanks to \textbf{(H2)$'$} and \eqref{nerr}, we get
\begin{align*}
&\|S'_\phi(\phi^n_{0})-S'_\phi(\phi_{0} )\|_{H^1(\Omega)}\notag\\
&\quad \leq \|S'_\phi(\phi^n_{0})-S'_\phi(\phi_{0} )\| + \|S''_\phi(\phi^n_{0})\nabla (\phi^n_{0}-\phi_{0} )\| +\| (S''_\phi(\phi^n_{0})-S''_\phi(\phi_{0}))\nabla \phi_0\|\notag\\
&\quad \leq C\| \phi^n_{0}-\phi_{0}\|_{H^1(\Omega)}.
\end{align*}
From the above estimates, we infer from \eqref{eq:step53} and \eqref{eq:step55} that for
\textcolor{black}{every fixed $k > k^*$, for any} $\omega\in (0, 1]$, $\varepsilon \in (0,\epsilon_3(k))$, $n>n^{**}(k)$, it holds
		\[
		\|\mu^{n,\varepsilon}_\omega(0)\|_{H^1(\Omega)}  + \|\psi^{n,\varepsilon}_\omega(0)\|_{H^1(\Omega)}
       \leq C\left( 1 + \|\phi_0\|_{H^5(\Omega)}^2 + \|{\rho}_0\|_{H^1(\Omega)}^2+\|\widehat{\psi}_{0}\|_{H^1(\Omega)}^2
       \right),
		\]
		where $C>0$ may depend on $\|\phi_0\|_{H^2(\Omega)}$ and the right-hand side is independent of the parameters $n$, $\omega$ as well as $\varepsilon$, \textcolor{black}{but may depend on $k$}. We can thus conclude that
		\begin{align*}
		\Lambda(t) + \int_0^t \mathcal{G}(\tau)\,\mathrm{d}\tau \leq C, \quad \forall\, t\in (0,T],
		\end{align*}
		where $C$ depends on $T$ and the initial conditions, but not on $n$, $\omega$ and $\varepsilon$, \textcolor{black}{however it may depend on $k$}.
From the definitions of $\Lambda$ and $\mathcal{G}$, we arrive at the expected conclusion.
\end{proof}

\textbf{Higher-order estimates in three dimensions.}
The difference with respect to the previous argument is mainly in the usage of Sobolev embeddings and interpolation inequalities.
\begin{lemma} \label{prop:ub53d}
Let $d = 3$. There exists a time $T^*\in (0,T]$ independent of $\omega$, $n$, $k$, $\varepsilon$ such that the sequences $\{\mu^{n,\varepsilon}_\omega\}$ and $\{\psi^{n,\varepsilon}_\omega\}$  are uniformly bounded in $L^\infty(0,T^*;H^1(\Omega))$ and the sequence $\{\uu^{n,\varepsilon}_\omega\}$ is uniformly bounded in $L^\infty(0,T^*;\mathbf{V}_\sigma) \cap L^2(0,T^*; \mathbf{W}_\sigma)$.
\end{lemma}
\begin{proof}
For the sake of brevity, we only report the main differences with respect to the proof of Lemma \ref{prop:ub5}. First, using interpolation inequalities in three dimensions, we can still recover the differential inequalities for $ \psi^{n, \varepsilon}_\omega$ and $ \mu^{n, \varepsilon}_\omega$, namely,
\begin{align}
&\td{}{t}\left( \dfrac{1}{2}\|\nabla \psi^{n, \varepsilon}_\omega\|^2 +  (\uu^{n, \varepsilon}_\omega \cdot \nabla \rho^{n, \varepsilon}_\omega, \psi^{n, \varepsilon}_\omega) \right)  +  \dfrac{3\beta}{4} \| \nabla \partial_t\rho^{n, \varepsilon}_\omega\|^2
- \dfrac{1}{16}\| \nabla \partial_t\phi^{n, \varepsilon}_\omega\|^2 \notag \\
&\quad \leq (\partial_t\uu^{n, \varepsilon}_\omega \cdot \nabla \rho^{n, \varepsilon}_\omega, \psi^{n, \varepsilon}_\omega)
\ \textcolor{black}{+ C\big( \|\partial_t \rho^{n, \varepsilon}_\omega\|_{V_0^*}^2 + \|\uu^{n, \varepsilon}_\omega\|_{\LL^3(\Omega)}^2
\|\nabla \psi^{n, \varepsilon}_\omega\|^2 \big),}
\label{eq:step1final3d}
\end{align}
and
\begin{align}
& \td{}{t}\left( \dfrac{1}{2}\|\nabla  \mu^{n, \varepsilon}_\omega\|^2
+  (\uu^{n, \varepsilon}_\omega \cdot \nabla \phi^{n, \varepsilon}_\omega, \mu^{n, \varepsilon}_\omega) \right)
+ \dfrac{\alpha}{2} \|\Delta \partial_t \phi^{n, \varepsilon}_\omega\|^2
+ \dfrac{9}{16} \| \nabla \partial_t\phi^{n, \varepsilon}_\omega\|^2
- \dfrac{\beta}{4}\| \nabla \partial_t\rho^{n, \varepsilon}_\omega\|^2
\notag  \\
&\quad \leq (\partial_t\uu^{n, \varepsilon}_\omega \cdot \nabla \phi^{n, \varepsilon}_\omega, \mu^{n, \varepsilon}_\omega)
+ \ \textcolor{black}{ C\big( \|\partial_t \phi^{n, \varepsilon}_\omega\|_{V_0^*}^2 + \|\uu^{n, \varepsilon}_\omega\|_{\LL^3(\Omega)}^2 \|\nabla \mu^{n, \varepsilon}_\omega\|^2 \big). }
\label{eq:step2final3d}
\end{align}
Concerning the estimates for $\uu^{n,\varepsilon}_\omega$, a number of modifications are needed. Some computations can be borrowed from \cite{GMT18}. For instance,
\begin{align}
	\big((\uu^n_\omega \cdot \nabla)\uu^{n, \varepsilon}_\omega, \partial_t \uu^{n, \varepsilon}_\omega\big)
                & \leq \|\uu^{n, \varepsilon}_\omega\|_{\LL^6(\Omega)}\|\nabla \uu^{n, \varepsilon}_\omega\|_{\LL^3(\Omega)}\|\partial_t \uu^{n, \varepsilon}_\omega\|\notag \\
				& \leq C\|\nabla \uu^{n, \varepsilon}_\omega\|^\frac{3}{2}\|\mathbf{A}\uu^{n, \varepsilon}_\omega\|^\frac{1}{2}\|\partial_t \uu^{n, \varepsilon}_\omega\| \notag \\
				& \leq \dfrac{1}{10}\|\partial_t \uu^{n, \varepsilon}_\omega\|^2 + C\left( \|\mathbf{A}\uu^{n, \varepsilon}_\omega\|^2 + \|\nabla \uu^{n, \varepsilon}_\omega\|^6 \right),
 \label{eq:step323d}
\end{align}
\begin{align}
&\big(\nabla \cdot(\nu(\phi^{n, \varepsilon}_\omega, \rho^{n, \varepsilon}_\omega) D\uu^{n, \varepsilon}_\omega), \partial_t\uu^{n, \varepsilon}_\omega\big)\notag \\
&\quad = \dfrac{1}{2}\big(\nu(\phi^{n, \varepsilon}_\omega, \rho^{n, \varepsilon}_\omega)\Delta\uu^{n, \varepsilon}_\omega, \partial_t\uu^{n, \varepsilon}_\omega\big)
+ \big(D\uu^{n, \varepsilon}_\omega\nabla\nu(\phi^{n, \varepsilon}_\omega, \rho^{n, \varepsilon}_\omega), \partial_t \uu^{n, \varepsilon}_\omega \big) \notag  \\
				&\quad \leq C\|\mathbf{A} \uu^{n, \varepsilon}_\omega\|\|\partial_t \uu^{n, \varepsilon}_\omega\| + C\big(\|\nabla \phi^{n, \varepsilon}_\omega\|_{\LL^6(\Omega)} + \|\nabla \rho^{n, \varepsilon}_\omega\|_{\LL^6(\Omega)}\big)\|D\uu^{n, \varepsilon}_\omega\|_{\LL^3(\Omega)}\|\partial_t \uu^{n, \varepsilon}_\omega\|
\notag \\
				&\quad  \leq \dfrac{1}{5}\|\partial_t \uu^{n, \varepsilon}_\omega\|^2 + C \|\mathbf{A}\uu^{n, \varepsilon}_\omega\|^2 + C\big(1 + \| \rho^{n, \varepsilon}_\omega\|_{H^2(\Omega)}^2\big)\|\nabla\uu^{n, \varepsilon}_\omega\|\|\mathbf{A}\uu^{n, \varepsilon}_\omega\| \notag \\
				& \quad \leq \dfrac{1}{5}\|\partial_t \uu^{n, \varepsilon}_\omega\|^2 + C\|\mathbf{A}\uu^{n, \varepsilon}_\omega\|^2 + C\big(1+\| \rho^{n, \varepsilon}_\omega\|_{H^2(\Omega)}^4\big)\|\nabla\uu^{n, \varepsilon}_\omega\|^2,
\label{eq:step3223d}
\end{align}
and
\begin{align}
& (\mu^{n,\varepsilon}_\omega\nabla\phi^{n, \varepsilon}_\omega, \partial_t \uu^{n, \varepsilon}_\omega) + (\psi^{n, \varepsilon}_\omega\nabla\rho^{n, \varepsilon}_\omega, \partial_t \uu^{n, \varepsilon}_\omega) \notag\\
&\quad \textcolor{black}{ = \big((\mu^{n,\varepsilon}_\omega-\overline{\mu^{n,\varepsilon}_\omega})\nabla\phi^{n, \varepsilon}_\omega, \partial_t \uu^{n, \varepsilon}_\omega\big)
+\big((\psi^{n, \varepsilon}_\omega - \overline{\psi^{n, \varepsilon}_\omega}) \nabla\rho^{n, \varepsilon}_\omega, \partial_t \uu^{n, \varepsilon}_\omega\big) }\notag\\
&\quad  \leq   \|\mu^{n, \varepsilon}_\omega-\overline{\mu^{n,\varepsilon}_\omega}\|_{L^6(\Omega)}\|\nabla \phi^{n, \varepsilon}_\omega\|_{\LL^3(\Omega)}\|\partial_t \uu^{n, \varepsilon}_\omega\|
+\|\psi^{n, \varepsilon}_\omega- \overline{\psi^{n, \varepsilon}_\omega}\|_{L^6(\Omega)}\|\nabla \rho^{n, \varepsilon}_\omega\|_{\LL^3(\Omega)}\|\partial_t \uu^{n, \varepsilon}_\omega\| \notag\\
&\quad  \leq \dfrac{1}{5}\|\partial_t \uu^{n, \varepsilon}_\omega\|^2
\textcolor{black}{ + C\big(1 +\|\nabla \rho^{n, \varepsilon}_\omega\|_{\LL^3(\Omega)}^2)\big(\|\nabla \mu^{n, \varepsilon}_\omega\|^2+\|\nabla \psi^{n, \varepsilon}_\omega\|^2\big). }
\label{eq:step333d}
\end{align}
On account of estimates \eqref{eq:step323d}--\eqref{eq:step333d}, from \eqref{eq:step31} we infer that
\begin{align} \label{eq:step3final3d}
\|\partial_t \uu^{n, \varepsilon}_\omega\|^2
&\leq C_{14} \|\mathbf{A}\uu^{n, \varepsilon}_\omega\|^2
\textcolor{black}{ + C\|\nabla \uu^{n, \varepsilon}_\omega\|^6 + C \big(1+ \|\rho^{n, \varepsilon}_\omega\|^4_{H^2(\Omega)}\big) \big(\|\nabla \uu^{n, \varepsilon}_\omega\|^2 +\|\nabla \mu^{n, \varepsilon}_\omega\|^2 +\|\nabla \psi^{n, \varepsilon}_\omega\|^2\big). }
\end{align}
Consider now the terms in \eqref{eq:step41}. We have
\begin{align}
\big|\big((\uu^{n, \varepsilon}_\omega \cdot \nabla)\uu^{n, \varepsilon}_\omega, \mathbf{A}\uu^{n, \varepsilon}_\omega\big)\big|
& \leq \|\uu^{n, \varepsilon}_\omega\|_{\LL^6(\Omega)}\|\nabla \uu^{n, \varepsilon}_\omega\|_{\LL^3(\Omega)}\|\mathbf{A}\uu^{n, \varepsilon}_\omega\| \notag \\
				& \leq C\|\nabla \uu^{n, \varepsilon}_\omega\|^\frac{3}{2}\|\mathbf{A}\uu^{n, \varepsilon}_\omega\|^\frac{3}{2} \notag \\
				& \leq \dfrac{\nu_*}{16}\|\mathbf{A}\uu^{n, \varepsilon}_\omega\|^2 + C\|\nabla \uu^{n, \varepsilon}_\omega\|^6,
	\label{eq:step433d}		
\end{align}
\begin{align}
		& \left|\frac12 \big(q^{n, \varepsilon}_\omega  \nabla\nu(\phi^{n, \varepsilon}_\omega, \rho^{n, \varepsilon}_\omega), \mathbf{A}\uu^n_\omega \big)
- \big(D\uu^{n, \varepsilon}_\omega\nabla\nu(\phi^{n, \varepsilon}_\omega, \rho^{n, \varepsilon}_\omega), \mathbf{A}\uu^n_\omega \big)\right| \notag \\
				& \qquad \leq C\big( \|\nabla \phi^{n, \varepsilon}_\omega\|_{\LL^6(\Omega)} + \|\nabla \rho^{n, \varepsilon}_\omega\|_{\LL^6(\Omega)} \big)\big( \|q^n_\omega\|_{L^3(\Omega)} + \|D\uu^{n, \varepsilon}_\omega\|_{\LL^3(\Omega)} \big) \|\mathbf{A}\uu^{n, \varepsilon}_\omega\| \notag \\
				& \qquad  \leq C\big( 1 + \|\rho^{n, \varepsilon}_\omega\|_{H^2(\Omega)}\big)\big(\|q^{n, \varepsilon}_\omega\|^\frac{1}{2}\|q^{n, \varepsilon}_\omega\|_{H^1(\Omega)}^\frac{1}{2} + \|\nabla \uu^{n, \varepsilon}_\omega\|^\frac{1}{2}\|\mathbf{A}\uu^{n, \varepsilon}_\omega\|^\frac{1}{2} \big)\|\mathbf{A}\uu^{n, \varepsilon}_\omega\| \notag \\	
				& \qquad  \leq C\big( 1 + \|\rho^{n, \varepsilon}_\omega\|_{H^2(\Omega)}\big)\big( \|\nabla \uu^{n, \varepsilon}_\omega\|^\frac{1}{4}\|\mathbf{A}\uu^{n, \varepsilon}_\omega\|^\frac{3}{4} + \|\nabla \uu^{n, \varepsilon}_\omega\|^\frac{1}{2}\|\mathbf{A}\uu^{n, \varepsilon}_\omega\|^\frac{1}{2} \big)\|\mathbf{A}\uu^{n, \varepsilon}_\omega\| \notag \\
				& \qquad  \leq \frac{\nu_*}{16}\|\mathbf{A}\uu^{n, \varepsilon}_\omega\|^2  + C\big( 1 + \|\rho^{n, \varepsilon}_\omega\|^8_{H^2(\Omega)}\big)\|\nabla \uu^{n, \varepsilon}_\omega\|^2,
\label{eq:step4223d}
\end{align}
and
\begin{align}
& (\mu^{n, \varepsilon}_\omega\nabla\phi^{n, \varepsilon}_\omega, \mathbf{A}\uu^{n, \varepsilon}_\omega) +(\psi^{n, \varepsilon}_\omega\nabla\rho^{n, \varepsilon}_\omega, \mathbf{A}\uu^{n, \varepsilon}_\omega)  \notag\\
&\quad \textcolor{black}{=  \big((\mu^{n, \varepsilon}_\omega-\overline{\mu^{n, \varepsilon}_\omega}) \nabla\phi^{n, \varepsilon}_\omega, \mathbf{A}\uu^{n, \varepsilon}_\omega\big)
+\big((\psi^{n, \varepsilon}_\omega-\overline{\psi^{n, \varepsilon}_\omega})\nabla\rho^{n, \varepsilon}_\omega, \mathbf{A}\uu^{n, \varepsilon}_\omega\big) } \notag\\
&\quad  \leq   \|\mu^{n, \varepsilon}_\omega-\overline{\mu^{n, \varepsilon}_\omega}\|_{L^6(\Omega)}\|\nabla \phi^{n, \varepsilon}_\omega\|_{\LL^3(\Omega)}\|\mathbf{A}\uu^{n, \varepsilon}_\omega\|
+ \|\psi^{n, \varepsilon}_\omega-\overline{\psi^{n, \varepsilon}_\omega}\|_{L^6(\Omega)}\|\nabla \rho^{n, \varepsilon}_\omega\|_{\LL^3(\Omega)}\|\mathbf{A}\uu^{n, \varepsilon}_\omega\| \notag \\
&\quad  \leq \dfrac{\nu_*}{16}\|\mathbf{A}\uu^{n, \varepsilon}_\omega\|^2
\ \textcolor{black}{+ C\big(1 +\|\nabla \rho^{n, \varepsilon}_\omega\|_{\LL^3(\Omega)}^2\big)\big( \|\nabla \mu^{n, \varepsilon}_\omega\|^2 +\|\nabla \psi^{n, \varepsilon}_\omega\|^2\big). }
 \label{eq:step443d}
\end{align}
Thus from \eqref{eq:step41} we obtain
\begin{align}
&\dfrac{1}{2}\td{}{t}\|\nabla\uu^{n, \varepsilon}_\omega\|^2 +  \dfrac{5\nu_*}{16}\|\mathbf{A}\uu^{n, \varepsilon}_\omega\|^2 \notag\\
&\quad \leq
\textcolor{black}{ C\|\nabla \uu^n_\omega\|^6 + C\big(1+ \|\rho^{n, \varepsilon}_\omega\|^8_{H^2(\Omega)}\big)(\|\nabla \uu^{n, \varepsilon}_\omega\|^2
+ \|\nabla \mu^n_\omega\|^2 + \|\nabla \psi^{n, \varepsilon}_\omega\|^2\big). }
\label{eq:step4final3d}
\end{align}
The controls on the leftover terms in \eqref{eq:step1final3d} and \eqref{eq:step2final3d} are similar to their two-dimensional counterpart, namely,
\begin{align}
	& (\partial_t\uu^{n, \varepsilon}_\omega \cdot \nabla \rho^{n, \varepsilon}_\omega, \psi^{n, \varepsilon}_\omega) +(\partial_t\uu^{n, \varepsilon}_\omega \cdot \nabla \phi^{n, \varepsilon}_\omega, \mu^{n, \varepsilon}_\omega) \notag\\
&\quad  \leq \|\partial_t \uu^{n, \varepsilon}_\omega\|\|\nabla \rho^{n, \varepsilon}_\omega\|_{\LL^3(\Omega)}\|\psi^{n, \varepsilon}_\omega-\overline{\psi^{n, \varepsilon}_\omega}\|_{L^6(\Omega)}
+ \|\partial_t \uu^{n, \varepsilon}_\omega\|\|\nabla \phi^{n, \varepsilon}_\omega\|_{\LL^3(\Omega)}\|\mu^{n, \varepsilon}_\omega-\overline{\mu^{n, \varepsilon}_\omega}\|_{L^6(\Omega)} \notag \\
	& \quad \leq \dfrac{\nu_*}{32C_{14}}\|\partial_t \uu^{n, \varepsilon}_\omega\|^2
+\ \textcolor{black}{ C\big(1 +\|\nabla \rho^{n, \varepsilon}_\omega\|_{\LL^3(\Omega)}^2\big)\big(  \|\nabla \mu^{n, \varepsilon}_\omega\|^2 +\|\nabla \psi^{n, \varepsilon}_\omega\|^2\big).}
\label{eq:step513d}
\end{align}

Defining the quantities $\Lambda$ and $\mathcal{G}$ as in Lemma \ref{prop:ub5}, we still have that $\Lambda$ is bounded from above and from below. Indeed, we have
		\[
		\begin{split}
		&	|(\uu^{n, \varepsilon}_\omega \cdot \nabla \rho^{n, \varepsilon}_\omega, \psi^{n, \varepsilon}_\omega) +(\uu^{n, \varepsilon}_\omega \cdot \nabla \phi^{n, \varepsilon}_\omega, \mu^{n, \varepsilon}_\omega)|\\
       & \quad \leq \|\uu^{n, \varepsilon}_\omega\|_{\LL^3(\Omega)}\|\nabla \rho^{n, \varepsilon}_\omega\| \|\psi^{n, \varepsilon}_\omega-\overline{\psi^{n, \varepsilon}_\omega}\|_{L^6(\Omega)} + \|\uu^{n, \varepsilon}_\omega\|_{\LL^3(\Omega)}\|\nabla \phi^{n, \varepsilon}_\omega\| \|\mu^{n, \varepsilon}_\omega-\overline{\mu^{n, \varepsilon}_\omega}\|_{L^6(\Omega)}\\
	   &\quad  \leq C\|\uu^{n, \varepsilon}_\omega\|^\frac12\|\uu^{n, \varepsilon}_\omega\|_{\mathbf{V}_\sigma}^\frac{1}{2}\|\psi^{n, \varepsilon}_\omega\|_{H^1(\Omega)} + C\|\uu^{n, \varepsilon}_\omega\|^\frac12\|\uu^{n, \varepsilon}_\omega\|_{\mathbf{V}_\sigma}^\frac{1}{2} \|\mu^{n, \varepsilon}_\omega\|_{H^1(\Omega)}\\
	   &\quad  \leq \dfrac{1}{4}\| \nabla \uu^{n, \varepsilon}_\omega\|^2 + \dfrac{1}{4}\| \nabla \mu^{n, \varepsilon}_\omega\|^2 + \dfrac{1}{4}\| \nabla \psi^{n, \varepsilon}_\omega\|^2 + C_{11}'\|\uu^{n, \varepsilon}_\omega\|^2,
		\end{split}
		\]
		where $C_{11}'>0$ only depend on $\Omega$.
		Then, recalling the estimates \eqref{eq:timederivatives}, \eqref{rhoL4h2}, we can derive the following differential inequality
		\begin{equation*}
			\td{\Lambda}{t} + \mathcal{G} \leq C(1 + \Lambda^2)\Lambda.
		\end{equation*}
		The quantity $\Lambda(0)$ can be controlled in a similar manner for the two dimensional case. Therefore, we conclude that there exists some $T^*\in(0,T]$ depending on the initial conditions (but not on $n, k, \varepsilon$) such that
\begin{equation*}
 \Lambda(t) + \ii{0}{t}{\mathcal{G}(\tau)}{\tau}\leq C, \quad \forall\,t\in (0,T^*],
\end{equation*}
where $C$ depends on $T^*$ and the initial data only. This easily yields the conclusion of this lemma.
\end{proof}

\subsection{Existence of strong solutions}\label{sec:exestr}
We are now in a position to prove the existence of strong solutions. Let us consider the two dimensional case first. From Lemma \ref{prop:ub5} and mimicking the proof of Lemma \ref{prop:ub4}, we can deduce that
\begin{lemma} \label{prop:ub6}
Let $d=2$.
The sequence $\{\phi^{n,\varepsilon}_\omega\}$ is uniformly bounded in $L^\infty(0,T;H^4(\Omega))$ and the sequence $\{\rho^{n,\varepsilon}_\omega\}$ is uniformly bounded in $L^\infty(0,T;H^2(\Omega))$.
	\end{lemma}

Summing up, Lemmas \ref{prop:ub5} and \ref{prop:ub6} say that \textcolor{black}{for every fixed $k > k^*$,}
\begin{align*}
		\uu^{n,\varepsilon}_\omega & \text{ is uniformly bounded in } L^\infty(0,T;\mathbf{V}_\sigma) \cap L^2(0,T;\mathbf{W}_\sigma) \cap H^1(0,T;\mathbf{H}_\sigma), \\
		\phi^{n,\varepsilon}_\omega & \text{ is uniformly bounded in } L^\infty(0,T;H^4(\Omega)) \cap H^1(0,T;H^2(\Omega)), \\
		\rho^{n,\varepsilon}_\omega & \text{ is uniformly bounded in } L^\infty(0,T;H^2(\Omega)) \cap H^1(0,T;H^1(\Omega)), \\
		\mu^{n,\varepsilon}_\omega & \text{ is uniformly bounded in } L^\infty(0,T;H^1(\Omega)),  \\
		\psi^{n,\varepsilon}_\omega & \text{ is uniformly bounded in } L^\infty(0,T;H^1(\Omega)),		
	\end{align*}
	with respect to the parameters $\omega, n, \varepsilon$, provided that $ \omega \in (0, 1]$,  $\varepsilon \in (0, \epsilon_3(k))$, and $n > n^{**}(k)$.

A standard compactness argument allows us to pass to the limit first as $n \to +\infty$,  then we let $\varepsilon \to 0^+$ and, finally, $k\to+\infty$ (see \cite{GMT18}). \textcolor{black}{We note that in the last limit procedure associated with $k$, \eqref{kkk} is no longer necessary so that the resulting estimates are indeed independent of $k$.} In this way, for any $\omega \in (0,1]$, we find a quintuplet $(\uu_{\omega}$, $\phi_{\omega}$, $\rho_{\omega}$, $\mu_{\omega}$, $\psi_{\omega})$ which preserves the regularity properties of the approximating sequence and solves the penalized problem \eqref{eq:strongNSCHpert}--\eqref{eq:bcspert}.

In order to take the limit $\omega \to 0^+$, we once again rely on the fact that $\rho_\omega \in L^\infty(\Omega \times (0,T))$, and, in particular, $0 \leq \rho_\omega \leq 1$ almost everywhere in $\Omega \times (0,T)$, just like in the proof for weak solutions. Then, we obtain the same energy estimates as in Lemma \ref{prop:enebound3} that are independent of $\omega$. In this way, repeating the previous argument, we are able to get further estimates that are independent of $\omega$. Hence, the usual compactness argument lets us get a quintuplet $(\uu, \phi, \rho, \mu, \psi)$ as $\omega \to 0^+$, which is a weak solution to the original problem  \eqref{eq:strongNSCH}--\eqref{eq:bcs} according to Definition \ref{def:solution} and has the following additional properties
\begin{align*}
        \uu & \in L^\infty(0,T;\mathbf{V}_\sigma) \cap L^2(0,T;\mathbf{W}_\sigma) \cap H^1(0,T;\mathbf{H}_\sigma), \\
		\phi & \in L^\infty(0,T;H^4_N(\Omega)) \cap H^1(0,T;H^2(\Omega)), \\
		\rho & \in L^\infty(0,T;H^2_N(\Omega)) \cap H^1(0,T;H^1(\Omega)), \\
		\mu & \in L^\infty(0,T;H^1(\Omega)),  \\
		\psi & \in L^\infty(0,T;H^1(\Omega)).		
\end{align*}
Then it is straightforward to check that
$\textcolor{black}{ \uu \cdot \nabla \phi \in L^2(0,T;H^2(\Omega))}$, $ \uu \cdot \nabla \rho \in L^2(0,T;H^1(\Omega))$. Also, by standard elliptic regularity estimates, we  deduce that $\textcolor{black}{\mu \in L^2(0,T;H^4(\Omega)\cap H^2_N(\Omega))}$ and $\psi\in L^2(0,T;H^3(\Omega)\cap H^2_N(\Omega))$. Therefore, we see that the solution $(\uu, \phi, \rho, \mu,\psi)$ is indeed a strong one, which satisfies the equations almost everywhere. In addition, using the same argument as in the proof of Theorem \ref{th:weaksolution}, we can deduce that  $\rho \in L^\infty(0,T; W^{2,p}(\Omega))$ for every $p \geq 2$ if $d = 2$ and $\phi \in L^\infty(0,T;H^5(\Omega))$. The pressure $\pi \in L^\infty(0,T;H^1(\Omega))$ can be recovered, up to a constant, through the De Rham theorem (see, for instance, \cite{BFNS, RTNS}).

Existence of strong solutions in the three dimensional case can be proved by arguing as in the two dimensional case. The only difference is that the higher-order estimates in Lemma \ref{prop:ub53d} are only local in time so that the strong solution is local in time as well, as expected. Besides, we can show $\rho \in L^\infty(0,T^*; W^{2,p}(\Omega))$ only for $p \in [2,6]$.

The existence part of Theorem \ref{th:wellposedlocal} is now proved.
\hfill$\square$

\subsection{Uniqueness of strong solutions}
On account of Theorem \ref{th:weaksolution}, we only need to consider the three dimensional case. Below we present some easy modifications of the argument in Subsection \ref{sec:proof2}, taking full advantage of the higher-order regularity properties of the strong solution. To this end, suppose that $(\uu_{0},\phi_{0}, \rho_{0})$ is an admissible set of initial data in the statement of Theorem \ref{th:wellposedlocal}. Then let $(\uu_1,\phi_1, \rho_1)$ and $(\uu_2,\phi_2, \rho_2)$ be two local strong solutions to problem \eqref{eq:strongNSCH}--\eqref{eq:bcs} defined in some time interval $[0,T^*]$ and both originating from $(\uu_{0},\phi_{0}, \rho_{0})$. Denoting by
$\mu_i$, $\psi_i$, $\pi_i$, $i = 1,2$, the corresponding chemical potentials and pressures, we set the differences by
\begin{align*}
      	&\uu = \uu_1 - \uu_2,  \quad \pi = \pi_1 - \pi_2, \quad \phi = \phi_1 - \phi_2,\\
		&\rho = \rho_1 - \rho_2,  \quad \mu = \mu_1 - \mu_2, \quad \psi= \psi_1 - \psi_2.
	\end{align*}
Using the Gagliardo--Nirenberg inequality in three dimensions (see \eqref{eq:gagnir}), we modify the estimate for $\mathcal{I}_3$ (cf. \eqref{eq:stab13}) as follows
\begin{align}
	\mathcal{I}_3
				& \leq C\|\Delta \phi\|_{L^3(\Omega)}\big( \|\Delta\phi\| +  \|\rho\|_{L^6(\Omega)} + \|\nabla \phi\|_{\LL^6(\Omega)}  + \|\nabla \rho\| \big) \notag \\
				& \leq C\|\phi\|^\frac{1}{6}\|\nabla \Delta \phi\|^\frac{5}{6} \left( \|\Delta\phi\| + \|\nabla \rho\| \right) \notag \\
				& \leq \dfrac{1}{18}\|\Delta \phi\|^2 + \dfrac{\alpha}{6}\|\nabla \Delta \phi\|^2 + \dfrac{\beta}{12}\| \nabla \rho\|^2  + C\| \phi\|^2.
\label{eq:stab133d}
\end{align}	
		The final result still reads as \eqref{eq:stab1final}, that is,
\begin{equation} \label{eq:stab1final-3d}
\dfrac{1}{2}\td{}{t}\| \phi\|^2 + \dfrac{5\alpha}{6}\| \nabla \Delta \phi\|^2+ \frac{5}{6}\|\Delta \phi\|^2
\leq \dfrac{\nu_*}{20}\|\uu\|^2 + \dfrac{\beta}{12}\| \nabla \rho\|^2
+ C\|\phi\|^2.		
\end{equation}
In a similar manner, we can derive (cf. \eqref{eq:stab2final})
\begin{equation} \label{eq:stab2final-3d}
			\dfrac{1}{2}\td{}{t}\| \rho \|^2_{V_0^*}
+ \frac{3\beta}{4}\|\nabla \rho\|^2 \leq \dfrac{\nu_*}{20}\|\uu\|^2 + \dfrac{1}{18}\| \nabla \phi\|^2
+   C\big(\|\phi\|^2+ \|\rho\|_{V_0^*}^2\big) .
\end{equation}
The estimate for $\mathcal{I}_7$ is revised as follows (cf. \eqref{eq:stab32})
		\begin{equation*}
			\begin{split}
				\mathcal{I}_7 & \leq \big( \|\uu_1\|_{\LL^6(\Omega)} + \|\uu_2\|_{\LL^6(\Omega)} \big) \|\uu\|\|\nabla \mathbf{A}^{-1}\uu\|_{\LL^3(\Omega)} \\
				& \leq C\|\uu\|_{\mathbf{V}_\sigma^*}^\frac{1}{2}\|\uu\|^\frac{3}{2} \\
				& \leq \dfrac{\nu_*}{20}\|\uu\|^2 + C\|\uu\|_{\mathbf{V}_\sigma^*}^2,
			\end{split}
		\end{equation*}
while for $\mathcal{I}_{13}$, we write (cf. \eqref{eq:stab38})
		\begin{equation*}
			\begin{split}
				\mathcal{I}_{13} & \leq C\|D\uu_2\|_{\LL^3(\Omega)}\|\nabla \mathbf{A}^{-1}\uu\|\big(\|\rho\|_{L^6(\Omega)} + \|\phi\|_{L^6(\Omega)} \big) \\
				& \leq \dfrac{\beta}{6}\|\nabla \rho\|^2 + \dfrac{1}{18}\|\Delta \phi\|^2 + C\|D\uu_2\|^2_{\LL^3(\Omega)}\|\uu\|^2_{\mathbf{V}_\sigma^*}.
			\end{split}
		\end{equation*}
Finally, for $\mathcal{I}_{15}$, we have (cf. \eqref{eq:stab310})
		\begin{equation*}
			\begin{split}
				\mathcal{I}_{15}
				& \leq C\big( \|\nabla \phi_1\|_{\LL^6(\Omega)} + \|\nabla \rho_1\|_{\LL^6(\Omega)}\big)\|\uu\|\|q\|_{L^3(\Omega)} \\
				& \leq C\|\uu\|\|q\|^\frac{1}{2}\|q\|_{H^1(\Omega)}^\frac{1}{2}\\
				& \leq C\|\uu\|^\frac{7}{4}\|\nabla \mathbf{A}^{-1}\uu\|^\frac{1}{4}\\
				& \leq \dfrac{\nu_*}{20}\|\uu\|^2 + C\|\uu\|_{\mathbf{V}_\sigma^*}^2.
			\end{split}
		\end{equation*}
Using the estimates for strong solution in higher norms on $[0,T^*]$, we arrive at (cf. \eqref{eq:stab3final})
		\begin{align}
		& \dfrac{1}{2}\td{}{t}\|\uu\|^2_{\mathbf{V}_\sigma^*} + \dfrac{7\nu_*}{20} \|\uu\|^2
- \dfrac{\alpha}{3}\|\nabla \Delta \phi\|^2
- \dfrac{5}{18}\|\Delta \phi\|^2
-\dfrac{\beta}{3}\|\nabla \rho\|^2
\notag \\
&\quad \leq \big( 1 +\|D\uu_2\|_{\LL^3(\Omega)}^2\big) \|\uu\|^2_{\mathbf{V}_\sigma^*} + C\big( \|\phi\|^2 + \|\rho\|_{V_0^*}^2\big).
\label{eq:stab3final-3d}
		\end{align}
Collecting \eqref{eq:stab1final-3d}, \eqref{eq:stab2final-3d}, and \eqref{eq:stab3final-3d}, we get the differential inequality
		\[
		\td{\widetilde{\mathcal{Y}}}{t} \leq \widetilde{\mathcal{H}} \widetilde{\mathcal{Y}}, \quad \text{for a.a.}\ t\in (0,T^*),
		\]
where
\begin{align*}
        \widetilde{\mathcal{Y}}(t) & := \|\uu(t)\|^2_{\mathbf{V}_\sigma^*} + \|\phi(t)\|^2 + \|\rho(t)\|^2_{V_0^*},\\
		\widetilde{\mathcal{H}}(t) &:= C\big( 1 +\|D\uu_2(t)\|_{\LL^3(\Omega)}^2 \big).
\end{align*}
Since $\widetilde{\mathcal{Y}}(0)=0$ and $\widetilde{\mathcal{H}} \in L^1(0,T^*)$, an application of Gronwall's lemma easily implies that $\widetilde{\mathcal{Y}}(t)\equiv 0$ for all $t\in [0,T^*]$. That is, the local strong solution to problem \eqref{eq:strongNSCH}--\eqref{eq:bcs} in three dimensions is unique.

The proof of Theorem \ref{th:wellposedlocal} is complete.
\hfill $\square$

\section{Further Results in Two Dimensions}
\label{sec:proof4}

Throughout this section we assume $d=2$. \textcolor{black}{An important consequence of the additional assumption \textbf{(H5)} is the  strict separation property for the strong solution component $\rho$ which is given by}
\begin{proposition}\label{prop:sepstr}
\textcolor{black}{Let $d=2$. In addition to the assumptions in Theorem \ref{th:wellposedlocal}, assume that \textbf{(H5)} is satisfied. Then for the global strong solution $\rho$, there exists $\eta \in (0,1/2]$ such that}
		\begin{equation}
		\eta \leq \rho(x,t) \leq 1-\eta,\quad \text{for all } x\in\Omega, \ t\geq 0.
        \label{regg2s}
		\end{equation}
		
\end{proposition}
\begin{proof}
In Theorem \ref{th:wellposedlocal}, we have shown that in two dimensions, $\phi\in L^\infty(0,+\infty; H^5(\Omega))$, $\psi\in L^\infty(0,+\infty; H^1(\Omega))$, whose corresponding norms are uniformly bounded in time. Besides, we have $\rho\in \mathcal{C}([0,+\infty); H^{1-r}(\Omega))$, $r\in (0,1/2)$, which implies $\rho\in \mathcal{C}([0,+\infty); \mathcal{C}(\overline{\Omega}))$ thanks to the Sobolev embedding theorem in two dimensions. Observe that
		\[
		\|\Delta|\nabla \phi|^2\|  \leq C \|\phi\|_{W^{3,4}(\Omega)}^2,
		\]
where the constant $C$ only depends on $\Omega$. Thus, $\Delta|\nabla \phi|^2 \in L^\infty(0,+\infty;L^2(\Omega))$ is uniformly bounded. Thus, it follows that the source term on the right-hand side in \eqref{eq:nnlp1} belongs to $L^\infty(0, +\infty; H^1(\Omega))$. This fact and the assumption \textbf{(H5)} enable us to apply \cite[Lemma 3.2]{HeWu} and conclude that
$$
\|\widehat{S}^\prime_\rho(\rho)\|_{L^\infty(0,+\infty; W^{1,p}(\Omega))}\leq C,\quad \forall\, p\in[2,+\infty),
$$
where $C>0$ is independent of $t$ (cf. \cite{GGM,GMT18}, where the argument therein requires the source term on the right-hand side in \eqref{eq:nnlp1} to be in $L^\infty(0, +\infty; H^2(\Omega))$). Then the above estimate implies $\eta\leq \|\rho\|_{L^\infty(0,+\infty; L^\infty(\Omega))}\leq 1-\eta$ for some $\eta\in (0,1/2]$, which combined with the continuity of $\rho$ yield the strict separation property \eqref{regg2s}. We note that under the additional assumption \textbf{(H5)}, the initial datum $\rho_0$ is already strictly separated from the pure states $0$ and $1$.
\end{proof}

\subsection{Continuous dependence estimate for strong solutions}
\label{subs:contdep}
	\textbf{Proof of Theorem \ref{th:contdep}}.
Let us consider two sets of admissible initial data $(\phi_{01}, \rho_{01}, \uu_{01})$, $(\phi_{02}, \rho_{02},\uu_{02})$ in $H^5(\Omega) \times H^2(\Omega) \times \mathbf{V}_\sigma$ (namely, complying with the hypotheses of Theorem \ref{th:wellposedlocal}). Let $(\phi_1, \rho_1, \uu_1)$ and $(\phi_2, \rho_2, \uu_2)$ be the corresponding strong solutions to problem \eqref{eq:strongNSCH}--\eqref{eq:bcs}, with chemical potentials $\mu_i$ and $\psi_i$ for $i = 1,2$, accordingly. Recalling \eqref{initial}, we consider again the system \eqref{eq:stab0} equipped with
\begin{equation*}
		\begin{cases}
\mathbf{u} = \mathbf{0} & \quad \text{a.e. on } \partial \Omega \times (0,T), \\
			\partial_\mathbf{n}\phi = \partial_\mathbf{n}\Delta \phi =\partial_\mathbf{n}\mu=0 & \quad \text{a.e. on } \partial \Omega \times (0,T), \\
				\partial_\mathbf{n}\rho = \partial_\mathbf{n}\psi = 0 & \quad \text{a.e. on } \partial \Omega \times (0,T), \\
		\mathbf{u}|_{t=0} = \mathbf{u}_0,\ \	\phi|_{t=0} = \phi_0,\ \ \rho|_{t=0} = \rho_0 & \quad \text{a.e. in } \Omega,
		\end{cases}
	\end{equation*}
where
$$
\quad \uu_0 = \uu_{01}-\uu_{02},\quad \phi_0 = \phi_{01} - \phi_{02}, \quad \rho_0 = \rho_{01} - \rho_{02}.
$$	

Taking advantage of higher regularity of strong solutions (recall Theorem \ref{th:wellposedlocal}), we proceed to estimate their difference in stronger norms (cf. the argument in Subsection \ref{sec:proof2}). Besides, in the present case we have to take care of the fact that the initial data are no longer null.

		\textbf{Step 1}. Testing the evolution equation for $\phi$ in \eqref{eq:stab0} by \textcolor{black}{$\Delta^2 \phi$,} we have
		\[
		\textcolor{black}{
        ( \partial_t\phi,  \Delta^2 \phi) + ( \uu_1 \cdot \nabla \phi, \Delta^2 \phi)
        +  ( \uu \cdot \nabla \phi_2, \Delta^2 \phi) = (\Delta \mu, \Delta^2 \phi).
        }
		\]
		Using the expression of $\mu$ and integration by parts, we get
\begin{equation} \label{eq:hstab1}
			\dfrac{1}{2}\td{}{t}\| \Delta \phi\|^2 + \alpha\|\nabla \Delta^2 \phi\|^2 + \|\Delta^2 \phi\|^2 = \mathcal{J}_1 + \mathcal{J}_2 + \mathcal{J}_3 + \mathcal{J}_4,
\end{equation}
where we have set
\begin{align*}
			\mathcal{J}_1 & := -( \uu_1 \cdot \nabla \phi, \Delta^2 \phi) -  ( \uu \cdot \nabla \phi_2, \Delta^2 \phi) \\
			\mathcal{J}_2 & := -\big( \nabla (S'_\phi(\phi_1) - S'_\phi(\phi_2)), \nabla \Delta^2\phi\big) \\
			\mathcal{J}_3 & := - \theta\big(\nabla (\nabla \cdot(\rho_1 \nabla \phi)), \nabla \Delta^2 \phi\big) \\
            \mathcal{J}_4 & := - \theta\big(\nabla (\nabla \cdot(\rho \nabla \phi_2)), \nabla \Delta^2\phi\big).
\end{align*}
By the Gagliardo--Nirenberg inequality \eqref{eq:gagnir} and elliptic estimates, we deduce that
\begin{align}
\mathcal{J}_1 & \leq \textcolor{black}{	\|\uu_1\|_{\LL^6(\Omega)}\|\nabla \phi\|_{\mathbf{L}^3(\Omega)}\|\Delta^2\phi\| + \|\nabla\phi_2\|_{\LL^\infty(\Omega)}\|\uu\|\|\Delta^2 \phi\| } \notag \\
				& \leq C\|\nabla \phi\|_{\mathbf{H}^1(\Omega)}\|\Delta\phi\|^\frac{1}{3}\|\nabla \Delta^2 \phi\|^\frac{2}{3} + C\|\uu\|\|\Delta\phi\|^\frac{1}{3}\|\nabla \Delta^2 \phi\|^\frac{2}{3} \notag \\
				& \leq \dfrac{\alpha}{6}\|\nabla \Delta^2\phi\|^2 + C\big( \|\uu\|^2 + \| \Delta \phi\|^2 \big).
\label{eq:hstab11}
\end{align}
In the above estimates, we have used the mass conservation property such that $$\overline{\phi}(t)=\overline{\phi_0},\quad \forall\,t\geq 0,$$
and the elliptic estimate given by Lemma \ref{th:ellreg}. Next, thanks to the regularity of $S_\phi$, we get
\begin{align}
\mathcal{J}_2 & \leq \textcolor{black}{ \left|\left(\nabla (S'_\phi(\phi_1) - S'_\phi(\phi_2)), \nabla \Delta^2 \phi \right)\right| } \notag  \\
				& \leq \left|\left(\int_0^1S''_\phi(\tau\phi_1+(1-\tau)\phi_2)\nabla \phi\,\mathrm{d}\tau,\, \nabla \Delta^2\phi\right)\right|\notag \\
&\quad +\left|\left(\int_0^1 S'''_\phi(\tau\phi_1+(1-\tau)\phi_2)\nabla (\tau\phi_1+(1-\tau)\phi_2) \phi\,\mathrm{d}\tau,\, \nabla \Delta^2\phi\right)\right|  \notag \\
				& \leq C \|\nabla \phi\|\| \nabla \Delta^2 \phi\|
                +C \|\phi\|\|\nabla \Delta^2\phi\|\notag \\
				& \leq \dfrac{\alpha}{6}\|\nabla \Delta^2 \phi\|^2 + C\|\phi\|_{H^1(\Omega)}^2 \notag \\
&\leq \dfrac{\alpha}{6}\|\nabla \Delta^2 \phi\|^2 + C(\|\Delta \phi\|^2+|\phi|^2),
\label{eq:hstab12}
\end{align}
whereas for $\mathcal{J}_3$, applications of H\"{o}lder's inequality and Young's inequality entail
\begin{align}
   \mathcal{J}_3
   &\leq \textcolor{black}{ C\|\rho_1\nabla \phi\|_{H^2(\Omega)}\|\nabla\Delta^2\phi\| } \notag\\
   & \leq C\|\rho_1\|_{L^\infty(\Omega)} \|\nabla \phi\|_{\mathbf{H}^2(\Omega)}\|\nabla \Delta^2\phi\| +  C\|\rho_1\|_{H^2(\Omega)}\|\nabla \phi\|_{\mathbf{L}^\infty(\Omega)}\|\nabla \Delta^2\phi\|  \notag \\
   & \leq C\|\nabla \phi\|_{\mathbf{H}^2(\Omega)} \|\nabla \Delta^2\phi\| \notag \\
   & \leq C \|\Delta \phi\|^\frac{2}{3} \|\nabla \Delta^2\phi\|^\frac{4}{3}   \notag \\
   & \leq \dfrac{\alpha}{12}\|\nabla \Delta^2 \phi\|^2  + C \|\Delta\phi\|^2.
\label{eq:hstab13}
\end{align}
 In a similar manner, using that
 mass conservation property such that $$\overline{\rho}(t)=\overline{\rho_0},\quad \forall\,t\geq 0,$$
 we have
\begin{align}
   \mathcal{J}_4
   & \leq \textcolor{black}{C\|\rho\nabla \phi_2\|\|\nabla \Delta^2\phi\| } \notag\\
   & \leq C\|\rho\|_{L^\infty(\Omega)} \|\nabla \phi_2\|_{H^2(\Omega)}\|\nabla \Delta^2\phi\|
      +  C \|\rho\|_{H^2(\Omega)} \|\nabla \phi_2\|_{\mathbf{L}^\infty(\Omega)}\|\nabla \Delta^2\phi\| \notag \\
  & \leq C (\|\Delta \rho\|+|\overline{\rho}|) \|\nabla \Delta^2\phi\|  \notag \\
				& \leq \dfrac{\alpha}{12}\|\nabla \Delta^2 \phi\|^2 + \dfrac{\beta}{12}\| \nabla \Delta \rho\|^2 + C\big( \|\nabla \rho \|^2 +|\overline{\rho}|^2 \big).
\label{eq:hstab13b}
\end{align}
		On account of estimates \eqref{eq:hstab11}--\eqref{eq:hstab13}, from \eqref{eq:hstab1} and the Poincar\'{e}--Wirtinger inequality we deduce that
\begin{equation} \label{eq:hstab1final}
			\dfrac{1}{2}\td{}{t}\| \Delta \phi\|^2 + \dfrac{\alpha}{2}\|\nabla \Delta^2 \phi\|^2 + \|\Delta^2 \phi\|^2 -  \dfrac{\beta}{12}\| \nabla \Delta \rho\|^2
\leq C\big( \|\uu\|^2+ \|\Delta  \phi\|^2 + \|\nabla \rho \|^2+ |\overline{\phi}|^2+ |\overline{\rho}|^2   \big).
\end{equation}

		\textbf{Step 2}. Taking $\textcolor{black}{-\Delta \rho}$ as a test function in \eqref{eq:stab0} yields
		\[
		\textcolor{black}{
        ( \partial_t\rho, -\Delta \rho ) - (\mathbf{u}_1\cdot \nabla \rho, \Delta \rho) - (\mathbf{u}\cdot \nabla \rho_2, \Delta \rho)= -(\Delta \psi, \Delta \rho).}
		\]
		Using the expression of $\psi$ and integration by parts, we get
		\begin{equation} \label{eq:hstab2}
			\dfrac{1}{2}\td{}{t}\| \nabla \rho \|^2 + \beta\| \nabla \Delta\rho\|^2 = \mathcal{J}_5 + \mathcal{J}_6 + \mathcal{J}_7,
		\end{equation}
		where
		\begin{align*}
			\mathcal{J}_5 & := (\mathbf{u}_1\cdot \nabla \rho, \Delta \rho) +  (\mathbf{u}\cdot \nabla \rho_2, \Delta \rho),\\
			\mathcal{J}_6 & := \big(\nabla (S'_\rho(\rho_1)-S'_\rho(\rho_2)), \nabla \Delta \rho\big), \\
			\mathcal{J}_7 & := -\dfrac{\theta}{2}(\nabla (\nabla (\phi_1+ \phi_2)\cdot\nabla \phi), \nabla \Delta \rho).
		\end{align*}
For $\mathcal{J}_5$, we have
\begin{align}
		\mathcal{J}_5 & \textcolor{black}{ \leq \|\mathbf{u}_1\|_{\mathbf{L}^4(\Omega)}\|\nabla \rho\|_{\mathbf{L}^4(\Omega)}\|\Delta \rho\|+ \|\nabla \rho_2\|_{\mathbf{L}^4(\Omega)}\|\uu\|\|\Delta \rho\|_{L^4(\Omega)} } \notag \\
        &\leq C\|\nabla \rho\|^\frac54\|\nabla \Delta\rho\|^\frac34+ C\|\mathbf{u}\|\|\nabla \rho\|^\frac14\|\nabla \Delta\rho\|^\frac34\notag \\
				& \leq \frac{\beta}{12}\|\nabla \Delta\rho\|^2+ C\big( \|\uu\|^2 + \|\nabla \rho\|^2 \big),
        \label{eq:hstab21}
\end{align}
Next, \textcolor{black}{using the strict separation property \eqref{regg2s} for $\rho_1$, $\rho_2$,} we can deduce that
 \begin{align}
		\mathcal{J}_6 & = \textcolor{black}{ \left(\nabla  \ii{0}{1}{S''_\rho(s\rho_1 +(1-s)\rho_2)\rho}{s}, \nabla \Delta \rho \right) } \notag  \\
        &\leq \left|\left(  \ii{0}{1}{S''_\rho(s\rho_1 +(1-s)\rho_2)}{s}\nabla\rho, \nabla \Delta \rho \right)\right|\notag\\
        &\quad + \left|\left(  \ii{0}{1}{S'''_\rho(s\rho_1 +(1-s)\rho_2)\nabla (s\rho_1 +(1-s)\rho_2)}{s}\rho, \nabla \Delta \rho \right)\right| \notag\\
		& \leq C\|\rho\|_{H^1(\Omega)}\|\nabla \Delta \rho\| \notag \\
		& \leq \dfrac{\beta}{12}\|\nabla \Delta \rho\|^2 + C(\|\nabla \rho\|^2+|\overline{\rho}|^2).
			\label{eq:hstab22}
\end{align}
Finally, for $\mathcal{J}_7$, it holds
\begin{align}
\mathcal{J}_7
& \leq \textcolor{black}{  C\big(\| \nabla \phi_1\|_{\LL^\infty(\Omega)} + \| \nabla \phi_2\|_{\LL^\infty(\Omega)}\big)\| \nabla \phi \|_{H^1(\Omega)} \|\nabla \Delta \rho\| } \notag\\
&\quad  +
C\big(\| \phi_1\|_{W^{2,\infty}(\Omega)} + \| \phi_2\|_{W^{2,\infty}(\Omega)}\big)\| \nabla \phi \|\|\nabla \Delta \rho\|
 \notag \\
& \leq \dfrac{\beta}{12}\|\nabla \Delta \rho\|^2 + C(\|\Delta \phi \|^2+|\overline{\phi}|^2).
\label{eq:hstab23a}
\end{align}
Thus, estimates \eqref{eq:hstab21}--\eqref{eq:hstab23a} combined with \eqref{eq:hstab2} give
\begin{equation} \label{eq:hstab2final}
\textcolor{black}{ \dfrac{1}{2}\td{}{t}\|\nabla \rho \|^2 + \dfrac{3\beta}{4} \| \nabla \Delta\rho\|^2 \leq C\big( \|\uu\|^2 + \|\Delta \phi\|^2 + \|\nabla \rho\|^2 + |\overline{\phi}|^2+ |\overline{\rho}|^2 \big).  }
\end{equation}

\textbf{Step 3}. We get rid of the chemical potentials exactly as in Step 3 of the uniqueness proof of Theorem \ref{th:weaksolution}. This procedure yields
		\begin{equation} \label{eq:hstab3}
			\dfrac{1}{2}\td{}{t}\|\uu\|^2 + (\nu(\phi_1,\rho_1)D\uu, D\uu)= \sum_{k = 8}^{14} \mathcal{J}_{k},
		\end{equation}
		where we have set
		\begin{align*}
			\mathcal{J}_8 & := -((\uu\cdot \nabla)\uu_2, \uu)  \\
			\mathcal{J}_9 & := (\nabla \phi_1 \otimes \nabla \phi, \nabla \uu) + (\nabla \phi \otimes \nabla \phi_2, \nabla \uu), \\
			\mathcal{J}_{10} & := \beta(\nabla \rho_1 \otimes \nabla \rho, \nabla \uu) + \beta(\nabla \rho \otimes \nabla \rho_2, \nabla \uu), \\
			\mathcal{J}_{11} & := -\theta(\rho_1\nabla \phi_1 \otimes \nabla \phi, \nabla \uu) - \theta(\rho_1\nabla \phi \otimes \nabla \phi_2, \nabla \uu) - \theta(\rho\nabla \phi_2 \otimes \nabla \phi_2, \nabla \uu), \\
			\mathcal{J}_{12} & := \alpha(\nabla\Delta\phi_1 \otimes \nabla \phi, \nabla \uu) + \alpha(\nabla \Delta \phi \otimes \nabla \phi_2, \nabla \uu),\\
			\mathcal{J}_{13} & := -\alpha((\nabla\Delta\phi_1 \cdot \nabla)\nabla\phi,  \uu) - \alpha((\nabla\Delta\phi \cdot \nabla)\nabla\phi_2,  \uu),\\
			\mathcal{J}_{14} & := - ((\nu(\phi_1, \rho_1) - \nu(\phi_2, \rho_2))D\uu_2, \nabla \uu).
		\end{align*}
		We now provide controls for each of the terms defined above.
		Using the H\"{o}lder inequality and the embedding $\mathbf{V}_\sigma  \hookrightarrow \LL^6(\Omega)$
\begin{equation} \label{eq:hstab32}
			\mathcal{J}_8 \leq \|\uu\|\|\nabla \uu_2\|_{\LL^3(\Omega)}\|\uu\|_{\LL^6(\Omega)}
				\leq \dfrac{\nu_*}{24}\|\nabla \uu\|^2 + C\|\nabla \uu_2\|_{\LL^3(\Omega)}^2\|\uu\|^2,
\end{equation}
		whereas,
		\begin{align}
				\mathcal{J}_9 &\leq \left( \| \nabla \phi_1 \|_{\LL^\infty(\Omega)} + \| \nabla \phi_2 \|_{\LL^\infty(\Omega)} \right) \| \nabla \phi \| \|\nabla \uu\| \notag\\
				& \leq \dfrac{\nu_*}{24}\|\nabla \uu\|^2 + C\|\nabla \phi\|^2\notag\\
                & \leq \dfrac{\nu_*}{24}\|\nabla \uu\|^2 + C\big(\|\Delta \phi\|^2+|\overline{\phi}|^2\big).
                \label{eq:hstab33}
		\end{align}
		The term $\mathcal{J}_{10}$ can be treated in a similar way,
		\begin{align}
			\mathcal{J}_{10} & \leq \textcolor{black}{\left( \| \nabla \rho_1 \|_{\LL^4(\Omega)} + \| \nabla \rho_2 \|_{\LL^4(\Omega)} \right) \| \nabla \rho \|_{\LL^4(\Omega)} \|\nabla \uu\| }\notag \\
				& \leq \dfrac{\nu_*}{24}\|\nabla \uu\|^2 + C\|\nabla \rho\|\|\Delta\rho\| \notag \\
				& \leq \dfrac{\nu_*}{24}\|\nabla \uu\|^2 + \dfrac{\beta}{12}\|\nabla \Delta \rho\|^2 + C\|\nabla \rho\|^2.
			\label{eq:hstab34}
		\end{align}
Also for $\mathcal{J}_{11}$, it holds
\begin{align}
		\mathcal{J}_{11} & \leq \dfrac{\nu_*}{24}\|\nabla \uu\|^2 + C\|\rho_1\|_{L^\infty(\Omega)}^2\big( \| \nabla \phi_1 \|_{\LL^\infty(\Omega)}^2 + \| \nabla \phi_2 \|_{\LL^\infty(\Omega)}^2 \big)\|\nabla \phi\|^2 + C\| \rho\| ^2\| \nabla \phi_2 \|_{\LL^\infty(\Omega)}^4 \notag \\
		& \leq \dfrac{\nu_*}{24}\|\nabla \uu\|^2 + C\big(  \|\nabla \phi\|^2 +\|\rho\|^2 \big)\notag\\
&\leq  \dfrac{\nu_*}{24}\|\nabla \uu\|^2 + C\big(  \|\Delta \phi\|^2 +\|\nabla \rho\|^2 + |\overline{\rho}|^2\big).
	\label{eq:hstab35}
\end{align}
As far as $\mathcal{J}_{12}$ and $\mathcal{J}_{13}$ are concerned, we have
\begin{align}
		\mathcal{J}_{12} & \leq  \textcolor{black}{ C\|\nabla \Delta \phi_1\|_{\LL^\infty}\|\nabla \phi\|\| \nabla \uu\| + C\|\nabla \Delta \phi\|\|\nabla \phi_2\|_{\LL^\infty(\Omega)}\|\nabla \uu\| } \notag \\
&\leq C\|\Delta \phi\|\| \nabla \uu\| + C\|\Delta\phi\|^\frac12\|\Delta^2\phi\|^\frac{1}{2}\|\nabla \uu\|\notag\\
				& \leq \dfrac{\nu_*}{24}\|\nabla \uu\|^2 + \dfrac{1}{4}\|\Delta^2 \phi\|^2 + C\|\Delta \phi\|^2,
		\label{eq:hstab36}
\end{align}
\begin{align}
\mathcal{J}_{13}
& \leq \textcolor{black}{ C\|\nabla \Delta \phi_1\|_{\LL^\infty}\|\nabla \phi\|_{\mathbf{H}^1(\Omega)}\|\uu\| + C\|\nabla \Delta \phi\|\|\nabla\phi_2\|_{\mathbf{W}^{1,\infty}(\Omega)}\|\uu\| } \notag  \\
&\leq C\|\Delta \phi\|\|  \uu\| + C\|\Delta\phi\|^\frac12\|\Delta^2\phi\|^\frac{1}{2}\| \uu\|  \notag\\
	& \leq \dfrac{1}{4}\| \Delta^2 \phi\|^2 + C(\|\uu\|^2+\|\Delta \phi\|^2).
		\label{eq:hstab37}
\end{align}
Finally, for $\mathcal{J}_{14}$, we have
\begin{align}
	\mathcal{J}_{14} & \leq \textcolor{black}{ C\left( \|\phi\|_{L^\infty(\Omega)}\|D\uu_2\|\|\nabla \uu\| + \|\rho\|_{L^\infty(\Omega)}\|D\uu_2\|\|\nabla \uu\|\right) } \notag \\
				& \leq \dfrac{\nu_*}{24}\| \nabla \uu\|^2 + C\big( \|\phi\|_{H^2(\Omega)}^2 + \|\rho\|_{H^2(\Omega)}^2 \big) \notag \\
				& \leq \dfrac{\nu_*}{24}\| \nabla \uu\|^2 + C\big( \|\Delta \phi\|^2 + \|\Delta \rho\|^2 + |\overline{\phi}|^2 +|\overline{\rho}|^2 \big) \notag \\	
				& \leq \dfrac{\nu_*}{24}\| \nabla \uu\|^2 + \dfrac{\beta}{12}\|\nabla \Delta \rho\|^2 + 	C\big( \|\Delta \phi\|^2 + \|\nabla \rho\|^2 + |\overline{\phi}|^2 +|\overline{\rho}|^2 \big).		
			\label{eq:hstab38}
\end{align}
Thus, collecting the results \eqref{eq:hstab32}--\eqref{eq:hstab38}, from \eqref{eq:hstab3} we infer that
\begin{align}
& \dfrac{1}{2}\td{}{t}\|\uu\|^2 + \dfrac{\nu_*}{4} \|\nabla \uu\|^2
- \dfrac{1}{2}\|\Delta^2 \phi\|^2
-\dfrac{\beta}{6}\|\nabla \Delta \rho\|^2
\notag \\
&\quad \leq  C\big(1+\|\nabla \uu_2\|_{\LL^3(\Omega)}^2\big)\big(\|\uu\|^2+ \|\Delta \phi\|^2 + \|\nabla \rho\|^2\big) + C\big(|\overline{\phi}|^2 +|\overline{\rho}|^2  \big).
\label{eq:hstab3final}
\end{align}

		\textbf{Step 4.} Set
		\begin{align*}
			\mathcal{Z}(t) & := \|\uu(t)\|^2 + \|\Delta \phi(t)\|^2 +\|\nabla \rho(t)\|^2 ,\\
			\mathcal{W}(t) & := \dfrac{\nu_*}{2}\|\nabla \uu(t)\|^2+ \alpha\|\nabla \Delta^2 \phi(t)\|^2 + \|\Delta^2 \phi\|^2 + \beta\|\nabla \Delta \rho(t)\|^2.
		\end{align*}
		Collecting \textcolor{black}{ \eqref{eq:hstab1final}, \eqref{eq:hstab2final}, and \eqref{eq:hstab3final},} we arrive at
\begin{equation*} \label{eq:stab41}
			\td{\mathcal{Z}}{t} + \mathcal{W} \leq C\big(1+\|\nabla \uu_2\|_{\LL^3(\Omega)}^2\big)\mathcal{Z} + C\big(|\overline{\phi_0}|^2+|\overline{\rho_0}|^2\big),
\end{equation*}
where we have used again the mass conservation property for $\phi$ and $\rho$. Noticing that $\|\nabla \uu_2\|_{\LL^3(\Omega)}^2\in L^1(0,T)$, therefore, an application of Gronwall's lemma implies that
\begin{equation} \label{eq:stabfinal}
			 \|\uu(t)\|^2+ \|\phi(t)\|_{H^2(\Omega)}^2 + \|\rho(t)\|^2_{H^1(\Omega)} +  \ii{0}{t}{\mathcal{W}(\tau)}{\tau}
\leq C_T\left(  \|\uu_0\|^2+\|\Delta \phi_0\|^2 + \|\nabla \rho_0\|^2 + |\overline{\phi_0}|^2+|\overline{\rho_0}|^2\right),
\end{equation}
for any $t\in [0,T]$, where $C_T>0$ is a constant depending on $\|\uu_{0i}\|_{\mathbf{V}_\sigma}$, $\|\phi_{0i}\|_{H^5(\Omega)}$, $\|\rho_{0i}\|_{H^2(\Omega)}$, $\|\psi_{0i}\|_{H^1(\Omega)}$, $i=1,2$, coefficients of the system, $\Omega$ and $T$.

The proof of Theorem \ref{th:contdep} is complete. \hfill $\square$

\subsection{Instantaneous regularization of weak solutions}
\label{sec:regularity}

We now establish the instantaneous regularization property of global weak solutions in dimension two. In particular, we show that every weak solution becomes a strong one as long as $t>0$.

\vspace{0.3\baselineskip}

\textbf{Proof of Theorem \ref{th:regularity3}.}
		Let $\delta > 0$ be an arbitrary but fixed constant. Consider the global weak solution $(\uu, \phi, \rho, \mu, \psi)$ defined on $[0,+\infty)$ (see Remark \ref{globalsol}). Owing to the regularity properties provided in Theorem \ref{th:weaksolution}, as well as the dissipative nature of the system (see \eqref{weak-IEN}), we can infer the existence of some Lebesgue point $\xi \in (\delta/2, \delta)$ such that, taking \textcolor{black}{$(\uu(\xi), \phi(\xi), \rho(\xi))$} as initial data, the hypotheses of Theorem \ref{th:wellposedlocal} are satisfied. Thus, we can infer the existence of a strong solution on the time interval $[\xi, T]$, for any given $T>\xi$.
Then the uniqueness result implies that the strong and weak solutions coincide on $[\xi,T]$ so that the global weak solution constructed above is indeed a strong one on $[\xi,+\infty)$.

Next, we prove some uniform in time estimates for $t\geq \delta$. The proof relies on   computations that are parallel to those performed in the proof of Lemma \ref{prop:ub5}. Hence, we just sketch the main steps.

Following the argument in the proof of Lemma \ref{prop:ub5}, we set
		\[
		\begin{split}
			\widetilde{\Lambda}(t)
&:=  \dfrac{1}{2}\|\nabla\uu(t)\|^2 +   \dfrac{1}{2}\|\nabla \mu(t)\|^2 + \dfrac{1}{2}\|\nabla \psi(t)\|^2 + (\uu(t) \cdot \nabla \rho(t), \psi(t)) +(\uu(t) \cdot \nabla \phi(t), \mu(t)) \textcolor{black}{ +\widetilde{C}_{11}\|\uu(t)\|^2, } \\
			\widetilde{\mathcal{G}}(t) & :=  \frac{\nu_*}{32\widetilde{C}_{10}} \|\partial_t\uu(t)\|^2+ \dfrac{\nu_*}{4}\|\mathbf{A}\uu(t)\|^2+\dfrac{\alpha}{2}\|\Delta \partial_t \phi(t)\|^2  + \dfrac{1}{2}\|\nabla \partial_t \phi(t)\|^2 +\dfrac{\beta}{2}\|\nabla \partial_t \rho(t)\|^2 ,
		\end{split}
		\]
for some $\widetilde{C}_{10}>0$ depending on the initial energy, $\Omega$ and coefficients of the system, $\widetilde{C}_{11}>0$ depending only on $\Omega$.
Then it is straightforward to check that an inequality similar to \eqref{equiL} holds for $\widetilde{\Lambda}$ with $C_{11}$ replaced by $\widetilde{C}_{11}$.  Moreover, we obtain
\begin{equation} \label{eq:diffineq}
			\td{\widetilde{\Lambda}}{t} + \widetilde{\mathcal{G}} \leq \widetilde{C}_{12}\big(1 + \widetilde{\Lambda}\big)\widetilde{\Lambda},
\end{equation}
for a.a. $t\in (\xi,+\infty)$. Since $\delta>0$ can be arbitrary small, \eqref{eq:diffineq} holds for a.a. $t\in (0,+\infty)$. Then, recalling Remark \ref{Rm-ill}, from the energy inequality \eqref{weak-IEN} we deduce that
\[
\ii{t}{t+1}{  \|\sqrt{\nu(\phi(\tau),\rho(\tau))}D \mathbf{u}(\tau)\|^2+ \| \nabla \mu(\tau)\|^2 + \| \nabla \psi(\tau) \|^2}{\tau} \leq C,
\]
for every $t \geq 0$, where the constant $C$ is independent of $t$. This fact together with an analogue of \eqref{equiL} for $\widetilde{\Lambda}$ yields that
\begin{equation*}
\ii{t}{t+1}{\widetilde{\Lambda}(\tau)}{\tau} \leq C,
\end{equation*}
for every $t \geq 0$, where the constant $C$ is again independent of $t$. Hence, we can apply the uniform Gronwall lemma (see e.g., \cite[Chapter III, Lemma 1.1]{RTDDS}), choosing $r = \delta$ therein, to obtain the estimate $\widetilde{\Lambda}(t)\leq C$, for all $t\geq \delta$, where $C>0$ depends on $K$, $m_1$, $m_2$, $\delta$, coefficients of the system $\Omega$, but not on $t$. Hence, we get
\begin{align}
			 \|\uu(t)\|_{\mathbf{V}_\sigma}
			+  \|\mu(t)\|_{H^1(\Omega)}
			+  \|\psi(t)\|_{H^1(\Omega)} \leq C,\quad \forall\,t\geq \delta.
			\label{reg1}
\end{align}
Integrating \eqref{eq:diffineq} over $[t,t+1]$ yields, by definition of $\widetilde{\mathcal{G}}$, the following
\begin{align}
			\|\uu\|_{L^2(t,t+1;\mathbf{W}_\sigma)} + \|\partial_t \uu\|_{L^2(t,t+1;\mathbf{H}_\sigma)} + \|\partial_t \phi\|_{L^2(t,t+1;H^2(\Omega))} + \|\partial_t \rho\|_{L^2(t,t+1;H^1(\Omega))} \leq C,
			\label{reg2}
\end{align}
for all $t\geq \delta$. The estimate \eqref{reg1} further implies that
\begin{align}
			 \|\phi(t)\|_{H^5(\Omega)}
			+  \|\rho(t)\|_{W^{2,p}(\Omega)} \leq C,\quad \forall\,t\geq \delta,
			\label{reg1a}
\end{align}
for any  $p\in [2,+\infty)$. Moreover, thanks to Proposition \ref{prop:sepstr}, we have
\begin{equation}
		\eta \leq \rho(x,t) \leq 1-\eta,\quad \text{for all } x\in\Omega, \ t\geq \delta,
        \label{regg2sw}
		\end{equation}
for some $\eta\in (0,1/2]$, depends on $K$, $m_1$, $m_2$, $\delta$, coefficients of the system $\Omega$, but not on $t$. Hence, the estimates \eqref{regg1} and \eqref{regg2} are proved.

Higher-order estimate can be derived by using the method of difference quotients (cf. \cite{GMT18}). Given any function $f: [a,+\infty) \to X$, with $X$ being a Banach space, we define
	\[
	\partial_t^h f := \dfrac{f(t+h)-f(t)}{h}
	\]
	for any $h > 0$ and $t \in [a,+\infty)$. Arguing as in Subsection \ref{subs:contdep}, we can derive a system for the difference quotients. Performing the same computations as in the proof of Theorem \ref{th:contdep} (that is, for differences of two solutions) and setting
\begin{align*}
			\widehat{\mathcal{Z}}(t) & :=  \|\dq\uu(t)\|^2
+\|\Delta \dq{\phi}(t)\|^2 +\|\nabla \dq\rho(t)\|^2,\\
			\widehat{\mathcal{W}}(t) & :=  \dfrac{\nu_*}{2}\|\nabla \dq\uu(t)\|^2 + \alpha\|\nabla \Delta^2 \dq \phi(t)\|^2 + \|\Delta^2 \dq \phi(t)\|^2 + \beta\|\nabla \Delta \dq\rho(t)\|^2,
\end{align*}
we deduce that
\begin{equation*} \label{eq:diffineq2}
			\td{\widehat{\mathcal{Z}}}{t} + \widehat{\mathcal{W}} \leq C\big( 1 + \|\nabla \uu(t+h)\|^2_{\LL^3(\Omega)}\big)\widehat{\mathcal{Z}},
\end{equation*}
where we recall that by conservation of mass $\overline{\dq{\phi}} =\overline{\dq{\rho}}= 0$ for every $t \geq 0$ and $h > 0$. Owing to \eqref{reg1}, \eqref{reg2} and the standard estimate $\|\dq{f}\|_{L^2(t,t+1;L^2(\Omega))} \leq \|\partial_t f\|_{L^2(t,t+2;L^2(\Omega))}$, it holds that
		\[
		\ii{t}{t+1}{\Big(\widehat{\mathcal{Z}}(\tau) + \textcolor{black}{\|\uu(\tau)\|^2_{\mathbf{W}^{1,3}(\Omega)}} \Big)}{\tau} \leq C,
		\]
		for every $t \geq \delta$. Moreover, the constant $C$ is independent of $h$. A further application of the uniform Gronwall lemma with $r = \delta$ and then letting $h \to 0^+$ yields the following estimates
\begin{equation} \label{reg3}
			\|\partial_t \uu(t)\| + \|\partial_t\phi(t)\|_{H^2(\Omega)} + \|\partial_t\rho(t)\|_{H^1(\Omega)} \leq C,\quad \forall\, t\geq 2\delta,
\end{equation}
and
\[
\|\partial_t \uu\|_{L^2(t,t+1;\mathbf{V}_\sigma)} + \|\partial_t \phi\|_{L^2(t,t+1;H^5(\Omega))} +  \|\partial_t \rho\|_{L^2(t,t+1;H^3(\Omega))} \leq C,\quad \forall\,t \geq 2\delta.
\]
Namely, the estimates \eqref{regg11} and \eqref{regg12} are obtained.

Consider the Navier--Stokes system written as follows
		\begin{equation} \label{eq:dnlp2}
			\begin{cases}
				-\mathrm{div}(\nu(\phi, \rho)D \uu) + \nabla\pi = \mu\nabla \phi + \psi\nabla \rho - \partial_t\uu - (\uu \cdot \nabla)\uu, & \quad \text{ in } \Omega\times(2\delta,+\infty),\\
				\nabla \cdot \uu = 0 & \quad \text{ in } \Omega\times (2\delta,+\infty),\\
				\uu = \mathbf{0} &  \quad \text{ in } \partial \Omega \times (2\delta, +\infty).
			\end{cases}
		\end{equation}
Using the estimates \eqref{reg1}, \eqref{reg1a} and \eqref{reg3} we can apply the regularity theory for the Stokes problem \cite[Theorem 4]{Ab09} with a bootstrap argument like in \cite[Theorem 4.3]{GMT18}, to show that the source term in \eqref{eq:dnlp2} belongs to $L^\infty(2\delta,+\infty; \mathbf{L}^2(\Omega))$ and thus
\begin{equation}
\textcolor{black}{\uu\in L^\infty(2\delta,+\infty;\mathbf{W}_\sigma)}\label{reg55}
\end{equation}
is uniformly bounded. It also holds that $\pi \in L^\infty(2\delta,+\infty;H^1(\Omega))$. As a consequence, from \eqref{reg1a}, \eqref{reg3} and \eqref{reg55}, we further infer that
\begin{equation*}
\partial_t \phi + \uu \cdot \nabla \phi\in L^\infty(2\delta,+\infty;\textcolor{black}{H^2(\Omega)}),\quad
\partial_t \rho + \uu \cdot \nabla \rho \in L^\infty(2\delta, +\infty; \textcolor{black}{H^1(\Omega)}),
\end{equation*}
which yield
$$
\Delta \mu\in L^\infty(2\delta,+\infty;H^2(\Omega)),\quad \Delta \psi \in L^\infty(2\delta, +\infty;H^1(\Omega)).
$$
Using the elliptic estimate like in Subsection \ref{sec:exestr}, we find that
$$
\mu\in L^\infty(2\delta,+\infty;H^4(\Omega)),\quad \psi\in L^\infty(2\delta,+\infty;H^2(\Omega))
$$
are uniformly bounded.

Thanks to \eqref{regg2sw}, $\widehat{S}^\prime_\rho$ can be considered as globally Lipschitz since $\rho$ only takes its values in a compact subinterval of $[0,1]$, everywhere in $\Omega \times [2\delta,+\infty)$. From \textbf{(H3)$'$}, we conclude that actually $\widehat{S}^\prime_\rho(\rho) \in L^\infty(2\delta,+\infty;H^2(\Omega))$ and thus elliptic regularity theory applied to \eqref{eq:nnlp1} implies that $\rho \in L^\infty(2\delta,+\infty;H^4(\Omega))$. Then, in a similar fashion, we observe that $\phi$ is solution to the Poisson problem
		\begin{equation} \label{eq:phiprob1}
			\begin{cases}
				-\Delta \phi = f & \quad \text{ in } \Omega, \\
				\partial_{\mathbf{n}}\phi = 0 & \quad \text{ on } \partial\Omega, \\
			\end{cases}
		\end{equation}
satisfying $\overline{\phi} = \overline{\phi_0}$, where $f$ solves the linear boundary value problem
		\begin{equation} \label{eq:phiprob2}
			\begin{cases}
				-\alpha \Delta f + f = \mu - S'_\phi(\phi) - \nabla \cdot(\rho\nabla\phi) & \quad \text{ in } \Omega, \\
				\partial_{\mathbf{n}}f= 0 & \quad \text{ on } \partial\Omega.
			\end{cases}
		\end{equation}
		Notice that, by the Lax--Milgram theorem, \eqref{eq:phiprob2} has a unique weak solution in $H^1(\Omega)$, which coincides with $-\Delta \phi$ almost everywhere in $\Omega\times [2\delta,+\infty)$. Thanks to \textbf{(H2)$'$} and \eqref{reg1a}, we easily infer that $S'_\phi(\phi)\in $ $L^\infty(2\delta, +\infty;H^2(\Omega))$. For the coupling term, it follows that
\begin{align*}
\|\nabla \cdot(\rho\nabla\phi)\|_{H^2(\Omega)}
&\leq C\big(\|\rho\|_{L^\infty(\Omega)}\|\Delta \phi\|_{H^2(\Omega)}+\|\rho\|_{H^2(\Omega)}\|\Delta \phi\|_{L^\infty(\Omega)}\big)\\
&\quad  + C\big(\|\nabla \rho\|_{\mathbf{L}^\infty(\Omega)}\| \nabla \phi\|_{\mathbf{H}^2(\Omega)}+\|\nabla \rho\|_{\mathbf{H}^2(\Omega)}\| \nabla \phi\|_{\mathbf{L}^\infty(\Omega)}\big)\\
&\leq C\big(\|\rho\|_{H^2(\Omega)}\|\phi\|_{H^4(\Omega)}+ \|\rho\|_{H^3(\Omega)}\|\phi\|_{H^3(\Omega)}\big),
\end{align*}
which implies $\nabla \cdot(\rho\nabla\phi)\in L^\infty(2\delta,+\infty;H^2(\Omega))$. As a consequence, $f\in L^\infty(2\delta,+\infty;H^4(\Omega))$. Applying the elliptic estimate for problem \eqref{eq:phiprob1}, we find that $\phi\in L^\infty(2\delta,+\infty;H^6(\Omega))$ is uniformly bounded. Hence, the estimate \eqref{regg3} follows.

The proof of Theorem \ref{th:regularity3} is complete.  \hfill $\square$

\section*{Acknowledgements} M.~ Grasselli is a member of Gruppo Nazionale per l'Analisi Ma\-te\-ma\-ti\-ca, la Probabilit\`{a} e le loro Applicazioni (GNAMPA), Istituto Nazionale di Alta Matematica (INdAM). H.~Wu was partially supported by NNSFC grant No. 12071084 and the Shanghai Center for Mathematical Sciences at Fudan University.

\printbibliography[title=Bibliography]
\end{document}